\documentclass[a4paper, 10pt, parskip=half]{scrartcl}

\usepackage[utf8]{inputenc}
\usepackage[T1]{fontenc}
\usepackage{lmodern}
\usepackage{csquotes}
\usepackage{amsmath}
\usepackage{amssymb}
\usepackage{amsthm}
\usepackage{amsfonts}
\usepackage{dsfont}
\usepackage{mathtools}
\usepackage{marginnote}
\usepackage[dvipsnames]{xcolor}
\usepackage{hyperref}
\hypersetup{%
  colorlinks=true,
  linkcolor=black,
  citecolor=black,
  urlcolor=blue,
  pdftitle={Global existence of classical solutions and numerical simulations of a cancer invasion model},
  pdfauthor={},
  pdfkeywords={},
  bookmarksopen=true,
}

\usepackage{authblk}
\usepackage{float}
\usepackage{subcaption}

\usepackage{xcolor}

\newcommand{\R}{\mathbb{R}}

\newcommand{\N}{\mathbb{N}}

\newcommand{\ur}[1]{\mathrm{#1}}
\newcommand{\ure}{\ur{e}}

\ifdefined\labelenumi%
  \renewcommand{\labelenumi}{(\roman{enumi})}
  
\fi

\newcommand{\eps}{\varepsilon}

\newcommand{\defs}{\coloneqq}
\newcommand{\sfed}{\eqqcolon}

\newcommand{\sea}{\searrow}

\newcommand{\ol}{\overline}

\newcommand{\ds}{\,\mathrm{d}s}
\newcommand{\dt}{\,\mathrm{d}t}

\newcommand{\dtau}{\,\mathrm{d}\tau}

\newcommand{\ddt}{\frac{\mathrm{d}}{\mathrm{d}t}}

\newcommand{\embed}{\hookrightarrow}

\newcommand{\hp}{\hphantom}
\newcommand{\pe}{\mathrel{\hp{=}}}

\newcommand{\tmax}{T_{\max}}

\newcommand{\intom}{\int_\Omega}

\newcommand{\intntom}{\int_0^T \int_\Omega}

\newcommand{\Ombar}{\ol \Omega}

\newcommand{\loc}{\mathrm{loc}}

\newcommand{\leb}[2][\Omega]{\ensuremath{L^{#2}(#1)}}

\newcommand{\sob}[3][\Omega]{\ensuremath{W^{#2, #3}(#1)}}

\newcommand{\con}[2][\Ombar]{\ensuremath{C^{#2}(#1)}}

\newcommand{\f}[2]{\frac{#1}{#2}}
\newcommand{\norm}[2][]{\|#2\|_{#1}}
\newcommand{\nnorm}[2]{\|#2\|_{#1}}
\newcommand{\set}[1]{\left\{#1\right\}}
\newcommand{\Om}{\Omega}
\newcommand{\Lom}[1]{L^{#1}(\Omega)}
\newcommand{\Tmax}{T_{\max}}
\newcommand{\upto}{\nearrow}
\newcommand{\matr}[1]{\begin{pmatrix}#1\end{pmatrix}}
\newcommand{\kl}[1]{\left(#1\right)}
\newcommand{\io}{\int_\Omega}

\usepackage{newunicodechar}
\newunicodechar{∇}{\nabla}
\newunicodechar{Δ}{\Delta}
\newunicodechar{α}{\alpha}
\newunicodechar{γ}{\gamma}
\newunicodechar{μ}{\mu}
\newunicodechar{ℕ}{\mathbb{N}}
\newunicodechar{ℝ}{\mathbb{R}}
\newunicodechar{∂}{\partial}
\newunicodechar{ν}{\nu}
\newunicodechar{Φ}{\Phi}
\newunicodechar{τ}{\tau}
\newunicodechar{ω}{\omega}
\newunicodechar{∞}{\infty}
\newunicodechar{ϕ}{\phi}
\newunicodechar{χ}{\chi}

\newcommand{\tops}{\texorpdfstring}

\makeatletter
\renewenvironment{proof}[1][\proofname]{\par
  \pushQED{\qed}%
  \normalfont \topsep0\p@\relax
  \trivlist
  \item[\hskip\labelsep\scshape
  #1\@addpunct{.}]\ignorespaces
}{%
  \popQED\endtrivlist\@endpefalse
}
\makeatother

\newtheorem{base}{Base}[section]
\numberwithin{equation}{section}

\newtheorem{theorem}[base]{Theorem} \newtheorem*{theorem*}{Theorem}
\newtheorem{lemma}[base]{Lemma} \newtheorem*{lemma*}{Lemma}
 \newtheorem*{prop*}{Proposition}
 \newtheorem*{cor*}{Corollary}
\newtheorem{algo}[base]{Algorithm} \newtheorem*{algo*}{Algorithm}

\theoremstyle{definition}
\newtheorem{remark}[base]{Remark} \newtheorem*{remark*}{Remark}
 \newtheorem*{definition*}{Definition}
 \newtheorem*{example*}{Example}
 \newtheorem*{cond*}{Condition}

\allowdisplaybreaks

\begin{document}
\setkomafont{title}{\normalfont\Large}
\title{Global existence of classical solutions and numerical simulations of a cancer invasion model}

\renewcommand\Affilfont{\itshape\footnotesize}
\author[1]{M.\ Fuest\footnote{e-mail:\ fuest@ifam.uni-hannover.de, corresponding author}}
\author[2]{Sh.\ Heydari\footnote{e-mail:\ heydari@karlin.mff.cuni.cz}}
\author[2]{P.\ Knobloch\footnote{e-mail:\ knobloch@karlin.mff.cuni.cz}}
\author[1]{J.\ Lankeit\footnote{e-mail:\ lankeit@ifam.uni-hannover.de}}
\author[1]{T.\ Wick\footnote{e-mail:\ thomas.wick@ifam.uni-hannover.de}}

\affil[1]{Leibniz University Hannover,
	  Institute of Applied Mathematics,
	  Welfengarten 1, 30167 Hannover, Germany}

\affil[2]{Charles University, 
          Faculty of Mathematics and Physics,
          Sokolovska 83, 18675 Praha 8, Czech Republic}

\date{}

\maketitle

\KOMAoptions{abstract=true}
\begin{abstract}
\noindent
In this paper, we study a cancer invasion model both theoretically and numerically. 
The model is a nonstationary, nonlinear system of three coupled partial differential equations modeling
the motion of cancer cells, degradation of the extracellular matrix,
and certain enzymes. We first establish existence of global classical 
solutions in both two- and three-dimensional bounded domains,
despite the lack of diffusion of the matrix-degrading enzymes and corresponding regularizing effects in the analytical treatment.
Next, we give a weak formulation and apply finite differences in time 
and a Galerkin finite element scheme for spatial discretization. The overall
algorithm is based on a fixed-point iteration scheme. In order to substantiate 
our theory and numerical framework, several numerical simulations are carried out 
in two and three spatial dimensions.
  \\[0.5pt]
 \textbf{AMS Classification (2020):} {
35A01, 
35K57, 
35Q92, 
65M22, 
65M60, 
92C17 
}
\end{abstract}

\section{Introduction}
\paragraph{The model.} One of the defining characteristics of a malignant tumour is its capability to invade adjacent tissues \cite{hallmarks_of_cancer}; accordingly the mathematical literature directed at understanding
underlying mechanisms is vast (see e.g.\ the surveys \cite{survey_lowengrub,sfakianakis_chaplain}).

In this paper we focus on the following variant of a cancer invasion model developed by Perumpanani et al.\ \cite{perumpanani99} for the malignant invasion of tumours and investigate
\begin{align}\label{eq:system}
  \begin{cases}
    u_t = \frac{1}{\alpha} \Delta u - \chi \nabla \cdot (u \nabla c) + \mu u (1 - u) & \text{in $\Omega \times (0, \infty)$}, \\
    c_t = - pc                                                                       & \text{in $\Omega \times (0, \infty)$}, \\
    p_t = \frac{1}{\eps} (uc - p)                                                    & \text{in $\Omega \times (0, \infty)$}, \\
    \frac{1}{\alpha} \partial_\nu u = \chi u \partial_\nu c                          & \text{on $\partial \Omega \times (0, \infty)$}, \\
    (u, c, p)(\cdot, 0) = (u_0, c_0, p_0)                                            & \text{in $\Omega$}.
  \end{cases}
\end{align} 
We aim for a rigorous existence proof for global solutions and 
the development of a reliable numerical scheme with an implementation 
in a modern open-source finite element library.

In \eqref{eq:system}, the motion of cancer cells (density denoted by $u$) mainly takes place by means of haptotaxis, i.e.\ directed motion toward higher concentrations of extracellular matrix (density $c$), of strength $χ\ge 0$. Motivated by experiments of Aznavoorian et al.\ \cite{aznavoorian}, who reported only ``a minor chemokinetic component'' of the cell motion, the original model of \cite{perumpanani99} does not include a term for random (chemokinetic) cell motility at all. Acknowledging that ``minor'' does not mean ``none at all'', we deviate from \cite{perumpanani99} in this aspect and incorporate this motility term in \eqref{eq:system} ($α\in (0,\infty)$, with the formal limit $α\to∞$ corresponding to the model of \cite{perumpanani99}). Additional growth of the population of tumour cells is described by a logistic term (with $μ$ being a positive parameter). The extracellular matrix is degraded upon contact with certain enzymes (proteases, concentration $p$), which, in turn, are produced where cancer cells and matrix meet and decay over time. The reaction speed of these protein dynamics can be adjusted via the parameter $\eps>0$. 
As many proteases remain bound to the cellular membrane -- or are only activated when on the cell surface (cf.\ the model derivation in \cite{perumpanani99}), no diffusion for $p$ is incorporated in the model.
This last point is in contrast to the otherwise similar popular models in the tradition of \cite{anderson_etal_2000,perumpanani_byrne} or \cite{chaplain_lolas}, the latter of which additionally included a chemotactic component of the motion of cancer cells.

\paragraph{Global solvability.} In order to construct global classical solutions of \eqref{eq:system},
it is necessary to control the haptotaxis term $-\chi \nabla \cdot (u \nabla c)$ in the first equation
and thus in particular to gain information on the spatial derivative of the second solution component.
For relatives of \eqref{eq:system} including a diffusion term $\Delta p$ in the third equation,
this has already been achieved in \cite{tao2007global} and \cite{LitcanuMorales-RodrigoAsymptoticBehaviorGlobal2010} by applying parabolic regularity theory to the equation for $p$,
first yielding estimates for the spatial derivative of $p$ and then also on $c$;
the results of \cite{LitcanuMorales-RodrigoAsymptoticBehaviorGlobal2010} even cover the long-term asymptotics of solutions.
Moreover, the presence of diffusion for the produced quantities has also been made use of to obtain global existence results for different cancer invasion models, see for instance \cite{ref13}.

However, the absence of any spatial regularization in both the second and third equation makes the corresponding analysis much more challenging.
Up to now, global classical solutions have only been constructed for a rather limited set of initial data:
Already in \cite{perumpanani99}, where \eqref{eq:system} has been proposed for $\alpha=\infty$, 
it has been shown that the model formally obtained by taking the limit $\eps \sea 0$ admits a family of travelling wave solutions.
Corresponding results for positive $\eps$ have then been achieved in \cite{marchant2001travelling}.
Moreover, if $\eps = 0$, travelling wave solutions may contain shocks \cite{perumpanani2000traveling}
and solutions of related systems without a logistic source may even blow up in finite time \cite{RascleZitiFiniteTimeBlowup1995}.
In general, the destabilizing effect of taxis terms such as $-\chi \nabla \cdot (u \nabla c)$ may not only make it challenging but even impossible to obtain global existence results for certain problems.
We refer to the survey \cite{LankeitWinklerFacingLowRegularity2019} for further discussion regarding the consequences of low regularity in chemotaxis systems.

Despite these challenges, in the first part of the present paper we are able to give an affirmative answer to the question whether \eqref{eq:system} also possesses global classical solutions for widely arbitrary initial data in the two and three-dimensional setting. Our analytical main result is the following
\begin{theorem}\label{th:global_ex}
  Suppose that $\alpha, \chi, \mu, \eps$ are positive constants, that
  \begin{align*}
    \Omega \text{ is a smooth bounded domain in $\R^n$, $n \in \{1, 2, 3\}$,}
  \end{align*}
  and that $u_0, c_0, p_0 \in \bigcup_{\gamma \in (0, 1)} \con{2+\gamma}$ are nonnegative and such that $\frac{1}{\alpha} \partial_\nu u_0 = \chi u_0 \partial_\nu c_0$ on $\partial \Omega$.
  Then there exists a unique global classical solution $(u, c, p)$ of \eqref{eq:system} with regularity
   \begin{align*}
     (u, c, p) \in \big( C^{2, 1}(\Ombar \times (0, \infty)) \cap C^1(\Ombar \times [0, \infty)) \big)^3,
   \end{align*}
   which, moreover, is nonnegative.
\end{theorem}

\paragraph{Numerical modeling.} In the second part of our paper we then analyze the behavior of 
these solutions numerically with an implementation in the modern open-source 
finite element library deal.II \cite{dealii2019design,dealII91}.
Related numerical studies in various software libraries using 
different numerical schemes are briefly described in the following.
The traditional method of lines has been widely used for simulations of 
the cancer invasion process~\cite{gerisch2008mathematical, domschke2014mathematical}. In addition, finite difference 
methods~\cite{kolev2011numerical}
have been considered and in \cite{chapwanya2014positivity, khalsaraei2016positivity}, 
the authors 
proposed a nonstandard finite difference method which satisfies 
the positivity-preservation of the solution, that is an important 
property in the stability of the model.
Moreover, the finite volume method \cite{chertock2008second}, spectral element 
methods \cite{valenciano2003computing}, 
algebraically stabilized finite element method \cite{strehl2013positivity}, 
the discontinuous Galerkin method \cite{epshteyn2009discontinuous}, 
combinations of level-set/adaptive finite elements \cite{zheng2005nonlinear, amoddeo2015adaptive}, 
and a hybrid finite volume/finite element method \cite{chaplain_lolas_big_model, anderson2005hybrid} 
have also been proposed in the literature for some cancer 
invasion models and chemotaxis. Finally, we mention that in
\cite{surulescu_winkler_2021} the authors illustrate their theoretical results for a related 
cancer model employing discontinuous Galerkin finite elements implemented as well
in deal.II.

The main objective in the numerical part is the design
of reliable algorithms for \eqref{eq:system} and their corresponding implementation
in deal.II. First, 
we discretize in time using a $\theta$-method, which allows for implicit $A$-stable 
time discretizations. Then, a Galerkin finite element scheme is employed for 
spatial discretization. The nonlinear discrete system of equations is decoupled by 
designing a fixed-point algorithm. This algorithm is newly designed and then 
implemented and debugged in deal.II. 

These developments then allow to link our theoretical part and the numerical 
sections in order to carry out various numerical simulations to complement
Theorem~\ref{th:global_ex}. Specifically, several parameter variations 
of the proliferation coefficient $\mu$ and the haptotactic coefficient $\chi$ will
be studied in two- and three spatial dimensions. These studies are non-trivial 
due to the nonlinearities and the high sensitivity of \eqref{eq:system} 
with respect to such parameter variations.

\paragraph{Plan of the paper.}
The outline of this paper is as follows. In Section~\ref{sec_existence_classical},
we study the global existence of classical solutions. Next, in 
Section~\ref{sec_discretization}, we introduce the discretization in time and space 
using finite differences in time and a Galerkin finite element scheme in space.
We also describe the solution algorithm. In Section~\ref{sec_tests}, we carry 
out several numerical simulations demonstrating the properties of our model and the corresponding 
theoretical results. Therein, we specifically study  parameter variations.
Finally, our work is summarized in Section~\ref{sec_conclusions}.

\paragraph{Notation.}
Let $\Omega \subset \R^n$, $n \in \N$ be a bounded domain.
By $L^p(\Omega)$ and $W^{1, p}(\Omega)$, we denote the usual Lebesgue and Sobolev spaces, respectively,
and we abbreviate $H^1(\Omega) \defs \sob12$.
Furthermore,
$\langle\cdot , \cdot \rangle$ denotes the duality product between $(H^1)^*$ and $H^1$.

For $m \in \N_0$ and $\gamma \in (0, 1)$, we denote by $\con{m+\gamma}$ the space of functions $\varphi \in \con m$ with finite norm
\begin{align*}
  \|\varphi\|_{\con{m+\gamma}} \defs \|\varphi\|_{\con m} + \sup_{x, y \in \Ombar, x \neq y} \frac{|\varphi(x) - \varphi(y)|}{|x-y|}.
\end{align*}
Moreover, for $m_1, m_2 \in \N_0$, $\gamma_1, \gamma_2 \in [0, 1)$ and $T > 0$,
we denote by $\con[{\Ombar \times [0, T]}]{m_1+\gamma_1, m_2+\gamma_2}$ the space of all functions $\varphi$
whose derivatives $D_x^\alpha D_t^\beta \varphi$, $|\alpha| \le m_1$, $0 \le \beta \le m_2$, (exist and) are continuous, and which have finite norm
\allowdisplaybreaks[0]
\begin{align*}
  &      \|\varphi\|_{\con[{\Ombar \times [0, T]}]{m_1+\gamma_1, m_2+\gamma_2}}
   \defs \sum_{\substack{|\alpha| \le m_1, \\0 \le \beta \le m_2}} \|D_x^\alpha D_t^\beta \varphi\|_{\con[{\Ombar \times [0, T]}]{0}} \\
  &+     \sum_{\substack{|\alpha| = m_1, \\0 \le \beta \le m_2}} \sup_{\substack{x, y \in \Ombar, x \neq y, \\ t \in [0, T]}} \frac{|D_x^\alpha D_t^\beta \varphi(x, t) - D_x^\alpha D_t^\beta \varphi(y, t)|}{|x-y|^{\gamma_1}} 
   +     \sum_{\substack{|\alpha| \le m_1, \\ \beta = m_2}} \sup_{\substack{x \in \Ombar, \\ s, t \in [0, T], s \neq t}} \frac{|D_x^\alpha D_t^\beta \varphi(x, t) - D_x^\alpha D_t^\beta \varphi(x, s)|}{|t-s|^{\gamma_2}}.
\end{align*}
\allowdisplaybreaks
Notationally, we do not distinguish between spaces of scalar- and vector-valued functions.

\section{Global existence of classical solutions}\label{sec_existence_classical}
As a first step in the proof of Theorem~\ref{th:global_ex}, we apply two transformations in Subsection~\ref{sec:transformations};
the first one allows us to get rid of some parameters in \eqref{eq:system}, the second one changes the first equation to a more convenient form.
We then employ a fixed point argument to obtain a local existence result for the transformed system in Lemma~\ref{lm:local_ex}.

The proof that these solutions are global in time consists of two key parts,
both relying on the fact that the second and third equation in \eqref{eq:system} at least regularize in time (which allows us to prove Lemma~\ref{lm:p_linfty_le_w_linfty} and Lemma~\ref{lm:c_w1q_rel}).
First, in order to prove boundedness in $L^\infty$,
the comparison principle allows us to conclude boundedness in small time intervals (cf.\ Lemma~\ref{lm:w_linfty_t0}).
We then iteratively apply this bound to obtain the result also for larger times (cf.\ Lemma~\ref{lm:w_linfty_tmax}).
As to bounds for the spatial derivatives, we secondly apply a testing procedure to derive estimates valid on small time intervals (cf.\ Lemma~\ref{lm:c_w14_t1}),
which then is again complemented by an iteration procedure (cf.\ Lemma~\ref{lm:c_w14_tmax}).
Finally, we are able to make use of parabolic regularity theory (inter alia in the form of maximal Sobolev regularity) to conclude in Lemma~\ref{lm:global_ex} that the solutions exist globally.

\subsection{Two transformations}\label{sec:transformations}
We first note that with regards to Theorem~\ref{th:global_ex} we may without loss of generality assume $\chi = 1$ and $\eps = 1$.
Indeed, suppose that Theorem~\ref{th:global_ex} holds for this special case.
Then, assuming the conditions of Theorem~\ref{th:global_ex} to hold, we set
\begin{align*}
  \tilde \alpha \defs \frac{\alpha \chi}{\eps}, \quad
  \tilde \chi \defs 1, \quad
  \tilde \mu \defs \eps \mu, \quad
  \tilde \eps \defs 1
\end{align*}
and further
\begin{align*}
  \tilde u_0(\tilde x) = u_0(\sqrt \chi \tilde x), \quad
  \tilde c_0(\tilde x) = \eps c_0(\sqrt \chi \tilde x), \quad 
  \tilde p_0(\tilde x) = \eps p_0(\sqrt \chi \tilde x)
\end{align*}
for $\tilde x \in \tilde \Omega \defs \frac{1}{\sqrt \chi} \Omega$.
By Theorem~\ref{th:global_ex}, there then exists a global classical solution of \eqref{eq:system} (with all parameters and initial data replaced by their pendants with tildes) $(\tilde u, \tilde c, \tilde p)$.
Then
\begin{align}\label{eq:trans_1}
  (u, c, p)(x, t) \defs \big( \tilde u(\tfrac{x}{\sqrt \chi}, \tfrac{t}{\eps}),\, \tfrac1\eps \tilde c(\tfrac{x}{\sqrt \chi}, \tfrac{t}{\eps}),\, \tfrac1\eps \tilde p(\tfrac{x}{\sqrt \chi}, \tfrac{t}{\eps}) \big), \quad (x, t) \in (\Ombar \times [0, \infty)),
\end{align}
fulfills
\begin{alignat*}{2}
  \begin{cases}
     u_t = \frac{\chi}{\tilde \alpha \eps} \Delta  u - \frac{\eps \chi}{\eps} \nabla \cdot ( u \nabla  c) + \frac{\tilde \mu}{\eps} u (1 - u) & \text{in $\Omega \times (0, \infty)$}, \\[0.2em]
     c_t = - \frac{\eps^2}{\eps^2}  c  p                                                                                                      & \text{in $\Omega \times (0, \infty)$}, \\[0.2em]
     p_t = \frac1{\eps^2} (\eps  u  c - \eps p)                                                                                               & \text{in $\Omega \times (0, \infty)$}, \\[0.2em]
     \partial_\nu  u = \frac{\tilde \alpha \eps \sqrt{\chi}}{\sqrt{\chi}}  u \partial_\nu  c                                                  & \text{on $\partial \Omega \times (0, \infty)$}, \\[0.2em]
     ( u,  c,  p)(x, 0) = (\tilde u_0, \frac1\eps \tilde c_0, \frac1\eps \tilde p_0)(\frac{x}{\sqrt\chi})                                     & \text{for $x \in \Omega$}
  \end{cases}
\end{alignat*}
and thus \eqref{eq:system}.
Moreover, following precedents from, e.g.\ \cite[p.19]{corrias_perthame_zaag}, we set 
\begin{equation}
 w(x,t) \coloneqq u(x,t) \ure^{-αc(x,t)},\qquad x\in \Ombar,\; t\ge 0.
\end{equation}
Then $∇w = \ure^{-αc}∇u - α \ure^{-αc}u∇c$, so that
\begin{align}\label{eq:transformation:12_order_terms}
  \tfrac{1}{\alpha} \Delta w + \nabla c \cdot \nabla w = \ure^{-\alpha c} ∇\cdot(\tfrac{1}{α}\ure^{αc}∇w)=\ure^{-\alpha c}[\tfrac1{α}Δu - ∇\cdot(u∇c)]
\end{align}
and \eqref{eq:system} (with $\chi=\eps=1$) is equivalent to 
\begin{align}\label{eq:system-w}
  \begin{cases}
    w_t = \frac{1}{\alpha} \Delta w + \nabla c \cdot \nabla w + αpcw + μw-μ\ure^{αc}w^2 & \text{in $\Omega \times (0, \infty)$},\\
    c_t = - pc                                                                          & \text{in $\Omega \times (0, \infty)$}, \\
    p_t = w\ure^{αc}c -p                                                                & \text{in $\Omega \times (0, \infty)$}, \\
    \partial_\nu w = 0                                                                  & \text{on $\partial \Omega \times (0, \infty)$}, \\
    (w, c, p)(\cdot, 0) = (w_0 , c_0, p_0)                                              & \text{in $\Omega$}
  \end{cases}
\end{align}
for $w_0 \defs u_0\ure^{-αc_0}$.
Expanding $\nabla \cdot (u \nabla c)$ to $\nabla u \cdot \nabla c + u \Delta c$
shows that this transformation allows us to get rid of a term involving $\Delta c$ in the first equation at the price of adding several zeroth order terms.
In particular as the second equation does not regularize in space, \eqref{eq:system-w} turns out to be a more convenient form for the following analysis.

\subsection{Local existence}
In this subsection, we construct maximal classical solutions of \eqref{eq:system-w} in $\Ombar \times [0, \tmax)$ for some $\tmax \in [0, \infty)$ by means of a fixed point argument.
Moreover, we provide a criterion for when these solutions are global in time (that is, when $\tmax = \infty$ holds), which then will finally be seen to hold true in Lemma~\ref{lm:global_ex}.

As a preparation, we first collect results on (Hölder) continuous dependency of solutions to ODEs on the data.
\begin{lemma}\label{lm:ODEgeneral}
 Let $\Omega \subset \R^n$, $n \in \N$, be a bounded domain, $T>0$, $d\in ℕ$, $γ_1,γ_2\in[0,1)$, $v_0\in C^{γ_1}(\Ombar)$ and assume that $f\colon \Ombar \times [0,T]\times ℝ^d\to ℝ^d$ is $(γ_1,γ_2)$-Hölder continuous with respect to its first two arguments and locally Lipschitz continuous w.r.t.\ the third variable, in the sense that for every compact $K\subset ℝ^d$ there is $L>0$ such that 
 \begin{align*}
  \sup_{t\in[0,T],v\in K} \norm[C^{γ_1}(\Ombar)]{f(\cdot,t,v)} &\le L,\\
  \sup_{x\in\Ombar,v\in K} \nnorm{C^{γ_2}([0,T])}{f(x,\cdot,v)}&\le L,\\
  \sup_{(x,t)\in \Ombar\times[0,T]} |f(x,t,v)-f(x,t,w)| &\le L|v-w| \quad \text{for all } v,w\in K.
 \end{align*}
 Then for any compact $K\subset ℝ^d$ there is $C>0$ such that whenever $v \colon \Ombar \times [0, T] \to \R^d$ is such that $v(x,\cdot)\in C^0([0,T])\cap C^1((0,T))$ for all $x\in \Ombar$, $v$ solves
 \begin{equation}
  v_t(x,t) = f(x,t,v(x,t)) \quad \text{ for all } x\in \Ombar, t\in(0,T);\qquad v(x,0)=v_0(x)\quad \text{for all } x\in \Ombar
 \end{equation}
 and satisfies $v(\Ombar\times[0,T])\subset K$, then $v,v_t\in C^{γ_1,γ_2}(\Ombar\times[0,T])$ and 
 \[
  \nnorm{C^{γ_1,1+γ_2}(\Ombar\times[0,T])}{v}\le C.
 \]
%
\end{lemma}
\begin{proof} 
First, we let $K$ be a compact superset of $v(\Ombar\times[0,T])$ and let $L$ be as in the assumptions on $f$. We introduce $ω_1\in C^0([0,∞))$ such that $|f(x,t,w)-f(y,t,w)|\le ω_1(|x-y|)$ for all $x\in \Ombar$, $y\in \Ombar$, $t\in[0,T]$ and $w\in K$, 
$|v_0(x)-v_0(y)|\le ω_1(|x-y|)$ for all $x\in \Ombar$, $y\in \Ombar$ 
and such that $ω_1(0)=0$ and $\sup_{r>0} r^{-γ_1}ω_1(r)<∞$. 
We then fix $x\in \Ombar$, $y\in\Ombar\setminus\set{x}$ and let $\tilde v(t)=(v(x,t)-v(y,t))\cdot\frac1{ω_1(|x-y|)}$. Then 
\allowdisplaybreaks[0]
\begin{align*}
  \tilde v_t(t) &= \f1{ω_1(|x-y|)}(f(x,t,v(x,t))-f(y,t,v(x,t)) + \f1{ω_1(|x-y|)} (f(y,t,v(x,t))-f(y,t,v(y,t))) \\
  &\le 1 + \f L{ω_1(|x-y|)} |v(x,t)-v(y,t)| \qquad \text{for all $t \in (0, T)$},
 \end{align*}
\allowdisplaybreaks
so that $\tilde v_t \le 1 + L|\tilde v|$ and, analogously, $\tilde v_t \ge - 1 - L|\tilde v|$, so that boundedness of $|\tilde v|$  results from Grönwall's inequality
and hence $v \in C^{\gamma_1, 0}(\Ombar \times [0, T])$, and $\sup_{t\in[0,T]}\norm[C^{γ_1}(\Ombar)]{v(\cdot,t)}$ is bounded due to the choice of $\omega_1$.

For $τ>0$, we treat $\bar v(x,t)=(v(x,t)-v(x,t+τ))/ω_2(τ)$ with some $ω_2$ such that $|f(x,t,v)-f(x,t+τ,v)|\le ω_2(τ)$ in the same way. 
This ensures the claimed regularity of $v$, whereupon that of $v_t$ follows from $v_t(x,t)=f(x,t,v(x,t))$, $(x,t)\in \Ombar\times(0,T)$ and continuity of the right-hand side up to $t=0$ and $t=T$.
\end{proof}

\begin{lemma}\label{lm:ODEgeneral-higher}
 In addition to the assumptions of Lemma~\ref{lm:ODEgeneral}, let $m \in \N$ and $v_0\in C^{m+γ_1}(\Ombar)$. If all derivatives of $f$ w.r.t.\ $x$ and $v$ up to order $m$ satisfy the conditions Lemma~\ref{lm:ODEgeneral} poses on $f$, then any solution $v$ as in Lemma~\ref{lm:ODEgeneral} belongs to $C^{m+γ_1,1+γ_2}(\Ombar\times[0,T])$.
\end{lemma}
\begin{proof}
 For $i\in\set{1,\ldots,n}$, $\tilde v=∂_{x_i} v$ satisfies 
 \[
  \tilde v_t = f_{x_i}(x,t,v(x,t)) + f_v(x,t,v(x,t))\tilde v \quad \text{in } \Om\times(0,T), \qquad \tilde v(\cdot,0)=(v_0)_{x_i} \text{ in } \Om
 \]
 and Lemma~\ref{lm:ODEgeneral} can be applied to $\tilde v$. An inductive argument takes care of higher derivatives. 
\end{proof}

By applying these results to the ODEs appearing in \eqref{eq:system-w}, we obtain
\begin{lemma}\label{lm:ODE}
 Let $\Omega\subset ℝ^n$, $n \in \N$, be a bounded domain, $α\ge 0$ and $T\in(0,\infty)$. Let $0\le w\in C^0(\Ombar\times[0,T])$ and $c_0,p_0\in C^0(\Ombar;[0,\infty))$. \\
 a) Then 
 \begin{align}\label{ode-cp}
  \begin{cases}
   c_t = -pc\\
   p_t = w\ure^{αc}c - p
  \end{cases}
 \end{align}
 has a unique solution $(c,p)\in (C^{0,1}(\Ombar\times[0,T]))^2$.\\
 b) For any $M>0$ there is $C=C(M)>0$ such that whenever $\tilde T\in(0,T]$, $w\in C^{1,0}(\Ombar\times[0,\tilde T])$, $c_0, p_0\in C^1(\Ombar)$ with 
 \[
  \sup_{t\in[0,\tilde T]}\norm[C^1(\Ombar)]{w(\cdot,t)} \le M, \quad \norm[C^1(\Ombar)]{c_0}\le M,\quad \norm[C^1(\Ombar)]{p_0}\le M, 
 \]
 then 
 \[
  \nnorm{C^1(\Ombar\times[0,\tilde T])}{(c,p)}\le C.
 \]
 c) If, for some $k\in ℕ_0$ and $γ_1,γ_2\in[0,1)$, $w\in C^{k+γ_1,γ_2}(\Ombar\times[0,T])$ and $c_0, p_0\in C^{k+γ_1}(\Ombar)$, then $c,p\in C^{k+γ_1,1+γ_2}(\Ombar\times[0,T])$.
\end{lemma}
\begin{proof}
a) Given $w\in C^0(\Ombar\times[0,T])$, for every $x\in\Ombar$ the existence and uniqueness of a solution $(c,p)(x,\cdot)\in C^0([0,\Tmax(x)])\cap C^1((0,\Tmax(x)))$ of \eqref{ode-cp}, with some $\Tmax(x)\in(0,T]$ such that $\limsup_{t\upto\Tmax(x)}(|c(x,t)|+|p(x,t)|)=\infty$ or $\Tmax(x)=T$, follows from Picard--Lindelöf's theorem. By an ODE comparison argument, nonnegativity of $c$ follows from that of $c_0$, and nonnegativity of $p$ from that of $p_0$, $w$ and $c$. Therefore, $0\le c(x,t)\le c_0(x)$ for all $x\in \Ombar$ and $t\in(0,\Tmax(x))$ due to the sign of $c_t=-pc$. Consequently, 
\allowdisplaybreaks[0]
 \begin{align*}
  p(x,t) &= \ure^{-t} p_0(x) + \int_0^t \ure^{-(t-s)}w(x,s)\ure^{αc(x,s)}c(x,s)\ds\\
  &\le \max\set{p_0(x), c_0(x)\ure^{αc_0(x)}\max_{s\in[0,T]}w(x,s)} \qquad \text{for all } x\in\Ombar, t\in[0,\Tmax(x)), 
 \end{align*}
\allowdisplaybreaks
 which also shows that $\Tmax(x)=T$ for all $x\in\Ombar$. The remainder of part a) follows from an application of Lemma~\ref{lm:ODEgeneral} with $γ_1=γ_2=0$.\\
 b) We apply Lemma~\ref{lm:ODEgeneral} to $v=(∇c,∇p)$, the solution of  
 \[
  v_t = \matr{-p I_{n \times n}&-cI_{n \times n}\\(αw\ure^{αc}c+w\ure^{αc})I_{n \times n} &-I_{n \times n}}v + \matr{0\\\ure^{αc}c∇w},
  \qquad v(0) = \begin{pmatrix}\nabla c_0 \\ \nabla p_0 \end{pmatrix} \in \R^{2n},
 \]
 noting that sufficient regularity is given by the assumption on $w$ and part a).\\
 c) follows from Lemma~\ref{lm:ODEgeneral-higher}.
\end{proof}

Both as an ingredient to the proof of Lemma~\ref{lm:local_ex} below and also for its own interest, we note that classical solutions of \eqref{eq:system-w} are unique.
\begin{lemma}\label{lm:uniqueness}
 Let $\Omega \subset \R^n$, $n \in \N$, a bounded domain and let $\alpha, \mu > 0$ as well as $\chi = \eps = 1$.
 Suppose $(u_1,c_1,p_1)$ and $(u_2,c_2,p_2)$ are two solutions of \eqref{eq:system} with the same initial data $(u_0,c_0,p_0)\in C^0(\Ombar)\times C^1(\Ombar)\times C^1(\Ombar)$
 and assume that they belong to 
 \[
  \kl{C^{2,1}(\Ombar\times(0,T))\cap C^1(\Ombar\times[0,T])}^3
 \]
 for some $T>0$. Then $(u_1,c_1,p_1)=(u_2,c_2,p_2)$ in $\Ombar\times[0,T]$.
\end{lemma}
\begin{proof}
 We let $C>0$ be such that 
 \[
  \max\set{\norm[\Lom\infty]{u_i(\cdot,t)},\norm[\Lom\infty]{∇u_i(\cdot,t)},\norm[\Lom\infty]{c_i(\cdot,t)},\norm[\Lom\infty]{p_i(\cdot,t)},\norm[\Lom\infty]{∇c_i(\cdot,t)},\norm[\Lom\infty]{∇p_i(\cdot,t)}}\le C
 \]
 for all $t\in[0,T]$ and $i\in\set{1,2}$. Then computing 
 \begin{align*}
 \f12& \ddt \kl{\io (u_1-u_2)^2 + \io (c_1-c_2)^2 + \io (p_1-p_2)^2 + \io |∇(c_1-c_2)|^2 + \io |∇(p_1-p_2)|^2} \\
 &= -\f1{α} \io |∇(u_1-u_2)|^2 +\io ∇(u_1-u_2)[(u_1-u_2)∇c_1+u_2∇(c_1-c_2)] + μ\io (u_1-u_2)^2 \\
 &\quad - μ\io (u_1-u_2)^2(u_1+u_2) - \io (c_1-c_2)(p_1-p_2)c_1 - \io (c_1-c_2)^2p_2 -\io (p_1-p_2)^2\\
 &\quad +\io (p_1-p_2)u_1(c_1-c_2)+\io (p_1-p_2)(u_1-u_2)c_2 -\io ∇(c_1-c_2)∇(p_1-p_2)c_1\\
 &\quad  - \io ∇(c_1-c_2)(p_1-p_2)∇c_1 - \io p_2|∇(c_1-c_2)|^2 - \io ∇(c_1-c_2)(c_1-c_2)∇p_2\\
 &\quad -\io |∇(p_1-p_2)|^2 + \io ∇(p_1-p_2)(c_1-c_2)∇u_1 + \io ∇(p_1-p_2)u_1∇(c_1-c_2)\\
 &\quad +\io ∇(p_1-p_2)∇(u_1-u_2)c_2 + \io ∇(p_1-p_2)(u_1-u_2)∇c_2 \\
 &\le ((2+αC)C+μ+2Cμ)\io (u_1-u_2)^2 +5C\io (c_1-c_2)^2 + (4C-1)\io (p_1-p_2)^2 \\
 &\quad  + (5+αC)C \io |∇(c_1-c_2)|^2+ ((4+αC)C-1)\io |∇(p_1-p_2)|^2 \qquad \text{in } (0,T), 
 \end{align*}
we obtain $(u_1,c_1,p_1)=(u_2,c_2,p_2)$ from Grönwall's inequality. 
\end{proof}

Making use of Schauder's fixed point theorem and applying Lemmata~\ref{lm:ODEgeneral}--\ref{lm:uniqueness}, we now obtain a local existence result for the system \eqref{eq:system-w}.
\begin{lemma}\label{lm:local_ex}
  Assume that
  \begin{align}\label{eq:main_ass}
    \Omega \text{ is a smooth bounded domain in $\R^n$, $n \in \{1, 2, 3\}$, }
    \alpha, \mu > 0
  \end{align}
  and that
  \begin{align}\label{eq:local_ex:init_cond}
    w_0, c_0, p_0 \in \con{2+\gamma}
    \quad \text{for some $\gamma \in (0, 1)$ are nonnegative and fulfill} \quad
    \partial_\nu w_0 = 0 \text{ on $\partial \Omega$}.
  \end{align}
 Then \eqref{eq:system-w} has a nonnegative unique solution 
 \[
  (w,c,p)\in \kl{C^{2+\gamma,1+\frac{\gamma}{2}}(\Ombar\times(0,\Tmax))\cap C^1(\Ombar\times[0,\Tmax))}^3
 \]
 for some $\Tmax>0$, which can be chosen such that 
 \begin{equation}\label{eq:extensibility1}
 \Tmax=\infty \quad \text{ or }\quad \limsup_{t\upto \Tmax} \norm[C^{1+γ}(\Ombar)]{w(\cdot,t)} =\infty
 \end{equation}
\end{lemma}
\begin{proof}
 For $T>0$ and $M>0$ we introduce the set 
 \[
  S_{M,T} = \set{w\in C^0([0,T];C^1(\Ombar)) \mid 0\le w, \sup_{t\in[0,T]}\norm[C^1(\Ombar)]{w(\cdot,t)}\le M }
 \]
 and given $w\in S_{M,T}$ we let $(c,p)$ be the solution of \eqref{ode-cp} (cf.\ Lemma~\ref{lm:ODE} a)). We then let $v$ be the unique (weak) solution (see \cite[Theorem~III.5.1]{LSU}, \cite[Theorem~6.39]{lieberman_book}) of 
 \begin{equation}\label{def:w-locex}
  v_t = \f1{α} Δv + f\cdot ∇v + g v \text{ in } \Omega\times(0,T), \qquad ∂_{ν}v=0 \text{ on } ∂\Omega\times(0,T), \qquad v(\cdot,0)=w_0 \text{ in } \Omega, 
 \end{equation}
where $f=∇c$ and $g=αpc+μ-μ\ure^{αc}w$ belong to $L^\infty(\Om\times(0,T))$ according to Lemma~\ref{lm:ODE} a) and b), and define $Φ(w)=v$ in $\Ombar\times[0,T]$.

We choose $M>0$ such that
\begin{equation}\label{eq:chooseM}
 M>\norm[C^1(\Ombar)]{c_0}, \quad M>\norm[C^1(\Ombar)]{p_0}, \quad M>\norm[C^{1+γ}(\Ombar)]{w_0}, \qquad M>\norm[C^1(\Ombar)]{w_0}+1
\end{equation}
and introduce the constants $c_1=c_1(M)$ from  Lemma~\ref{lm:ODE} b) (for $T=1$), $c_2>0$ from \cite[Theorem~6.40]{lieberman_book} such that all solutions $v$ of \eqref{def:w-locex} with $\norm[L^\infty(\Om\times(0,T))]{f}\le c_1$, $\norm[L^\infty(\Om\times(0,T))]{g}\le αc_1^2+μ+μ\ure^{αc_1}M$ and $\norm[\Lom\infty]{w_0}\le M$ satisfy $\norm[L^\infty(\Om\times(0,T))]{v}\le c_2$, and $c_3>0$ such that by \cite[Theorems~1.1 and 1.2]{lieberman_gradient}, all solutions $v$ of \eqref{def:w-locex} with 
$∂_{ν}w_0=0$ on $∂\Omega$, 
$\norm[C^{1+γ}(\Ombar)]{w_0}\le M$, $\norm[L^\infty(\Om\times(0,T))]{f}\le c_1$, $\norm[L^\infty(\Om\times(0,T))]{g}\le αc_1^2+μ\ure^{αc_1}+μ\ure^{2αc_1}M$ 
and $\norm[L^\infty(\Om\times(0,T))]{v}\le c_2$ also fulfil 
$\nnorm{C^{1+γ,\f{1+γ}2}(\Ombar\times[0,T])}{v}\le c_3$. 
Finally, we fix $T\in(0,1]$ such that $c_3 \sqrt{T}\le 1$. 

Successive applications of Lemma~\ref{lm:ODE}, \cite[Theorem 6.40]{lieberman_book} and \cite[Theorems~1.1 and 1.2]{lieberman_gradient} then ensure that 
\begin{equation}\label{eq:PhiHolder}
 \nnorm{C^{1+γ,\f{γ}2}(\Ombar\times[0,T])}{Φ(w)} \le c_3.
\end{equation}

In particular, 
\[
 \norm[C^1(\Ombar)]{Φ(w)(\cdot,t)} \le \norm[C^1(\Ombar)]{Φ(w)(\cdot, 0)} + \norm[C^1(\Ombar)]{Φ(w)(\cdot, t)-Φ(w)(\cdot, 0)} \le \norm[C^1(\Ombar)]{w_0} + c_3 t^{\f12} \le M
\]
for every $t\in[0,T]$. Apparently, $\Phi$ maps $S_{M,T}$ to itself and, according to \eqref{eq:PhiHolder} has a compact image. Schauder's fixed point theorem provides a fixed point $w$ of $\Phi$, that is, (together with $c$ and $p$) a solution of \eqref{eq:system-w} in $\Omega\times[0,T]$. 

As $w\in C^{1+γ,0}(\Ombar\times[0,\Tmax))$, $p,c,∇c$ belong to $C^{γ,\f{γ}2}(\Ombar\times[0,\Tmax))$ by Lemma~\ref{lm:ODE} c) with $γ_2=0$.
Since moreover $w \in C^{\gamma, \frac{\gamma}{2}}(\Ombar \times [0, \tmax))$, also $f$ and $g$ in \eqref{def:w-locex} belong to this space.
Then \cite[Theorem~IV.5.3]{LSU} (if combined with the uniqueness statement of \cite[Theorem~III.5.1]{LSU} and $C^{2+γ}$ regularity of $w_0$) makes $w$ a classical solution with $w_t, D_x^2 w \in C^{\gamma, \frac{\gamma}{2}}(\Ombar \times [0, \tmax))$.
An invocation of Lemma~\ref{lm:ODE}~c) yields $D_x^2 c, D_x^2 p, \nabla p \in C^{\gamma, 1+ \frac{\gamma}{2}}(\Ombar \times [0, \tmax))$.
Another application of \cite[Theorem~IV.5.3]{LSU} to $\zeta w$ for a temporal cuttoff function $\zeta$ ensures that $\nabla w_t \in C^{\gamma, \frac{\gamma}{2}}(\Ombar \times (0, \tmax))$.
Moreover, nonnegativity of $w$ as well as of $c$ and $p$ follows from the comparison principles for parabolic and ordinary differential equations, respectively.

Since $M$ and $T$ only depend on the quantities in \eqref{eq:chooseM} and as solutions are unique by Lemma~\ref{lm:uniqueness}; the above reasoning can therefore be applied to extend the solution until some maximal existence time characterized by 
 \begin{equation}\label{eq:extensibility2}
 \Tmax=\infty \quad \text{ or } \quad \limsup_{t\upto \Tmax} \kl{\norm[C^1(\Ombar)]{c(\cdot,t)}+\norm[C^1(\Ombar)]{p(\cdot,t)}+\norm[C^{1+γ}(\Ombar)]{w(\cdot,t)}} =\infty.
 \end{equation}
According to Lemma~\ref{lm:ODE} b), boundedness of the norm of $w$ in this expression already implies that of the norms of $c$ and $p$, therefore \eqref{eq:extensibility2} can be reduced to \eqref{eq:extensibility1}. 

Finally, uniqueness of this solution has already been asserted in Lemma~\ref{lm:uniqueness}.
\end{proof}

By fixing initial data as in \eqref{eq:local_ex:init_cond},
we henceforth implicitly also fix the unique classical solution $(w, c, p)$ of \eqref{eq:system-w} given by Lemma~\ref{lm:local_ex} and denote its maximal existence time by $\tmax$.

\subsection{\tops{$L^\infty$}{L infty} bounds}
The results in the previous subsection show that Theorem~\ref{th:global_ex} follows once we have shown that for the solutions constructed in Lemma~\ref{lm:local_ex} the second alternative in \eqref{eq:extensibility1} holds.
That is, we need to derive sufficiently strong a~priori estimates.
In this subsection, we begin with bounds in $L^\infty$.

For the second solution component, such a bound directly follows from the comparison principle.
\begin{lemma}\label{lm:c_linfty}
  Suppose \eqref{eq:main_ass} and that $(w_0, c_0, p_0)$ satisfies \eqref{eq:local_ex:init_cond}.
  Then the solution $(w, c, p)$ constructed in Lemma~\ref{lm:local_ex} fulfills
  \begin{align*}
    \|c(\cdot, t)\|_{\leb\infty} \le \|c_0\|_{\leb\infty}
    \qquad \text{for all $t \in (0, \tmax)$}.
  \end{align*}
\end{lemma}
\begin{proof}
  The function $\ol c \defs \|c_0\|_{\leb\infty}$ is a supersolution of the second equation in \eqref{eq:system-w} and $c \ge 0$ by Lemma~\ref{lm:local_ex}.
\end{proof}

As a preparation for obtaining $L^\infty$ estimates also for the other two solution components,
we note that the time regularization in the third equation in \eqref{eq:system-w} implies that we can bound $p$ by a quantity including an \emph{arbitrarily small} contribution of the $L^\infty$ norm of $w$ --
at least if we are willing to shrink the time interval on which this estimates holds accordingly.
\begin{lemma}\label{lm:p_linfty_le_w_linfty}
  Suppose \eqref{eq:main_ass} and let $M > 0$. Then there exists $T^\star > 0$ such that for all $(w_0, c_0, p_0)$ satisfying \eqref{eq:local_ex:init_cond} and
  \begin{align*}
    \|c_0\|_{\leb\infty} \le M,
  \end{align*}
  the solution $(w, c, p)$ of \eqref{eq:system-w} fulfills
  \begin{align*}
    \|p\|_{L^\infty(\Omega \times (0, T))} \le \|p_0\|_{\leb\infty} + \min\left\{\frac{\mu}{2M\alpha}, 1\right\} \|w\|_{L^\infty(\Omega \times (0, T))}
    \qquad \text{for all $T \in (0, T^\star] \cap (0, \tmax)$}.
  \end{align*}
\end{lemma}
\begin{proof}
  We choose $T^\star > 0$ so small that $M \ure^{\alpha M} \left(1 - \ure^{-T^\star}\right) \le \min\{\frac{\mu}{2M\alpha}, 1\}$
  and fix $T \in (0, T^\star] \cap (0, \tmax)$.
  By the variation-of-constants formula and Lemma~\ref{lm:c_linfty}, we have 
  \begin{align*}
          \|p(\cdot, t)\|_{\leb\infty}
    &\le  \ure^{-t} \|p_0\|_{\leb\infty} + \int_0^t \ure^{-(t-s)} \|w \ure^{\alpha c} c\|_{\leb\infty}(\cdot, s) \ds \\
    &\le  \|p_0\|_{\leb\infty} + \|c_0\|_{\leb\infty} \ure^{\alpha \|c_0\|_{\leb\infty}} \|w\|_{L^\infty(\Omega \times (0, T))} \int_0^{T^\star} \ure^{-s} \ds \\
    &\le  \|p_0\|_{\leb\infty} + M \ure^{\alpha M} \left(1 - \ure^{-T^\star}\right) \|w\|_{L^\infty(\Omega \times (0, T))}
    \qquad \text{for all $t \in (0, T)$},
  \end{align*}
  which implies the statement due to the definition of $T^\star$.
\end{proof}

We now turn our attention to $L^\infty$ estimates of $w$.
The fact that the transformed quantity $w$ fulfills an equation whose first- and second-order terms reduce to $\ure^{-\alpha c} \nabla \cdot (\frac1\alpha \ure^{\alpha c} \nabla w)$
(cf.~\eqref{eq:transformation:12_order_terms}), an expression without any explicit $\nabla c$, opens the door for certain testing procedures.
In related works, these have been used to first derive boundedness in $L^p$ and then, after an iteration argument, also in $L^\infty$
(see for instance \cite[Proposition~5.1]{ref13}, \cite[Lemma~3.5]{tao2007global} and \cite[Lemma~3.10]{tao_win_haptotaxis_energy}).
However, here we are able to employ a slightly faster method:
Another advantage of the transformation $w = \ure^{-\alpha c} u$ is that sufficiently large constant functions are supersolutions of the equation for $w$ in \eqref{eq:system-w},
at least as long both $c$ and $p$ are bounded.
Therefore, the previous two lemmata allow us to prove boundedness for $w$ on small timescales.

We also emphasize that the following proof crucially relies on the presence of a logistic source in the first equation, that is, on positivity of $\mu$.
(The same would be true for testing procedures similar to those performed in the works referenced above.)
In fact, this is essentially the only place where we directly make use of the assumption $\mu > 0$.
\begin{lemma}\label{lm:w_linfty_t0}
  Suppose \eqref{eq:main_ass} and let $M > 0$.
  Let $T^\star > 0$ be as given by Lemma~\ref{lm:p_linfty_le_w_linfty}.
  Then there is $K > 0$ with the following property:
  For all $L > 0$ and all $(w_0, c_0, p_0)$ satisfying \eqref{eq:local_ex:init_cond} and
  \begin{align*}
    \|w_0\|_{\leb\infty} \le L, \quad
    \|c_0\|_{\leb\infty} \le M 
    \quad \text{as well as} \quad
    \|p_0\|_{\leb\infty} \le L,
  \end{align*}
  the solution $(w, c, p)$ of \eqref{eq:system-w} fulfills
  \begin{align*}
    \|w(\cdot, t)\|_{\leb\infty} + \|p(\cdot, t)\|_{\leb\infty} \le K (L+1)
    \qquad \text{for all $t \in (0, T^\star] \cap (0, \tmax)$}.
  \end{align*}
\end{lemma}
\begin{proof}
  We fix $T \in (0, T^\star] \cap (0, \tmax)$.
  By Lemma~\ref{lm:c_linfty} and Lemma~\ref{lm:p_linfty_le_w_linfty}, we may estimate 
  \begin{align*}
          w_t - \frac{1}{\alpha} \Delta w - \nabla c \cdot \nabla w
    &=    w \left( \alpha pc + \mu - \mu\ure^{\alpha c} w \right)
     \le  w \left( M\alpha \|p\|_{L^\infty(\Omega \times (0, T))} + \mu - \mu w \right) \\
    &\le  w \left( LM\alpha + \frac{\mu}{2} \|w\|_{L^\infty(\Omega \times (0, T))} + \mu - \mu w \right)
  \end{align*}
  in $\Omega \times (0, T)$.
  Therefore, the comparison principle, applied with the constant supersolution
  \begin{align*}
    \ol w \defs \max\left\{L, \frac{LM\alpha}{\mu} + 1 + \frac12 \|w\|_{L^\infty(\Omega \times (0, T))} \right\},
  \end{align*}
  asserts 
  \begin{align*}
          \|w\|_{L^\infty(\Omega \times (0, T))}
    &\le  \ol w
    \le   \max\left\{L, LM\alpha \mu^{-1} + 1\right\} + \frac{1}{2} \|w\|_{L^\infty(\Omega \times (0, T))}
  \end{align*}
  and thus
  \begin{align*}
          \|w\|_{L^\infty(\Omega \times (0, T))}
    \le   2 \max\left\{L, LM\alpha \mu^{-1} + 1\right\}.
  \end{align*}
  Since $\|p\|_{L^\infty(\Omega \times (0, T))} \le L + \|w\|_{L^\infty(\Omega \times (0, T))}$ by Lemma~\ref{lm:p_linfty_le_w_linfty},
  this implies the statement for $K \defs 5\max\{1, M\alpha \mu^{-1}\}$.
\end{proof}

Next, iteratively applying Lemma~\ref{lm:w_linfty_t0} allows us to derive boundedness for all solution components on all bounded time intervals.
We note that a prerequisite for such an iteration procedure is that the time $T^\star$ given by Lemma~\ref{lm:p_linfty_le_w_linfty} (and to a lesser extent also the constant $K$ given by Lemma~\ref{lm:w_linfty_t0})
only depends on the data in a manageable way. This justifies why we have kept track of the dependencies of the constants in the previous lemmata.
\begin{lemma}\label{lm:w_linfty_tmax}
  Suppose \eqref{eq:main_ass} and let $M > 0$. Then there are $C_1, C_2 > 0$ with the following property:
  For all $L > 0$ and all $(w_0, c_0, p_0)$ satisfying \eqref{eq:local_ex:init_cond} and
  \begin{align*}
    \|w_0\|_{\leb\infty} \le L, \quad
    \|c_0\|_{\leb\infty} \le M 
    \quad \text{as well as} \quad
    \|p_0\|_{\leb\infty} \le L,
  \end{align*}
  the solution $(w, c, p)$ of \eqref{eq:system-w} fulfills
  \begin{align}\label{eq:w_linfty_tmax:statement}
    \|p(\cdot, t)\|_{\leb\infty} + \|w(\cdot, t)\|_{\leb\infty} \le \ure^{C_1 t} C_2 (L+1)
    \qquad \text{for all $t \in (0, \tmax)$}.
  \end{align}
\end{lemma}
\begin{proof}
  We set $w(\cdot, t) = w_0$ and $p(\cdot, t) = p_0$ for $t < 0$,
  and let $T^\star > 0$ and $K > 1$ be as given by Lemma~\ref{lm:p_linfty_le_w_linfty} and Lemma~\ref{lm:w_linfty_t0}, respectively.
  Moreover, setting
  \begin{align*}
    I_j \defs ((j-1) T^\star, j T^\star] \cap (-\infty, \tmax)
    \qquad \text{for $j \in \N_0$}
  \end{align*}
  and applying Lemma~\ref{lm:p_linfty_le_w_linfty} (which is applicable for the same $M$ by Lemma~\ref{lm:c_linfty}) and Lemma~\ref{lm:w_linfty_t0} to initial data $(w, c, p)(\cdot, (j-1) T^\star)$,
  we can estimate
  \begin{align*}
          \|p\|_{L^\infty(\Omega \times I_j)}
    &\le  \|p\|_{L^\infty(\Omega \times I_{j-1})} + \|w\|_{L^\infty(\Omega \times I_j)}
  \intertext{and}
          \|w\|_{L^\infty(\Omega \times I_j)}
    &\le  K \left( \|w\|_{L^\infty(\Omega \times I_{j-1})} + \|p\|_{L^\infty(\Omega \times I_{j-1})} + 1 \right)
  \end{align*}
  for all $j \in \N$ with $(j-1) T^\star < \tmax$.
  However, the same estimates hold trivially also in the case of $(j-1) T^\star \ge \tmax$, as then $I_j = \emptyset$.
  Thus,
  \begin{align*}
    A_j \defs \|p\|_{L^\infty(\Omega \times I_j)} + \|w\|_{L^\infty(\Omega \times I_j)} + 1, \quad j \in \N,
  \end{align*}
  fulfills
  \begin{align*}
        A_j
    \le \|p\|_{L^\infty(\Omega \times I_{j-1})}
        + 2 K \left( \|w\|_{L^\infty(\Omega \times I_{j-1})} + \|p\|_{L^\infty(\Omega \times I_{j-1})} + 1 \right)
        + 1
    \le (2K+1) A_{j-1}
    \qquad \text{for all $j \in \N$}.
  \end{align*}
  A straightforward induction then yields $A_j \le (2K+1)^j A_0 \le \ure^{j \ln(2K+1)} (2L+1)$ for all $j \in \N$.
  If $t \in I_j$ for some $j \in \N_0$ and thus $(j-1) T^\star < t$, that is $j < \frac{t}{T^\star} + 1$, then
  \begin{align*}
        \|w(\cdot, t)\|_{\leb\infty} + \|p(\cdot, t)\|_{\leb\infty}
    \le A_j
    \le \ure^{j \ln(2K+1)} (2L+1)
    \le \ure^{t (T^\star)^{-1} \ln(2K+1)} 2(2K+1)(L+1).
  \end{align*}
  Since $(0, \tmax) \subset \bigcup_{j \in \N_0} I_j$, this implies \eqref{eq:w_linfty_tmax:statement} for $C_1 = \frac{\ln(2K+1)}{T^\star}$ and $C_2 = 2(2K+1)$.
\end{proof}

\subsection{Gradient bounds in \tops{$L^4$}{L4}}
While the $L^\infty$ estimates proven in the previous subsection surely form an important step towards proving global existence,
the extensibility criterion \eqref{eq:extensibility1} also requires boundedness of the gradients --
which will be the topic of the present and the following subsection.

As a first step, we again make use of the time regularization in the second and third equation in \eqref{eq:system-w} to obtain
\begin{lemma}\label{lm:c_w1q_rel}
  Suppose \eqref{eq:main_ass} and let $T_0 \in (0, \infty)$ as well as $q \in (1, \infty)$.
  For all $M > 0$, there exists $C > 0$ such that if $(w_0, c_0, p_0)$ satisfying \eqref{eq:local_ex:init_cond} are such that the corresponding solution $(w, c, p)$ of \eqref{eq:system-w} fulfills
  \begin{align}\label{eq:c_w1q_rel:w_c_p_bdd_M}
    w \le M, \quad
    c \le M
    \quad \text{and} \quad
    p \le M
    \qquad \text{in $\Ombar \times [0, T)$},
  \end{align}
  where $T \defs \min\{T_0, \tmax\}$,
  then
  \begin{align}\label{eq:c_w1q_rel:statement}
          \left( \intom |\nabla c(\cdot, t)|^q + \intom |\nabla p(\cdot, t)|^q\right)
    &\le  C \left( \intom |\nabla c_0|^q + \intom |\nabla p_0|^q + \int_0^t \intom |\nabla w|^q \right)
    \qquad \text{for all $t \in [0, T)$}.
  \end{align}
\end{lemma}
\begin{proof}
  According to \eqref{eq:system-w},
  \begin{align*}
    (\nabla c)_t = - p \nabla c - c \nabla p
    \quad \text{and} \quad
    (\nabla p)_t = - \nabla p  + w (\alpha c + 1) \ure^{\alpha c} \nabla c + \ure^{\alpha c} c \nabla w
  \end{align*}
  hold in $\Omega \times (0, T)$.
  By testing these equations with $q |\nabla c|^{q-2} \nabla c$ and $q |\nabla p|^{q-2} \nabla p$, respectively, and applying Young's inequality, we obtain
  \begin{align*}
        \ddt \intom |\nabla c|^q
    =   - q \intom p |\nabla c|^q
        - q \intom c |\nabla c|^{q-2} \nabla c \cdot \nabla p
    \le Mq \left( \intom |\nabla c|^q + \intom |\nabla p|^q \right)
  \end{align*}
  and
  \begin{align*}
          \ddt \intom |\nabla p|^q
    &=    - q \intom |\nabla p|^q
          + q \intom w (\alpha c + 1) \ure^{\alpha c} |\nabla p|^{q-2} \nabla p \cdot \nabla c
          + q \intom \ure^{\alpha c} c |\nabla p|^{q-2} \nabla p \cdot \nabla w \\
    &\le  M (\alpha  M + 1) \ure^{\alpha M} q \left( \intom |\nabla c|^q + \intom |\nabla p|^q \right)
          + M \ure^{\alpha M} q \left( \intom |\nabla p|^q + \intom |\nabla w|^q \right)
  \end{align*}
  in $(0, T)$.
  Thus, setting $c_1 \defs Mq + M (\alpha  M + 1) \ure^{\alpha M} q + M \ure^{\alpha M} q$, we can conclude
  \begin{align*}
        \ddt \left( \intom |\nabla c|^q + \intom |\nabla p|^q \right)
    \le c_1 \left( \intom |\nabla c|^q + \intom |\nabla p|^q \right) + c_1 \intom |\nabla w|^q
    \qquad \text{in $(0, T)$},
  \end{align*}
  which after an application of Grönwall's inequality turns into
  \begin{align*}
          \left( \intom |\nabla c(\cdot, t)|^q + \intom |\nabla p(\cdot, t)|^q\right)
    &\le  \ure^{c_1 t} \left( \intom |\nabla c_0|^q + \intom |\nabla p_0|^q\right)
          + c_1 \int_0^t \ure^{c_1(t-s)} \intom |\nabla w(\cdot, s)|^q \ds \\
    &\le  \ure^{c_1 T} \left( \intom |\nabla c_0|^q + \intom |\nabla p_0|^q\right)
          + c_1 \ure^{c_1 T} \int_0^t \intom |\nabla w|^q,
          \quad t \in (0, T),
  \end{align*}
  and thus asserts \eqref{eq:c_w1q_rel:statement} for $C \defs \max\{c_1, 1\} \ure^{c_1 T}$.
\end{proof}

Next, we follow \cite[Lemma~3.14]{tao_win_haptotaxis_energy} and test the first equation in \eqref{eq:system-w} with $-\Delta w$,
which when combined with Lemma~\ref{lm:c_w1q_rel} allows us to obtain certain gradient bounds first on small time scales and then by means of an iteration argument also on each finite time interval.
\begin{lemma}\label{lm:c_w14_t1}
  Suppose \eqref{eq:main_ass} and let $M > 0$. There exists $T_1 \in (0, \infty)$
  such that if $(w_0, c_0, p_0)$ satisfying \eqref{eq:local_ex:init_cond} are such that the corresponding solution $(w, c, p)$ of \eqref{eq:system-w} fulfills \eqref{eq:c_w1q_rel:w_c_p_bdd_M},
  we can find $C > 0$ such that
  \begin{align}\label{eq:c_w14_t1:statement}
    \intom |\nabla c(\cdot, t)|^4 \le C
    \qquad \text{for all $t \in [0, T_1] \cap [0, \tmax)$}.
  \end{align}
\end{lemma}
\begin{proof}
  By $c_1$, we denote the constant appearing in \eqref{eq:c_w1q_rel:statement} given by Lemma~\ref{lm:c_w1q_rel} applied to $T_0 = 1$ and $q = 4$.
  Moreover, the Gagliardo--Nirenberg inequality (cf.\ \cite{NirenbergEllipticPartialDifferential1959} or \cite[(A.2)]{FuestGlobalSolutionsHomogeneous2020} for this form) asserts that there is $c_2 > 0$ such that
  \begin{align}\label{eq:c_w14_t_1:gni}
    \intom |\nabla \varphi|^4 \le c_2 \intom |\Delta \varphi|^2 + c_2
    \qquad \text{for all $\varphi \in \con2$ with $\partial_\nu \varphi = 0$ in $\partial \Omega$ and $\|\varphi\|_{\leb\infty} \le M$}.
  \end{align}
  We then choose $T_1 \in (0, 1)$ so small that
  \begin{align}\label{eq:c_w14_t1:def_t1}
    T_1 c_1 c_2 \alpha^3 \le \frac{1}{8c_2\alpha}
  \end{align}
  holds and fix a solution $(w, c, p )$ of $\eqref{eq:system-w}$ with maximal existence time $\tmax$ fulfilling \eqref{eq:c_w1q_rel:w_c_p_bdd_M}.

  These choices now allow us to infer from Lemma~\ref{lm:c_w1q_rel} that
  \begin{align*}
          \intntom |\nabla c|^4
    &\le  c_1 \left( \intntom |\nabla c_0|^4 + \intntom |\nabla p_0|^4 + \int_0^T \int_0^t \intom |\nabla w(\cdot, \tau)|^4 \dtau \dt \right) \\
    &\le  c_1 \left( \intntom |\nabla c_0|^4 + \intntom |\nabla p_0|^4 + \int_0^T \int_0^T \intom |\nabla w(\cdot, \tau)|^4 \dtau \dt \right) \\
    &\le  T_1 c_1 \left( \intom |\nabla c_0|^4 + \intom |\nabla p_0|^4 + \intntom |\nabla w|^4 \right)
  \end{align*}
  holds for all $T \in (0, T_1] \cap (0, \tmax)$.
  Next, we test the first equation in \eqref{eq:system-w} with $-\Delta w$ to obtain
  \begin{align*}
          \frac12 \ddt \intom |\nabla w|^2
    &=    - \frac{1}{\alpha} \intom |\Delta w|^2
          - \intom (\nabla c \cdot \nabla w) \Delta w
          - \intom (\alpha p c w + \mu w - \mu \ure^{\alpha c} w^2) \Delta w \\
    &\le  - \frac{1}{2\alpha} \intom |\Delta w|^2
          + \alpha \intom |\nabla c \cdot \nabla w|^2
          + \underbrace{\alpha |\Omega| (\alpha M^3 + \mu M + \mu \ure^{\alpha M} M^2)^2}_{\sfed c_3}
  \end{align*}
  and thus
  \begin{align}\label{eq:c_w14_t1:integrated}
    &\pe  \frac12 \intom |\nabla w(\cdot, T)|^2
          - \frac12 \intom |\nabla w_0|^2
          + \frac{1}{2\alpha} \intntom |\Delta w|^2 \notag \\
    &\le  \alpha \intntom |\nabla c \cdot \nabla w|^2
          + T c_3 \notag \\
    &\le  \frac1{4 c_2 \alpha} \intntom |\nabla w|^4 + c_2 \alpha^3 \intntom |\nabla c|^4 + T_1 c_3 \notag \\
    &\le  \left( \frac1{4 c_2 \alpha} + T_1 c_1 c_2 \alpha^3 \right) \intntom |\nabla w|^4 + T_1 c_1 c_2 \alpha^3 \left( \intom |\nabla c_0|^4 + \intom |\nabla p_0|^4 \right) + T_1 c_3
  \end{align}
  for all $T \in (0, T_1] \cap (0, \tmax)$.
  Since \eqref{eq:c_w14_t1:def_t1} and \eqref{eq:c_w14_t_1:gni} imply
  \allowdisplaybreaks[0]
  \begin{align}\label{eq:c_w14_t1:est_nabla_w14_factor}
          \left( \frac1{4 c_2 \alpha} + T_1 c_1 c_2 \alpha^3 \right) \intntom |\nabla w|^4
    &\le  \left( \frac1{2 c_2 \alpha} - \frac1{8c_2 \alpha} \right) \intntom |\nabla w|^4 \notag \\
    &\le  \frac{1}{2\alpha} \intntom |\Delta w|^2 + \frac{T}{2\alpha} - \frac1{8c_2 \alpha} \intntom |\nabla w|^4
  \end{align}
  \allowdisplaybreaks
  for all $T \in (0, T_1] \cap (0, \tmax)$,
  we conclude from \eqref{eq:c_w14_t1:integrated} and \eqref{eq:c_w14_t1:est_nabla_w14_factor} that
  \begin{align*}
    \frac1{8c_2 \alpha} \intntom |\nabla w|^4 \le \frac12 \intom |\nabla w_0|^2 + T_1 c_1 c_2 \alpha^3 \left( \intom |\nabla c_0|^4 + \intom |\nabla p_0|^4 \right) + T_1 c_3 + \frac{T}{2\alpha}
  \end{align*}
  for all $T \in (0, T_1] \cap (0, \tmax)$.
  Again applying Lemma~\ref{lm:c_w1q_rel}, we finally see that \eqref{eq:c_w14_t1:statement} holds for some $C > 0$ (depending on $w_0$, $c_0$ and $p_0$).
\end{proof}

\begin{lemma}\label{lm:c_w14_tmax}
  Suppose \eqref{eq:main_ass} and that $(w_0, c_0, p_0)$ satisfies \eqref{eq:local_ex:init_cond}.
  For all $T \in (0, \tmax] \cap (0, \infty)$, there exists $C > 0$ such that the solution of \eqref{eq:system-w} fulfills
  \begin{align}\label{eq:c_w14_tmax:statement}
    \intom |\nabla c(\cdot, t)|^4 \le C
    \qquad \text{for all $t \in [0, T]$}.
  \end{align}
\end{lemma}
\begin{proof}
  Lemma~\ref{lm:c_linfty} and Lemma~\ref{lm:w_linfty_tmax} assert that \eqref{eq:c_w1q_rel:w_c_p_bdd_M} holds for some $M > 0$.
  We fix $T_1 \in (0, \infty)$ be as given by Lemma~\ref{lm:c_w14_t1} for this $M$.
  If $j \in \N_0$ is such that $T_1 j < T$, 
  an application of Lemma~\ref{lm:c_w14_t1} to the solution with initial data $(w, c, p)(\cdot, T_1 j)$
  shows that there is $c_j > 0$ such that \eqref{eq:c_w14_tmax:statement} holds with $C$ replaced by $c_j$ for all $t \in [T_1 j, T_1 (j+1)]$.
  Thus, the statement follows for $C \defs \max\{\, c_j : j \in \N_0, j < \frac{T}{T_1} \,\}$.
\end{proof}

\subsection{Hölder estimates for the gradients}
Lemma~\ref{lm:c_linfty}, Lemma~\ref{lm:w_linfty_tmax} and Lemma~\ref{lm:c_w14_tmax} provide several bounds for the right-hand side of the first equation in \eqref{eq:system-w},
which allow us to make use of parabolic regularity theory to iteratively improve our bounds.
In particular, we adapt the techniques developed in \cite[pages~791--792]{tao_win_haptotaxis_energy}, where only planar domains are considered, to the three-dimensional setting.

As it is used multiple times in the proof of Lemma~\ref{lm:global_ex} below, we first state the following consequence of maximal Sobolev regularity results.
\begin{lemma}\label{lm:max_sob_reg}
  Suppose that $\Omega \subset \R^n$, $n \in \N$, is a smooth, bounded domain.
  Let $T > 0$, $\alpha > 0$, $s \in (0, \infty)$ and $q \in (n, \infty)$.
  For any $M > 0$, there is $C > 0$ such that if $w_0 \in \con2$ with $\partial_\nu w_0 = 0$ on $\partial \Omega$, $f \in L^\infty((0, T); L^q(\Omega))$ and $g \in L^s((0, T); \leb{q})$
  are such that
  \begin{align}\label{eq:max_sob_reg:ass_M}
    \|w_0\|_{\con2} \le M, \quad
    \|f\|_{L^\infty((0, T); L^q(\Omega))} \le M
    \quad \text{and} \quad
    \|g\|_{L^s((0, T); \leb{q})} \le M,
  \end{align}
  then every solution $w \in C^{2, 1}(\Ombar \times (0, T)) \cap C^1(\Ombar \times [0, T))$ of
  \begin{align*}
    \begin{cases}
      w_t = \frac{1}{\alpha} \Delta w + f \cdot \nabla w + g & \text{in $\Omega \times (0, T)$}, \\
      \partial_\nu w = 0                                     & \text{on $\partial \Omega \times (0, T)$}, \\
      w(\cdot, 0) = w_0                                      & \text{in $\Omega$}
    \end{cases}
  \end{align*}
  with $|w| \le M$ in $\Omega \times (0, T)$
  fulfills
  \begin{align}\label{eq:max_sob_reg:statement}
    \|w_t\|_{L^{s}((0, T); \leb q)} + \|\Delta w\|_{L^{s}((0, T); \leb q)} + \|\nabla w\|_{L^{s}((0, T); \leb\infty)} \le C.
  \end{align}
\end{lemma}
\begin{proof}
  We fix the data $w_0$, $f$ and $g$ and a solution $w$ but emphasize that the constants $c_1$ and $c_2$ below only depend on $M$ (and $\Omega$, $T$, $\alpha$, $s$ and $q$).
  Since $f \cdot \nabla w + g \in L_{\loc}^s([0, T); \leb q)$ by assumption, \cite[Theorem~2.3]{GigaSohrAbstractEstimatesCauchy1991} asserts
  that $w$ is also the unique solution of \cite[(2.6)]{GigaSohrAbstractEstimatesCauchy1991}
  and thus that the estimate \cite[(2.7)]{GigaSohrAbstractEstimatesCauchy1991} holds.
  From \cite[(2.7)]{GigaSohrAbstractEstimatesCauchy1991} in conjunction with \eqref{eq:max_sob_reg:ass_M}, we hence obtain $c_1 > 0$ such that
  \begin{align}\label{eq:max_sob_reg:basic}
        \|w_t\|_{L^{s}((0, T); \leb q)} + \frac{1}{\alpha} \|\Delta w\|_{L^{s}((0, T); \leb q)}
    \le c_1 \|f \cdot \nabla w\|_{L^{s}((0, T); \leb q)} + c_1.
  \end{align}
  Since $q > n$, the embedding $\sob2q \embed \embed \sob1{\infty}$ is compact,
  so that an application of Ehrling's lemma combined with elliptic regularity (cf.\ \cite[Theorem~19.1]{FriedmanPartialDifferentialEquations1976}) shows that there is $c_2 > 0$ such that
  \begin{align}\label{eq:max_sob_reg:ehrling}
        \|\nabla \varphi\|_{\leb\infty}
    \le \frac{1}{2M c_1 \alpha} \|\Delta \varphi\|_{\leb q} + c_2 \|\varphi\|_{\leb \infty}
    \qquad \text{for all $\varphi \in \con2$ with $\partial_\nu \varphi = 0$ on $\partial \Omega$}.
  \end{align}
  Additionally relying on Minkowski's inequality, we thus obtain
  \begin{align*}
          \|f \cdot \nabla w\|_{L^{s}((0, T); \leb{q})}
    &\le  \|f\|_{L^\infty((0, T); L^q(\Omega; \R^n))} \|\nabla w\|_{L^{s}((0, T); \leb{\infty})}  \\
    &\le  M \left( \frac{1}{2M c_1 \alpha} \|\Delta w\|_{L^{s}((0, T); \leb{q})} + c_2 \|w\|_{L^{s}((0, T); \leb \infty)} \right) \\
    &\le  \frac{1}{2c_1 \alpha} \|\Delta w\|_{L^{s}((0, T); \leb q)} + M^2 T^\frac1{s} c_2.
  \end{align*}
  In combination with \eqref{eq:max_sob_reg:basic}, this yields
  \begin{align*}
        \|w_t\|_{L^{s}((0, T); \leb q)} + \frac{1}{2\alpha} \|\Delta w\|_{L^{s}((0, T); \leb q)}
    \le c_1 + M^2 T^\frac1s c_1 c_2,
  \end{align*}
  whenceupon another application of \eqref{eq:max_sob_reg:ehrling} implies \eqref{eq:max_sob_reg:statement} for some $C > 0$.
\end{proof}

\begin{lemma}\label{lm:global_ex}
  Suppose \eqref{eq:main_ass} and that $(w_0, c_0, p_0)$ satisfies \eqref{eq:local_ex:init_cond}.
  Then the maximal classical solution $(w, c, p)$ of \eqref{eq:system-w} given by Lemma~\ref{lm:local_ex} is global in time.
\end{lemma}
\begin{proof}
  We may without loss of generality assume that $\gamma$ in \eqref{eq:local_ex:init_cond} satisfies $\gamma < \frac{7}{12}$,
  let $(w, c, p)$ the solution of \eqref{eq:system-w} provided by Lemma~\ref{lm:local_ex} and
  suppose that on the contrary the maximal existence time $\tmax$ is finite.
  By Lemma~\ref{lm:c_linfty} and Lemma~\ref{lm:w_linfty_tmax},
  \begin{align*}
    w_t = \tfrac{1}{\alpha} \Delta w + \nabla c \cdot \nabla w + g
    \qquad \text{in $\Ombar \times [0, \tmax)$}
  \end{align*}
  holds for $g = \alpha p c w + \mu w - \mu \ure^{\alpha c} w^2 \in L^\infty(\Omega \times (0, \tmax))$.
  Moreover, the initial data fulfill \eqref{eq:local_ex:init_cond},
  $\nabla c$ belongs to $L^\infty((0, \tmax); \leb4)$ by Lemma~\ref{lm:c_w14_tmax} (applied to $T = \tmax < \infty$)
  and $w$ is bounded by Lemma~\ref{lm:w_linfty_tmax},
  hence an application of Lemma~\ref{lm:max_sob_reg} with $s=12$ and $q=4$ shows that \eqref{eq:max_sob_reg:statement} holds, which entails that
  \begin{align*}
    \|\nabla w\|_{L^{12}(\Omega \times (0, \tmax))} \le c_1
  \end{align*}
  for some $c_1 > 0$.
  Therefore, Lemma~\ref{lm:c_w1q_rel} (applied to $T_0 = \tmax$) asserts that $\nabla c \in L^\infty((0, \tmax); \leb{12})$,
  so that we may again apply Lemma~\ref{lm:max_sob_reg}, this time with $s=12$ and $q=12$, to obtain $c_2 > 0$ such that
  \begin{align*}
    \|w_t\|_{L^{12}(\Omega \times (0, \tmax))} + \|\Delta w\|_{L^{12}(\Omega \times (0, \tmax))} \le c_2.
  \end{align*}
  This in turn renders \cite[Lemma~II.3.3]{LSU} applicable,
  which asserts finiteness of $\|w\|_{C^{\frac{19}{12}, \frac{19}{24}}(\Ombar \times [0, \tmax])}$,
  contradicting the extensibility criterion in Lemma~\ref{lm:local_ex}.
  Thus our assumption that $\tmax$ is finite must be false.
\end{proof}

\subsection{Proof of Theorem~\ref{th:global_ex}}
The proof of Theorem~\ref{th:global_ex} has now been reduced to referencing some of the lemmata above.
\begin{proof}[Proof of Theorem~\ref{th:global_ex}]
  Lemma~\ref{lm:local_ex} asserts the local existence of a unique maximal classical solution of \eqref{eq:system-w},
  which is global in time by Lemma~\ref{lm:global_ex}.
  Therefore, the statement follows by transforming back to the original variables;
  that is, first setting $u(x, t) \defs w(x, t) \ure^{\alpha c(x, t)}$ for $x \in \Ombar$ and $t \in [0, \infty)$
  and then applying the transformation in \eqref{eq:trans_1}.
\end{proof}

\section{Weak formulation, discretization and numerical solution}
\label{sec_discretization}
In this section, we address the numerical realization of \eqref{eq:system}.
To this end, we first derive a weak formulation and then apply the Rothe method, 
namely, first temporal discretization using finite differences, and afterward
spatial discretization based on a Galerkin finite element scheme. Due to the highly 
nonlinear behavior, we then propose and implement a 
fixed-point algorithm to solve all three equations sequentially. 
Similar algorithms and implementations are available in 
the deal.II library \cite{dealii2019design,dealII91}, and we have former experiences 
in solving highly nonlinear coupled PDE systems, e.g., \cite{Wi13_fsi_with_deal}, but 
the algorithmic design, implementation and code verification of \eqref{eq:system}
in deal.II is novel to the best of our knowledge.

\subsection{Weak formulation}
Using integration by parts and the homogeneous boundary conditions, 
the variational formulation for the system (\ref{eq:system}) reads: 
Find $u,c,p \in L^2(0, T, H^1(\Omega)) $ 
with $u_t, c_t, p_t \in L^2(0, T, (H^{1}(\Omega))^*)$ 
and the initial conditions $u^0 = u(0) \in H^1(\Omega),
c^0 = c(0) \in L^2(\Omega),p^0 = p(0) \in L^2(\Omega)$
such that for almost all times $t\in (0,T)$, we have
\begin{equation}
\label{eq3}
\begin{aligned}
& \langle u_t , \phi^u \rangle + \frac{1}{\alpha} \int_{\Omega} \nabla u \cdot \nabla \phi^u\,dx 
- \chi  \int_{\Omega} u \nabla c \cdot \nabla \phi^u \, dx 
- \mu \int_{\Omega} u(1-u)\phi^u \, dx = 0
\quad\forall \phi^u \in C^{\infty}(\overline \Omega), 
\\
& \langle c_t , \phi^c \rangle + \int_{\Omega} pc \phi^c \, dx = 0 
\quad\forall \phi^c\in C^{\infty}(\overline \Omega), 
\\
& \langle p_t , \phi^p \rangle - \varepsilon^{-1} \int_{\Omega} (uc-p) \phi^p \, dx = 0 
\quad\forall \phi^p \in C^{\infty}(\overline \Omega). 
\end{aligned}
\end{equation}

\subsection{Temporal discretization and fixed point scheme}
Let us now proceed and subdivide the time interval 
$\left[ 0, T \right]$ into $N$ subintervals 
$\left[ 0, T\right] =  \cup_{n=0}^{N-1} \left[ t^n, t^{n+1} \right]$ with 
the uniform time steps $\Delta t = t^{n+1} - t^n, n=0,1,2, ..., N-1 $. 
We use $c^{n+1}(x) \defs c(x, t^{n+1}), \, p^{n+1}(x) \defs p(x, t^{n+1})$ 
and $u^{n+1}(x) \defs u(x, t^{n+1})$ to denote the approximation of the solutions 
at time $t^{n+1}$. 
Specifically, for time discretization we employ the well-known 
$\theta$ method allowing us to work with implicit $A$-stable time-stepping 
schemes for choice $\theta \in [0.5;1]$ in each equation.
Further, a fixed-point scheme is used to decouple the previous system and to 
treat the nonlinear and coupled terms. 

Then, a semi-discrete and linearized form of the system (\ref{eq3}) 
in the interval $[t^{n}, t^{n+1}]$ reads: 
For given 
$u_0^{n+1} = u^{n} , \, c_0^{n+1} = c^{n}$ and $p_0^{n+1} = p^{n}$
find $u_k^{n+1} \in H^1(\Omega), \, c_k^{n+1} \in H^1(\Omega) $ and $p_k^{n+1} \in H^1(\Omega)$ such that 

\begin{align*}
 &\quad \int_{\Omega} u_k^{n+1} \phi^u \, dx + \theta \Delta t 
\left( \frac{1}{\alpha} \int_{\Omega} \nabla  u_{k}^{n+1} \cdot \nabla \phi^u dx  
- \chi \int_{\Omega} u_{k}^{n+1} \nabla c_{k-1}^{n+1} \cdot \nabla \phi^u dx 
- \mu \int_{\Omega} u_k^{n+1} (1 - u_{k-1}^{n+1})  \phi^u dx \right) \\ 
& = \int_{\Omega} u^{n} \phi^u \, dx - (1-\theta) \Delta t 
\left( \frac{1}{\alpha}\int_{\Omega} \nabla u^n  \cdot \nabla \phi^u dx
- \chi \int_{\Omega} u^{n} \nabla c^{n} \cdot \nabla \phi^u dx
- \mu \int_{\Omega} u^n (1 - u^n)  \phi^u dx \right) \quad\forall \phi^u\in  C^{\infty}(\overline \Omega)
\end{align*}
and 
\begin{align*}
\int_{\Omega} c_k^{n+1}  \phi^c dx   
+ \theta \Delta t  \int_{\Omega} p_{k-1}^{n+1} c_k^{n+1}  \phi^c dx 
= \int_{\Omega} c^{n}  \phi^c dx  - (1 -\theta) \Delta t 
 \int_{\Omega} p^{n}c^{n}  \phi^c dx  \quad\forall \phi^c \in  C^{\infty}(\overline \Omega) 
\end{align*}
and
\allowdisplaybreaks[0]
\begin{align*}
&\quad \left( 1 + \varepsilon ^{-1} \theta \Delta t \right) 
 \int_{\Omega} p_k^{n+1}  \phi^p dx   \\
& =\left( 1 - \varepsilon ^{-1} (1 - \theta) \Delta t \right) 
\int_{\Omega} p^{n}  \phi^p dx   +  \varepsilon ^{-1} \theta \Delta t 
 \int_{\Omega} u_{k}^{n+1} c_{k}^{n+1}  \phi^p dx +  \varepsilon ^{-1} 
(1 - \theta) \Delta t  \int_{\Omega} u^{n}c^{n}  \phi^p dx 
\quad\forall \phi^p \in  C^{\infty}(\overline \Omega) 
\end{align*}
\allowdisplaybreaks
for $k = 1, 2, \ldots, k^*$, where $k^*$ is the iteration index where 
some stopping criterion is met, and for $n=0,1,\ldots, N-1$. 
For details on the specific 
steps and stopping tolerances we refer the reader to Section~\ref{sec_algo}.

\subsection{Spatial Galerkin discretization with finite elements}
Our spatial discretization is based on a Galerkin finite 
element scheme using conforming finite elements 
(bilinear in two dimensions and trilinear in three dimensions). We refer 
the reader to the classical textbook \cite{Cia87} for more details.
To this end, $\Omega$ is decomposed into $\Omega_h$ using quadrilaterals or
hexahedra. Then, a conforming subspace $V_h\subset H^{1}(\Omega)$ for
approximating $u_h^{n+1}, c_h^{n+1}$ and $p_h^{n+1}$ is designed, 
which is composed of $Q_1^c$ functions. Moreover, we denote by $(\cdot,\cdot)$ 
the scalar product in $L^2(\Omega)$.

The discrete solutions $u_h^{n+1}, c_h^{n+1}$ and $p_h^{n+1}$ 
are written as linear combinations of standard basis functions of $V_h$:
$$ u_h^{n+1}(x) = \sum\limits_{i=1}^M u_{i}^{n+1} \phi_i(x), \quad
   c_h^{n+1}(x) = \sum\limits_{i=1}^M c_{i}^{n+1} \phi_i(x),\quad
   p_h^{n+1}(x) = \sum\limits_{i=1}^M p_{i}^{n+1}\phi_i(x), 
$$
where $M$ denotes the number of spatial degrees of freedom,
i.e., $\dim V_h =  M$.
The fully discrete system then reads as follows:

\begin{align} 
&\quad\sum \limits_{i=1}^M \left[ (\phi_i , \phi_j) + \theta \Delta t \left(
\frac{1}{\alpha} (\nabla \phi_i,\nabla \phi_j) - \chi (\phi_i \nabla
c_{h,k-1}^{n+1}, \nabla \phi_j) - \mu (\phi_i (1 - u_{h,k-1}^{n+1}) , \phi_j) \right) \right] u_{i,k}^{n+1} \nonumber \\
&= \sum \limits_{i=1}^M \left[ (\phi_i , \phi_j) - (1 -\theta) \Delta t \left( \frac{1}{\alpha} (\nabla \phi_i,\nabla \phi_j) - \chi (\phi_i \nabla c_{h}^{n}, \nabla \phi_j) - \mu (\phi_i (1 - u_{h}^{n}) , \phi_j) \right)  \right] u_{i}^{n}, \label{eq5.1}\\
&\text{and} \nonumber \\
&\quad \sum \limits_{i=1}^M \left[ (\phi_i , \phi_j) 
+ \theta \Delta t ( p_{h,k-1}^{n+1} \phi_i, \phi_j) \right] c_{i,k}^{n+1}  
= \sum \limits_{i=1}^M \left[ (\phi_i , \phi_j) - (1 - \theta) \Delta t ( p_h^{n} \phi_i, \phi_j) \right] c_{i}^{n}, \label{eq5.2}\\
&\text{and} \nonumber \\
&\quad\sum \limits_{i=1}^M \left[ (1 + \varepsilon^{-1} \theta \Delta t)(\phi_i , \phi_j) \right] p_{i,k}^{n+1} \nonumber \\
&= \sum \limits_{i=1}^M \left[ (1 - \varepsilon^{-1} (1 - \theta) \Delta t)(\phi_i
, \phi_j)  \right] p_{i}^{n} + \varepsilon^{-1} \theta \Delta t ( u_{h,k}^{n+1}
c_{h,k}^{n+1}, \phi_j) + \varepsilon^{-1} (1 -\theta) \Delta t ( u_{h}^{n}
c_{h}^{n}, \phi_j),
\label{eq5.3}
\end{align}
where $j=1,\ldots,M$ and the unknown solution coefficients
at each fixed-point iteration $k$ and each time step $n+1$ are 
$u_{h,k}^{n+1} \defs \{u_{i,k}^{n+1}\}_{i=1}^M \in\mathbb{R}^M$, 
$c_{h,k}^{n+1} \defs \{c_{i,k}^{n+1}\}_{i=1}^M\in\mathbb{R}^M$
and $p_{h,k}^{n+1} \defs \{p_{i,k}^{n+1}\}_{i=1}^M\in\mathbb{R}^M$.
Each linear system is solved with a sparse direct solver.

\subsection{Algorithm}
\label{sec_algo}
Collecting all pieces from the previous subsections, we arrive 
at the following final algorithm.

\begin{algo}[Fixed-point iterative scheme]
Let the fully discrete form \eqref{eq5.1}--\eqref{eq5.3} be given.
\begin{enumerate}
\item[Step 1]: initialize at time $t=0$ for $n=0$ 
with $u_h^0 =i_h\,u_0$, $c_h^0 = i_h\,c_0$ and $p_h^0 = i_h\,p_0$, where $i_h$
is the standard Lagrange interpolation operator,
\item[Step 2]: for $n\geq 0$ (time step number index)\\
        set $u_{h,0}^{n+1} = u_h^{n}, c_{h,0}^{n+1} = c_h^{n}$ and 
$p_{h,0}^{n+1} = p_h^{n}$\\
        for $k\geq 1$ (fixed-point iteration index)
        \begin{itemize}
 \item[(a)] Given $u_h^{n},c_h^{n}$ and 
$u_{h,k-1}^{n+1},c_{h,k-1}^{n+1}$. Determine $u_{h,k}^{n+1}$ with 
 \eqref{eq5.1}.
\item[(b)] Given $p_h^{n}$ and 
$p_{h,k-1}^{n+1}$. Determine $c_{h,k}^{n+1}$ with 
 \eqref{eq5.2}.
\item[(c)] Given $u_h^{n},c_h^{n},p_h^{n}$ and 
$u_{h,k}^{n+1},c_{h,k}^{n+1}$. Determine $p_{h,k}^{n+1}$ with 
 \eqref{eq5.2}.
        \item[(d)] if $\left\lbrace 
\Vert u_{h,k}^{n+1} - u_{h,k-1}^{n+1}\Vert_{l^2},
\Vert c_{h,k}^{n+1} - c_{h,k-1}^{n+1}\Vert_{l^2}, 
\Vert p_{h,k}^{n+1} - p_{h,k-1}^{n+1}\Vert_{l^2} 
\right\rbrace  < Tol=10^{-8} $ stop and set 
$$u_h^{n+1} = u_{h,k}^{n+1}, \, c_h^{n+1} = c_{h,k}^{n+1}, \, p_h^{n+1} = p_{h,k}^{n+1},$$
increment $n\mapsto n+1$ and go back to step 2 (proceed to next time point)
\item[(e)] else set
 $$u_{h,k}^{n+1} = \beta u_{h,k}^{n+1} + (1-\beta)u_{h,k-1}^{n+1}, $$
 $$c_{h,k}^{n+1} = \beta c_{h,k}^{n+1} + (1-\beta)c_{h,k-1}^{n+1}, $$
 $$p_{h,k}^{n+1} = \beta p_{h,k}^{n+1} + (1-\beta)p_{h,k-1}^{n+1}, $$
   for some $\beta \in [0, 1]$ and go to (a) 
and increment $k\mapsto k+1$ (next fixed-point iteration);
here we set $\beta = 0.5$.          
\end{itemize}
\end{enumerate}
\end{algo}
\begin{remark}
The system of algebraic equations of each equation at each step is solved 
using the sparse direct solver UMFPACK \cite{DaDu97}.
\end{remark}
\begin{remark}
We notice that the relaxation parameter 
$\beta$ can also be obtained via a backtracking procedure starting 
with $\beta=1$ and constructing a sequence with $\beta\to 0 $ for $k\to\infty$
until convergence is achieved.
\end{remark}
\begin{remark}
A rigorous numerical convergence analysis in weak function spaces of the discretized equations 
for $\Delta t\to 0$ and $h\to 0$ towards their continuous limits exceed the purpose 
of this paper and is left for future studies. However, there is hope 
for convergence in light of the classical solutions obtained in Section~\ref{sec_existence_classical}.
\end{remark}


\section{Numerical simulations}
\label{sec_tests}
In this section, we perform several numerical simulations in 
two and three spatial dimensions. The 
main objective are investigations of the influence of variations in the 
proliferation coefficient $\mu$
and the haptotactic coefficient $\chi$.

\subsection{Geometry, final time, parameters, and initial conditions} 
For all the experiments except those in Subsection~\ref{sec:3d}, the 
computations are performed on the 
square domain $\Omega = (0, 20)^2$, 
discretized uniformly using quadrilateral elements. 
This mesh is uniformly refined 5 times at the beginning of the computation 
resulting into $1089$ degrees of 
freedom.
The final time is $T=50$, we set $\theta = 0.5$ and use 
as time step size $\Delta t = 1$. As initial conditions, we use 
$$u_0(x) = \exp (-x^2), 
\quad  c_0(x) = 1 - 0.5 \exp (-x^2), \quad p_0(x) = 0.5 \exp (-x^2),$$
in all computations, unless otherwise mentioned.
As parameters, we use the fixed values $\alpha = 10$ and  $\varepsilon = 0.2$, while $\mu$ and $\chi$ are varied and specified in each respective 
subsection below. We notice that the smoothness conditions 
on the domain and satisfaction of the boundary conditions at the initial 
time $t=0$ are violated in this section in comparison to our theory 
established in Section~\ref{sec_existence_classical}.

\subsection{Simulations for different proliferation coefficients \tops{$\mu$}{mu}}
 First, we study the influence of the cancer cell proliferation 
coefficient on the cancer invasion for $\mu = 10^{-10}$, $0.5$, $1.0$ with small haptotactic rate $\chi=0.01$. 
We notice that our theory in Section~\ref{sec_existence_classical}
requires $\mu>0$ and for this reason we made the previous choice
$\mu=10^{-10}$. Numerically we are interested in a value 
being close to zero in order to study the behavior of the cancer invasion 
model.
Proliferation shows the ability of a cancer cell to copy its DNA and divide
into 2 cells, therefore an increase in the proliferation rate of tumours causes
an accelerated invasion of cancer cells into connective tissues domain. In all
the computations we use $\alpha^{-1} = 0.1$, $\varepsilon = 0.2$.

The results obtained with the standard Galerkin discretization of the system
(\ref{eq:system}) are displayed in Figs.~\ref{fig1}--\ref{fig6}, at time instances
5, 15, 25 and 35. The snapshots of cancer cell invasion, connective tissue and
protease are plotted in Fig.~\ref{fig1}, Fig.~\ref{fig3} and Fig.~\ref{fig5}.
We start with $\mu = 10^{-10}$, that is, almost no growth in the cancer cell
density. As we can see from Fig.~\ref{fig1}, there is no growth in the cancer during the initial stage, and despite a small amount of concentration at the initial period, the cancer cell density and also protease (which is produced by cancer cells upon contact with connective tissues) are decreased and spread slowly due to diffusion effect and the invasion does not continue after time $t=15$. Now, let us consider $\mu = 0.5$. 
As we can see from Fig.~\ref{fig3}, the growth starts at $t=15$, and it
continues during the time. The cancer invasion gradually increases and degrades
nearly half of the connective tissue by the time $t=25$. Fig.~\ref{fig5} shows
the growth effect at the initial stage itself, due to high proliferation rate,
cancer cells produce more protease, which helps them to invade the connective
tissues area rapidly. In particular, cancer cells complete invasion in three-quarters of the connective tissue domain at $t=25$ when $\mu = 1.0$ is used.
The snapshots of cancer cell invasion for different values of proliferation
rate are given in Fig.~\ref{fig2}, Fig.~\ref{fig4} and Fig.~\ref{fig6}. As 
explained, by increasing the value of $\mu$, cancer cells increase and the invasion happens more rapidly for all the considered time intervals.\\ 
\begin{figure}[H]
	\centering
	\begin{subfigure}{0.24\textwidth}
		\includegraphics[width=\textwidth]{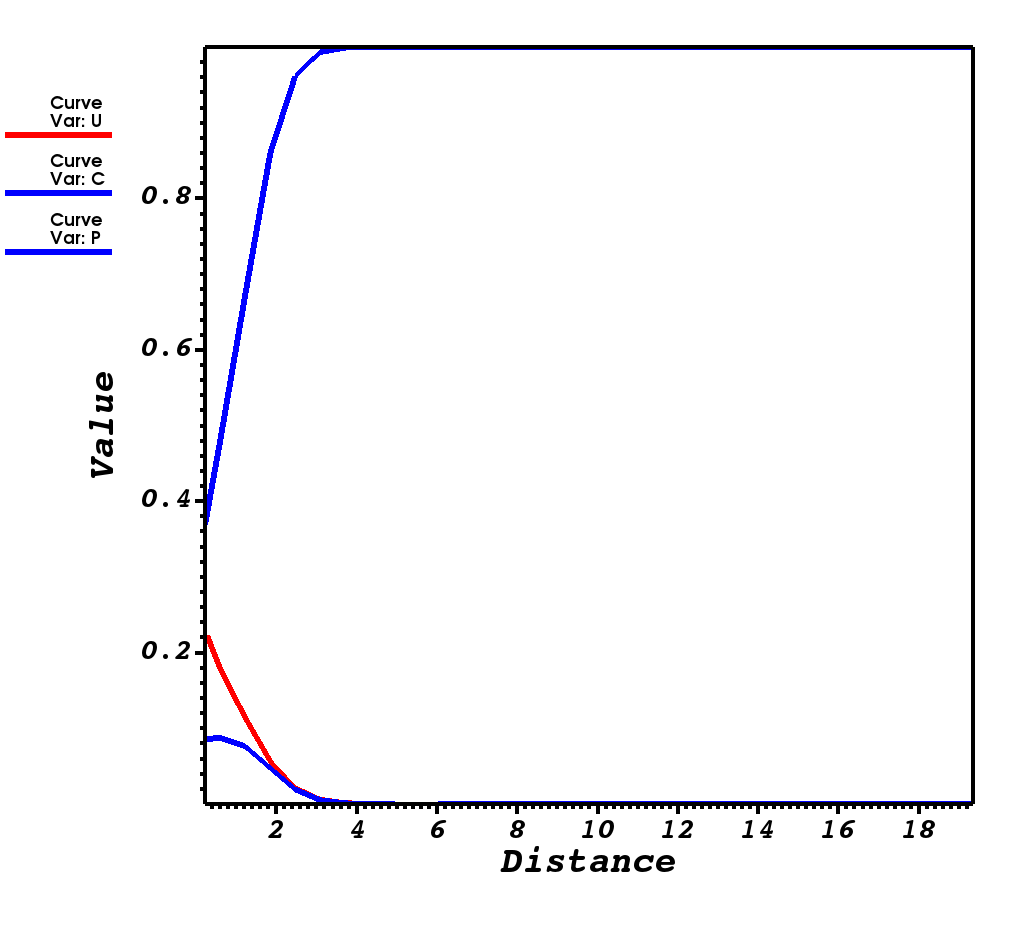}
		\caption{s = 5}
	\end{subfigure}
	\begin{subfigure}{0.24\textwidth}
		\includegraphics[width=\textwidth]{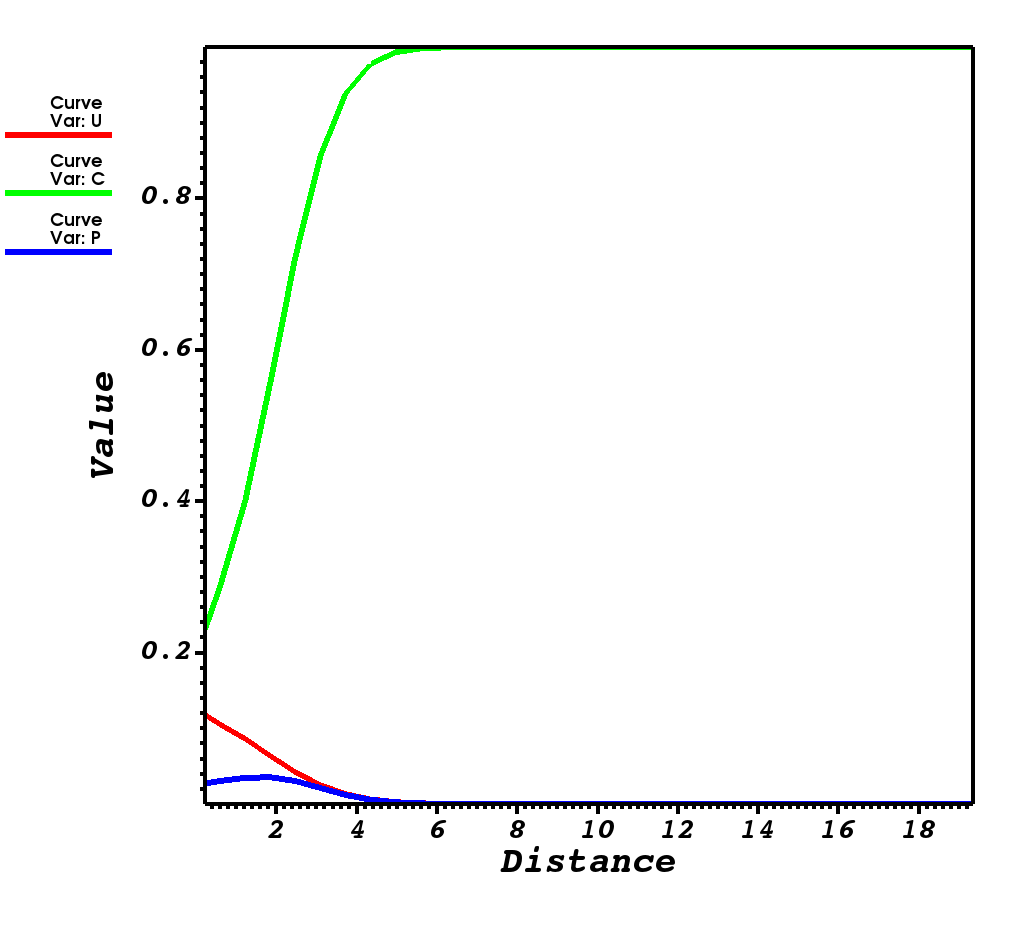}
		\caption{t = 15}
	\end{subfigure}
	\begin{subfigure}{0.24\textwidth}
		\includegraphics[width=\textwidth]{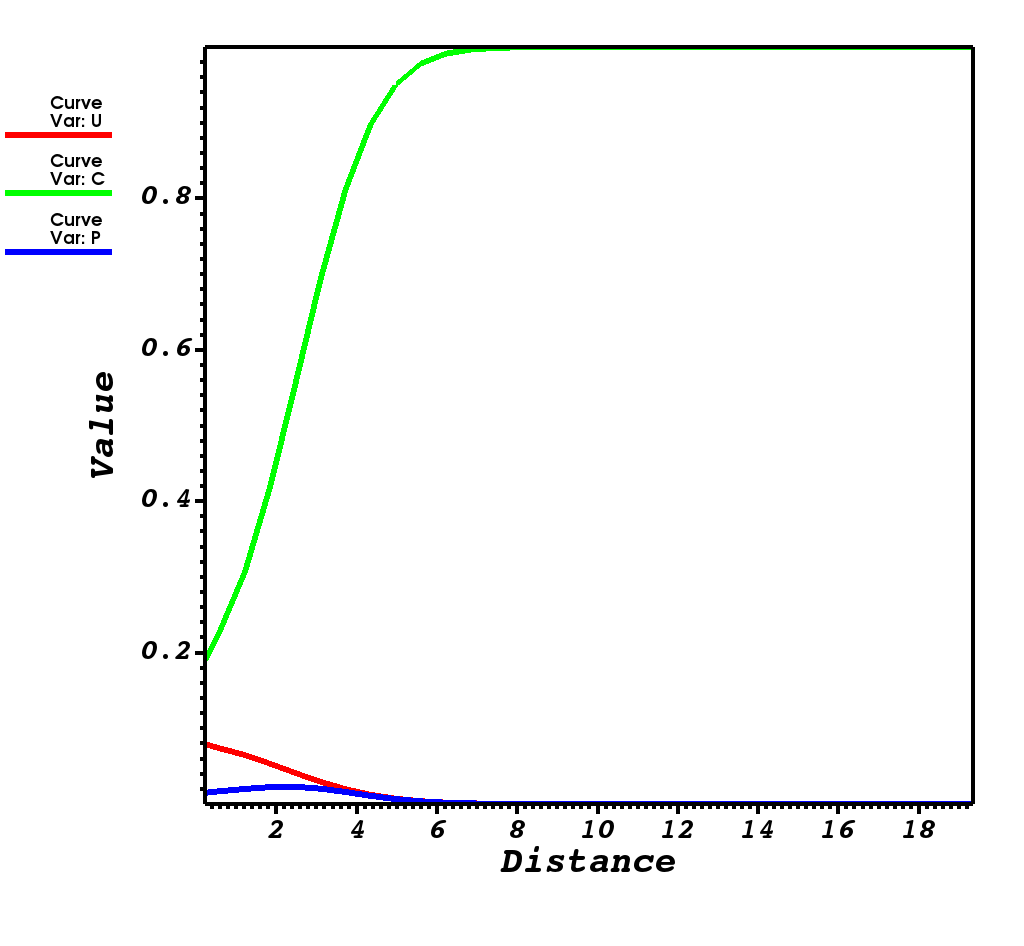}
		\caption{t = 25}
	\end{subfigure}
	\begin{subfigure}{0.24\textwidth}
		\includegraphics[width=\textwidth]{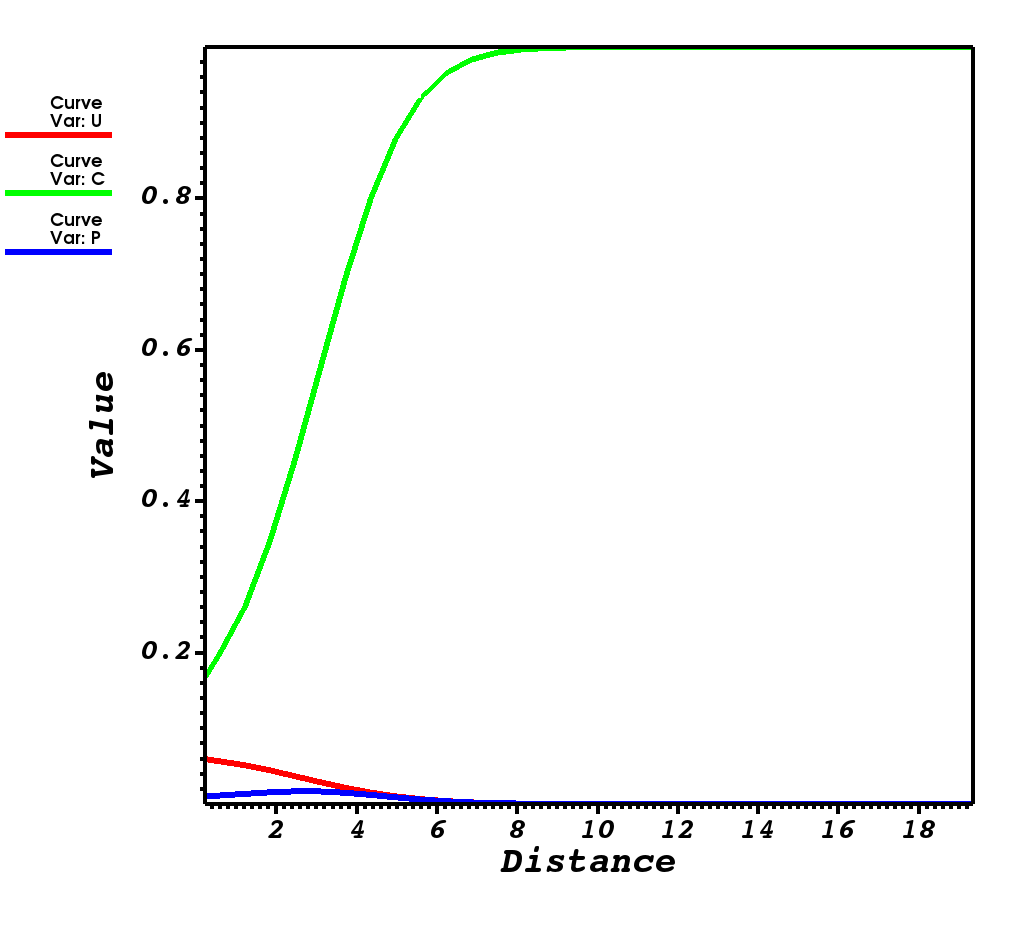}
		\caption{t = 35}
	\end{subfigure}
	\caption{
	The effect of the proliferation rate on cancer cell invasion $u$, 
connective tissue $c$ and protease $p$ at different time steps, $t=5, 15, 25, 35$ for  $\mu = 10^{-10}$.
	 }
	\label{fig1}
\end{figure} 

\begin{figure}[H]
	\centering
	\begin{subfigure}{0.24\textwidth}
		\includegraphics[width=\textwidth]{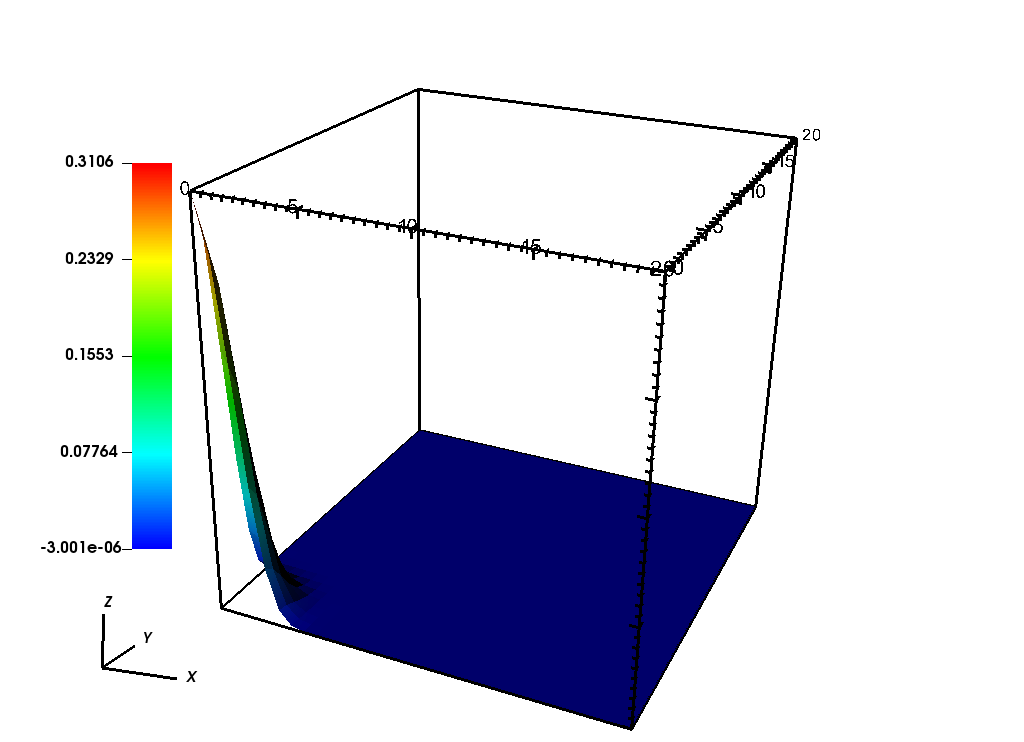}
		\caption{t = 5}
	\end{subfigure}
	\begin{subfigure}{0.24\textwidth}
		\includegraphics[width=\textwidth]{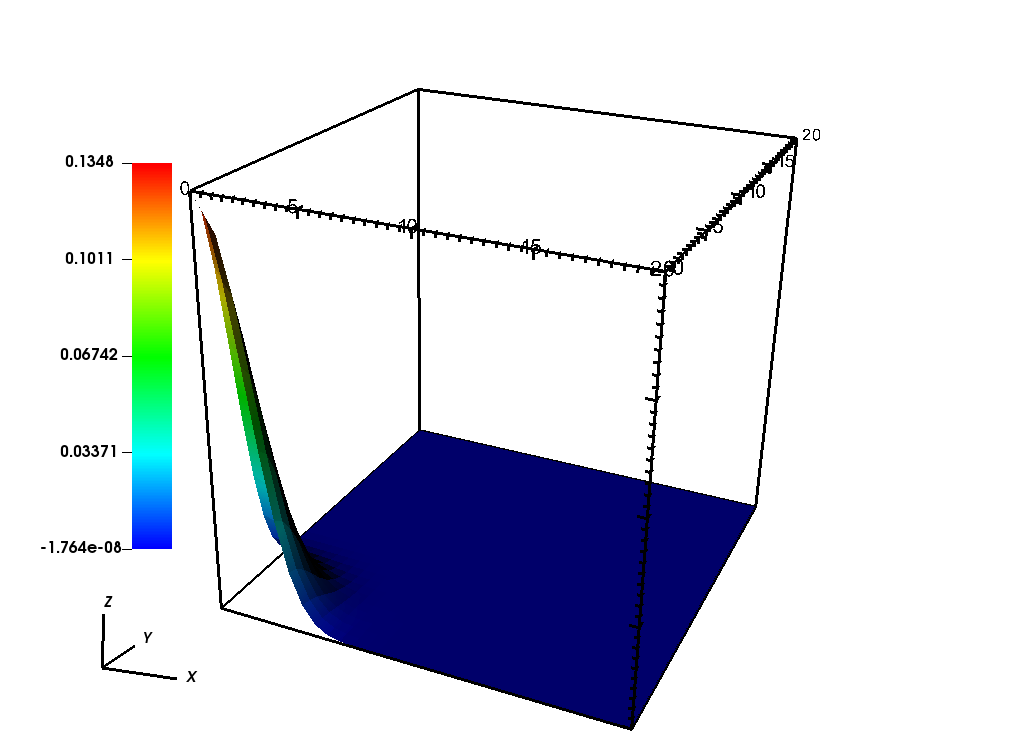}
		\caption{t = 15}
	\end{subfigure}
	\begin{subfigure}{0.24\textwidth}
		\includegraphics[width=\textwidth]{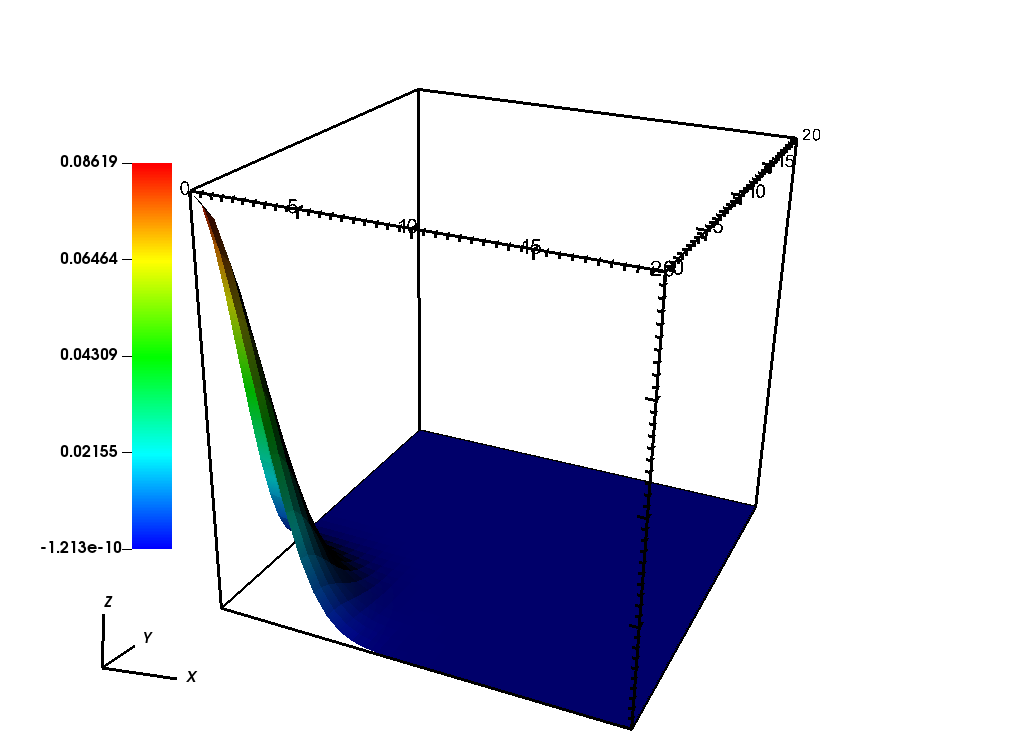}
		\caption{t = 25}
	\end{subfigure}
	\begin{subfigure}{0.24\textwidth}
		\includegraphics[width=\textwidth]{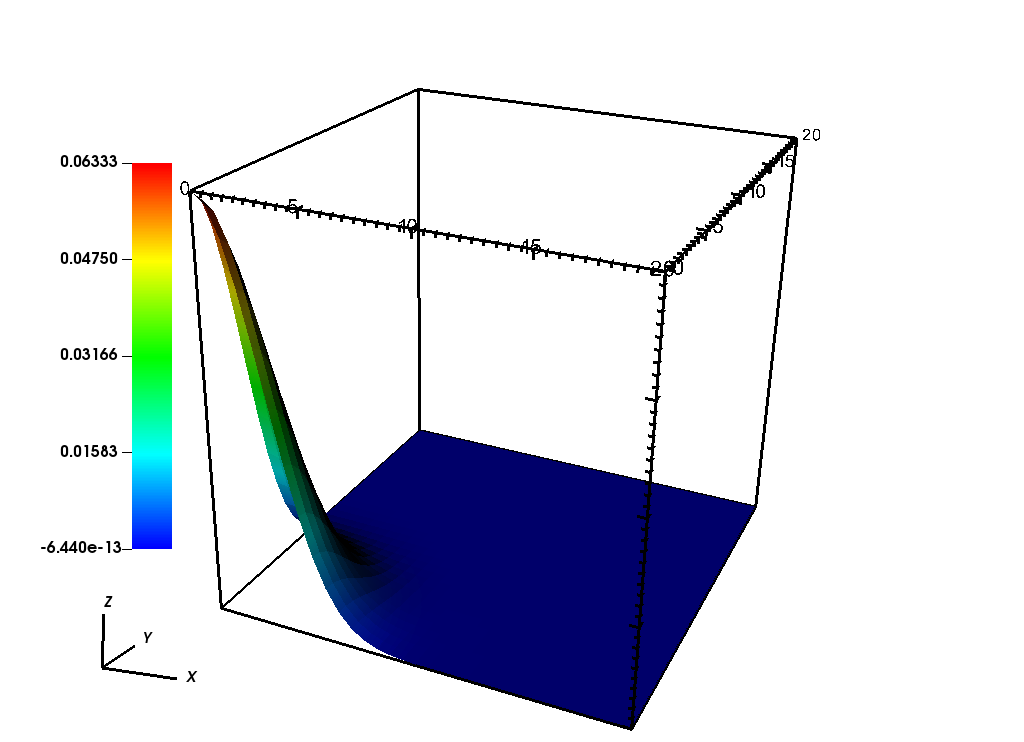}
		\caption{t = 35}
	\end{subfigure}
	\caption{
	The snapshots of cancer cell invasion $u$ for $ \mu = 10^{-10}$, the maximum amount of cancer cells decreasing from left to right is 0.3106, 0.1348, 0.08619, and 0.06333.
	 }
	\label{fig2}
\end{figure} 
 
\begin{figure}[H]
	\centering
	\begin{subfigure}{0.24\textwidth}
		\includegraphics[width=\textwidth]{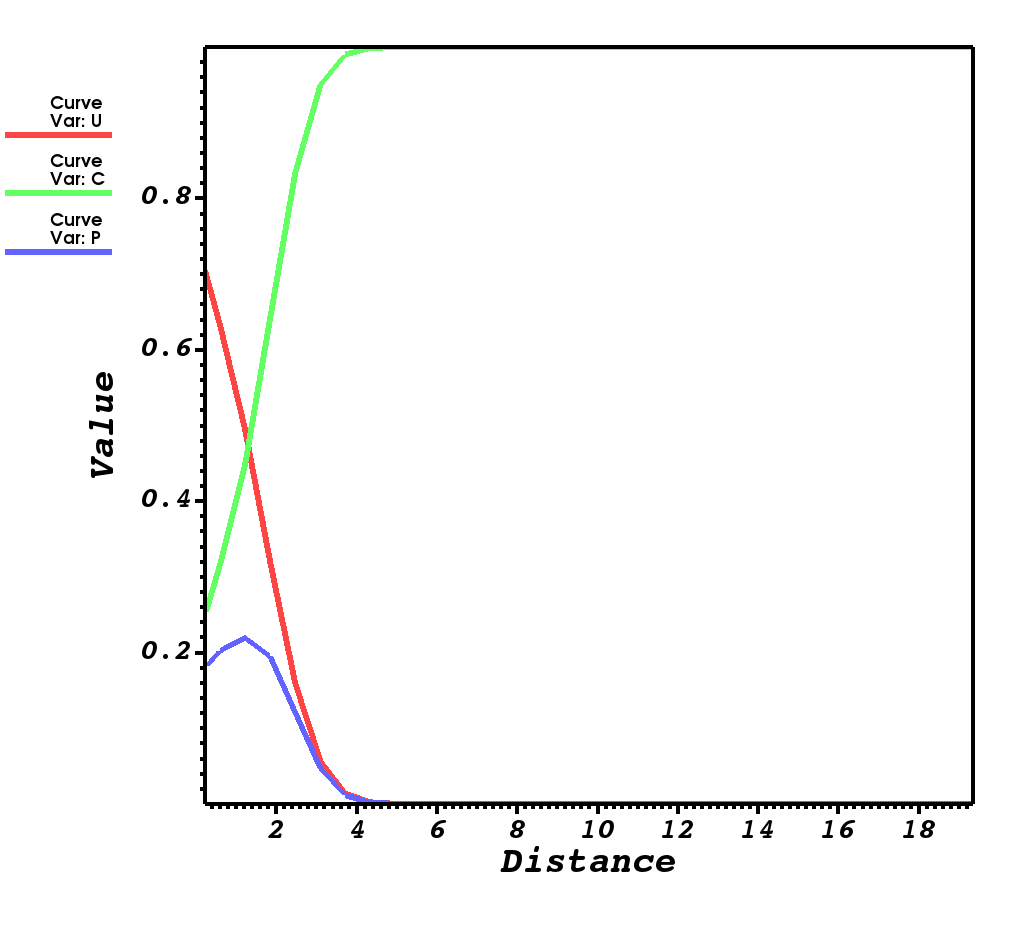}
		\caption{t = 5}
	\end{subfigure}
	\begin{subfigure}{0.24\textwidth}
		\includegraphics[width=\textwidth]{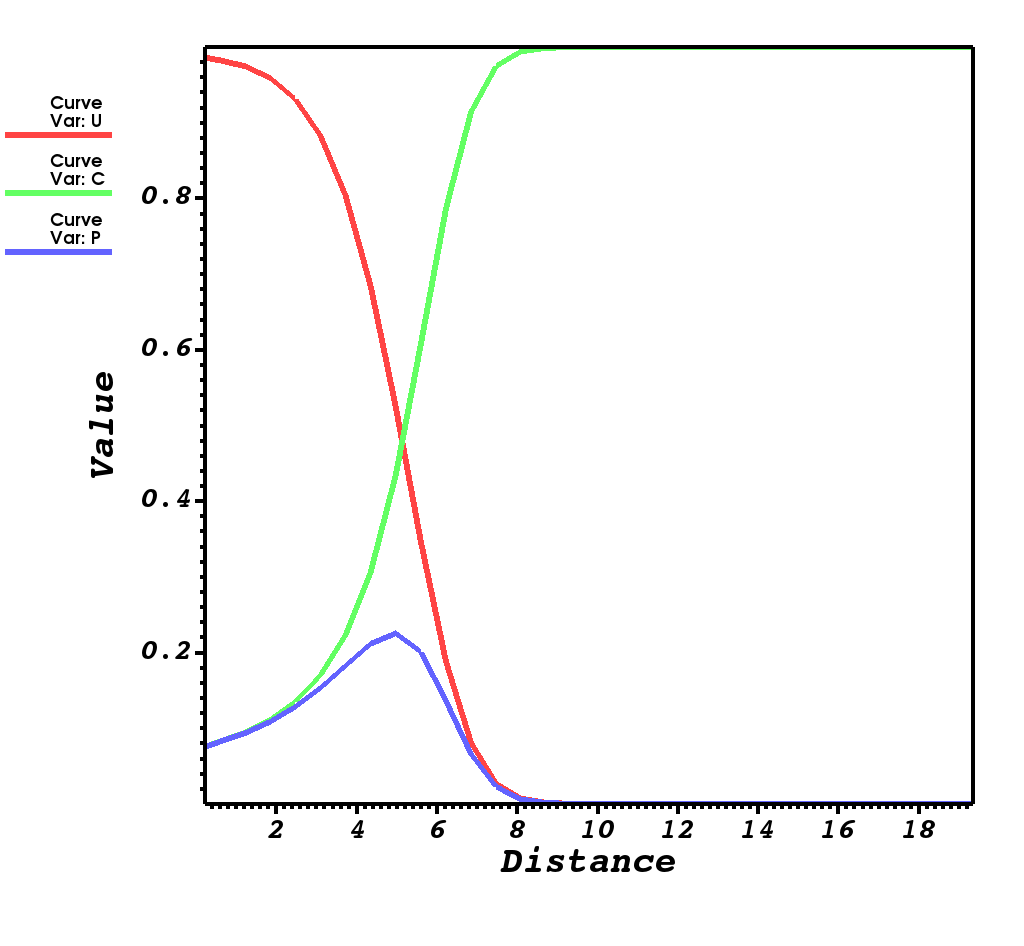}
		\caption{t = 15}
	\end{subfigure}
	\begin{subfigure}{0.24\textwidth}
		\includegraphics[width=\textwidth]{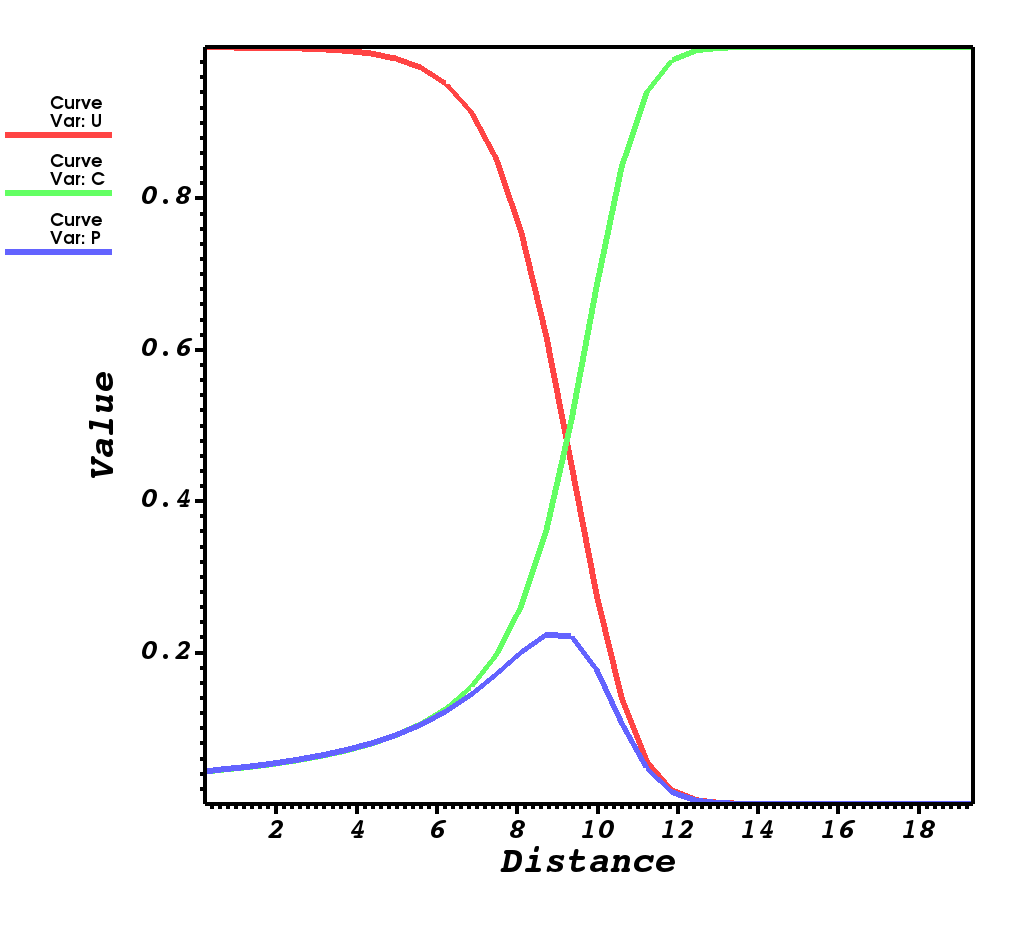}
		\caption{t = 25}
	\end{subfigure}
	\begin{subfigure}{0.24\textwidth}
		\includegraphics[width=\textwidth]{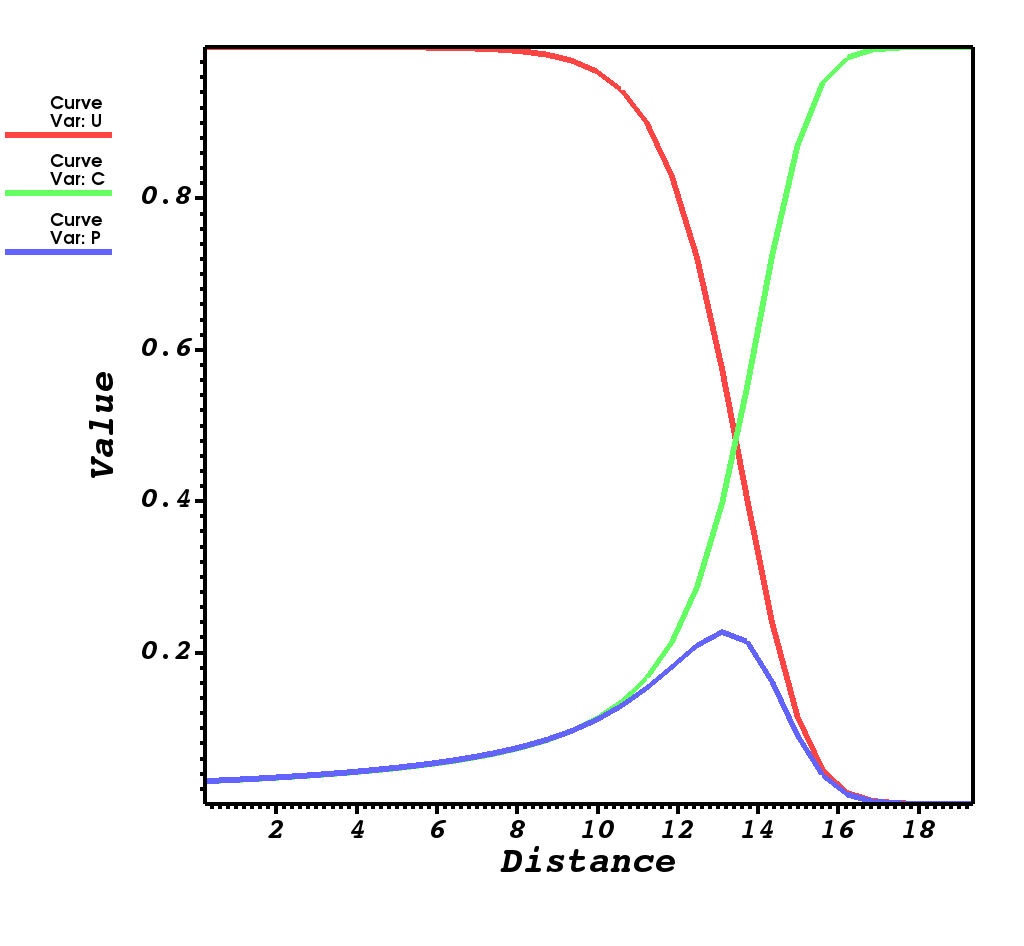}
		\caption{t = 35}
	\end{subfigure}
	\caption{
	The effect of proliferation rate on cancer cell invasion $u$, connective tissue $c$ 
and protease $p$ at different time steps, $t=5, 15, 25, 35$ for  $\mu = 0.5$.
	 }
	\label{fig3}
\end{figure} 

\begin{figure}[H]
	\centering
	\begin{subfigure}{0.24\textwidth}
		\includegraphics[width=\textwidth]{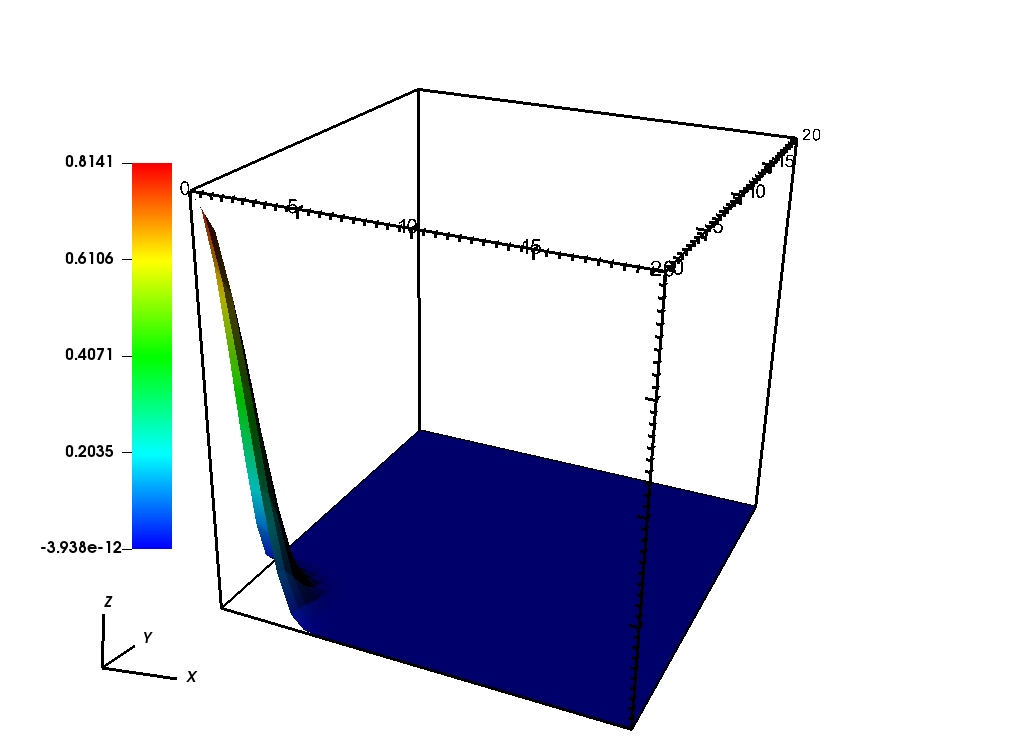}
		\caption{t = 5}
	\end{subfigure}
	\begin{subfigure}{0.24\textwidth}
		\includegraphics[width=\textwidth]{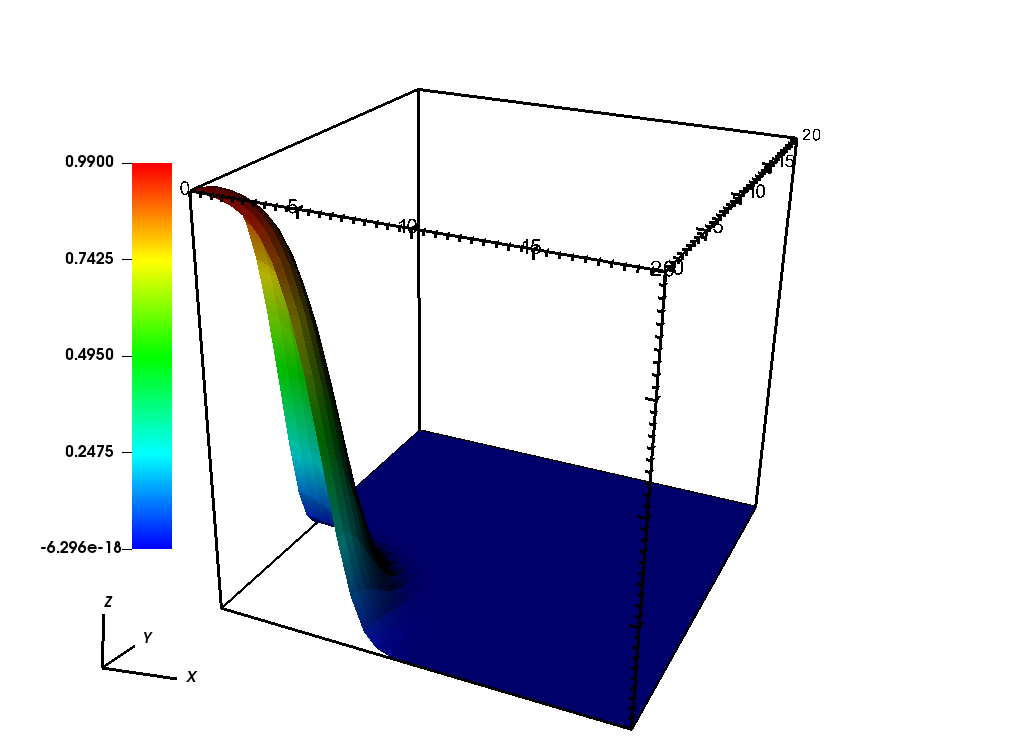}
		\caption{t = 15}
	\end{subfigure}
	\begin{subfigure}{0.24\textwidth}
		\includegraphics[width=\textwidth]{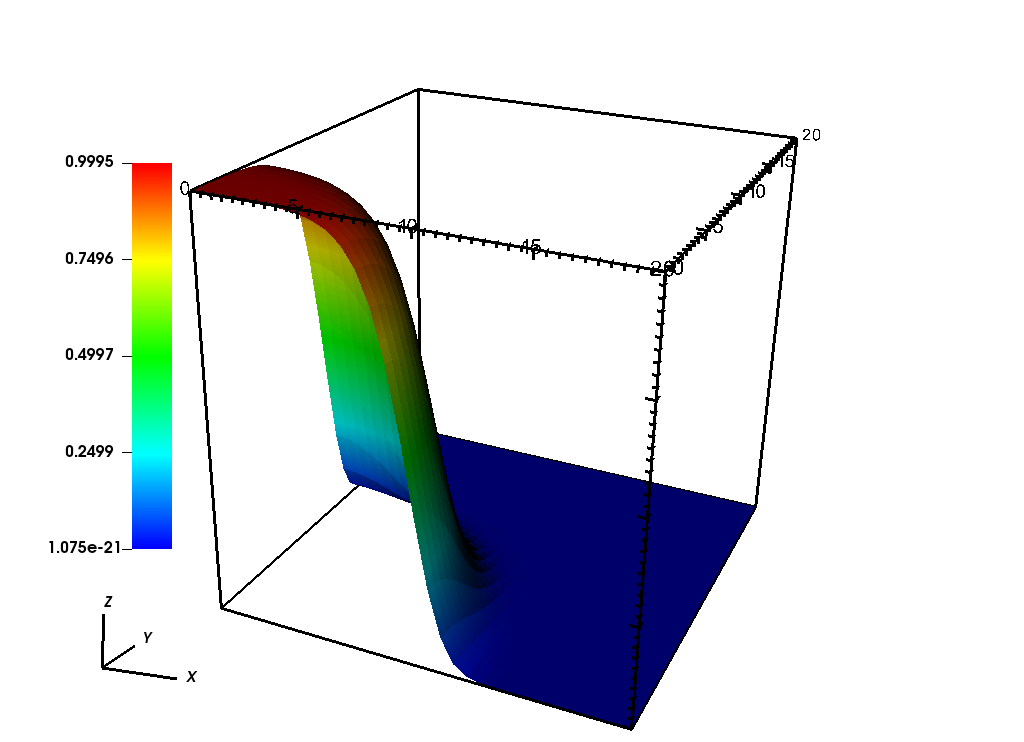}
		\caption{t = 25}
	\end{subfigure}
	\begin{subfigure}{0.24\textwidth}
		\includegraphics[width=\textwidth]{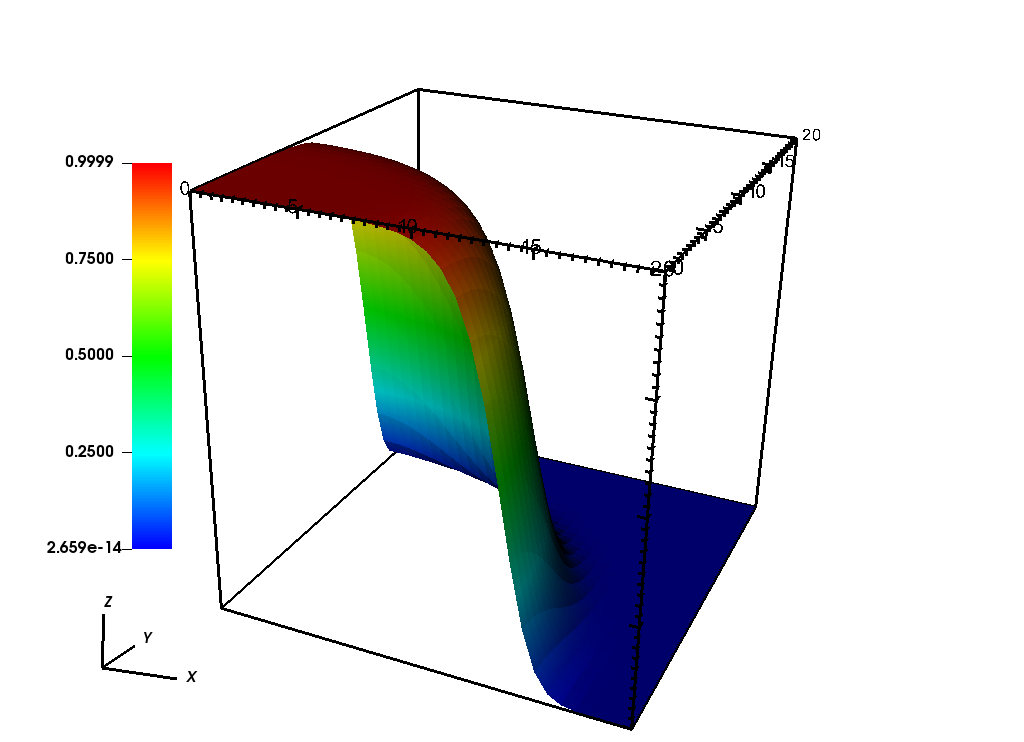}
		\caption{t = 35}
	\end{subfigure}
		\caption{
	The snapshots of cancer cell invasion $u$ for $\mu = 0.5$.
	 }
	\label{fig4}
\end{figure} 
 
\begin{figure}[H]
	\centering
	\begin{subfigure}{0.24\textwidth}
		\includegraphics[width=\textwidth]{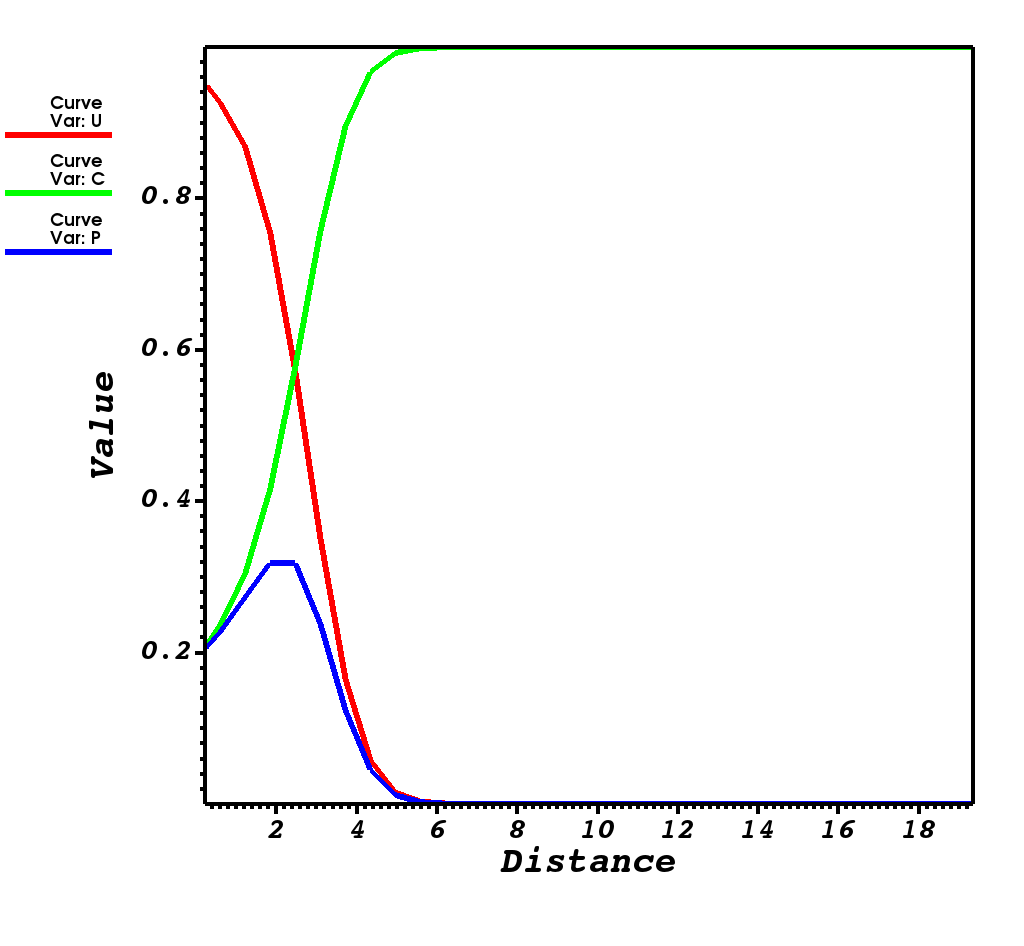}
		\caption{t = 5}
	\end{subfigure}
	\begin{subfigure}{0.24\textwidth}
		\includegraphics[width=\textwidth]{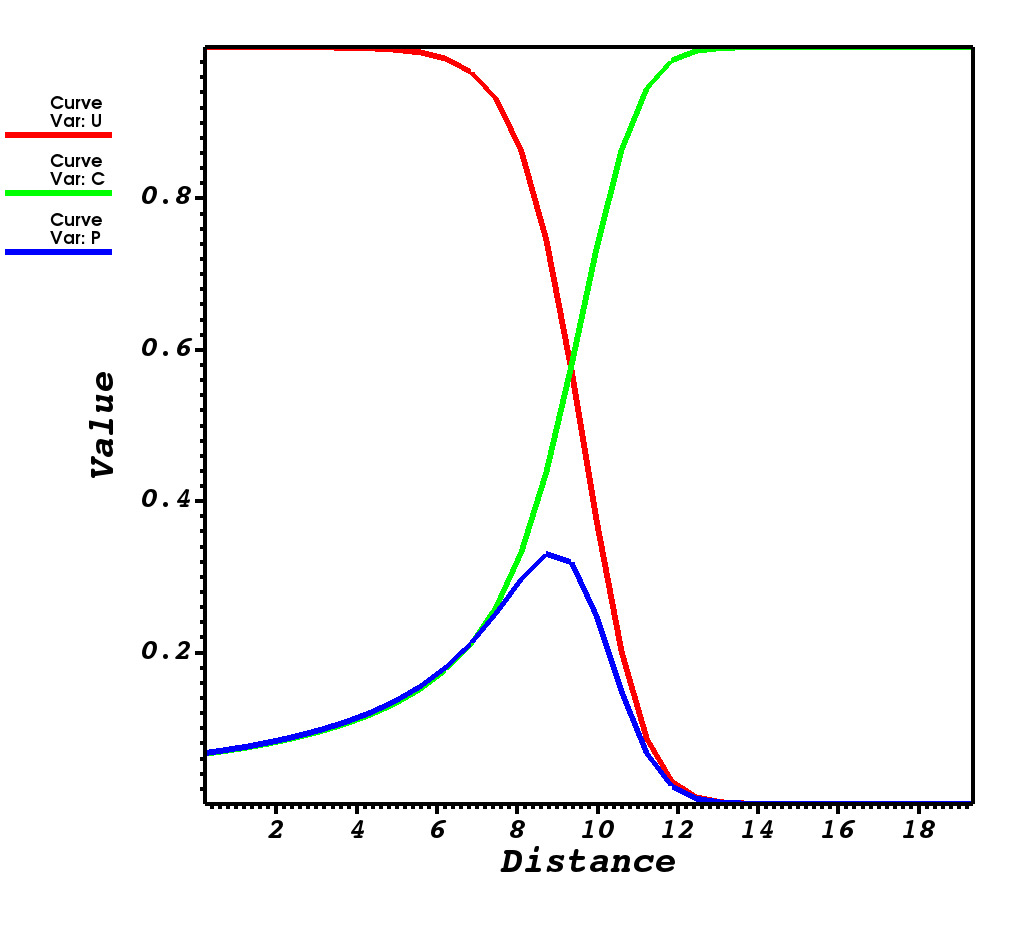}
		\caption{t = 15}
	\end{subfigure}
	\begin{subfigure}{0.24\textwidth}
		\includegraphics[width=\textwidth]{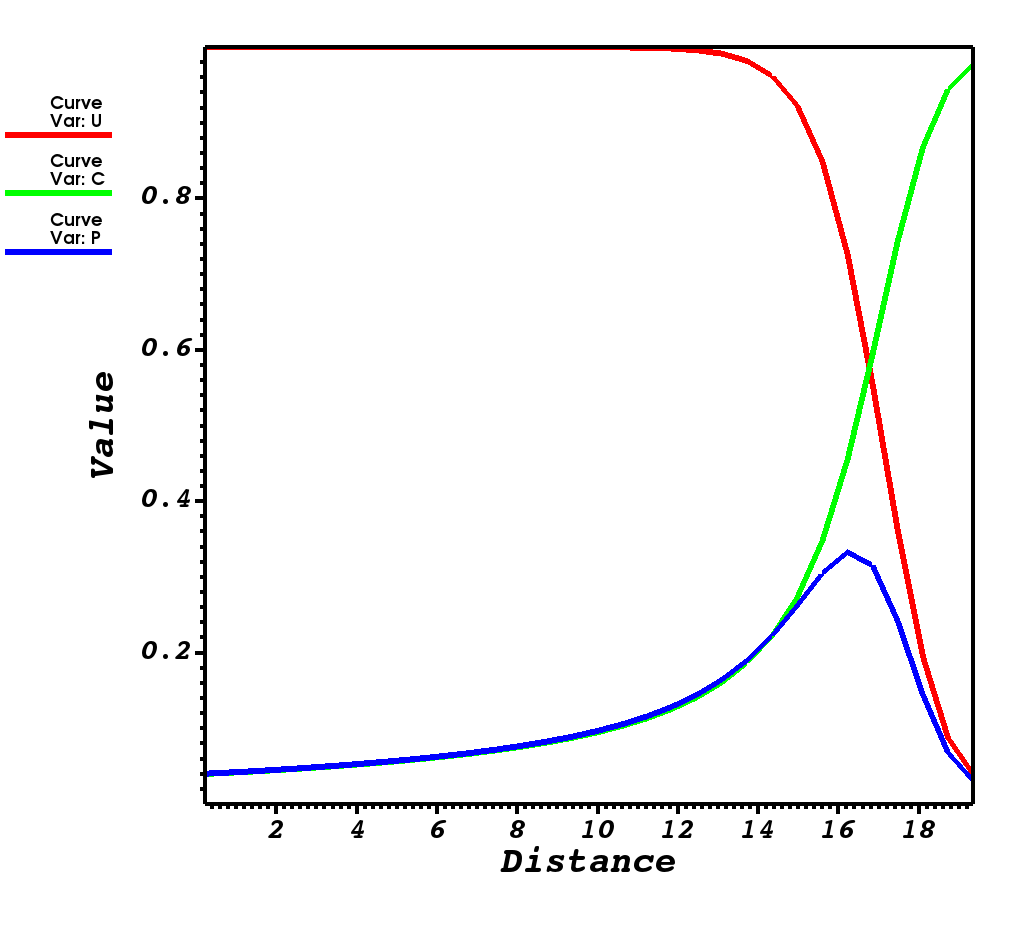}
		\caption{t = 25}
	\end{subfigure}
	\begin{subfigure}{0.24\textwidth}
		\includegraphics[width=\textwidth]{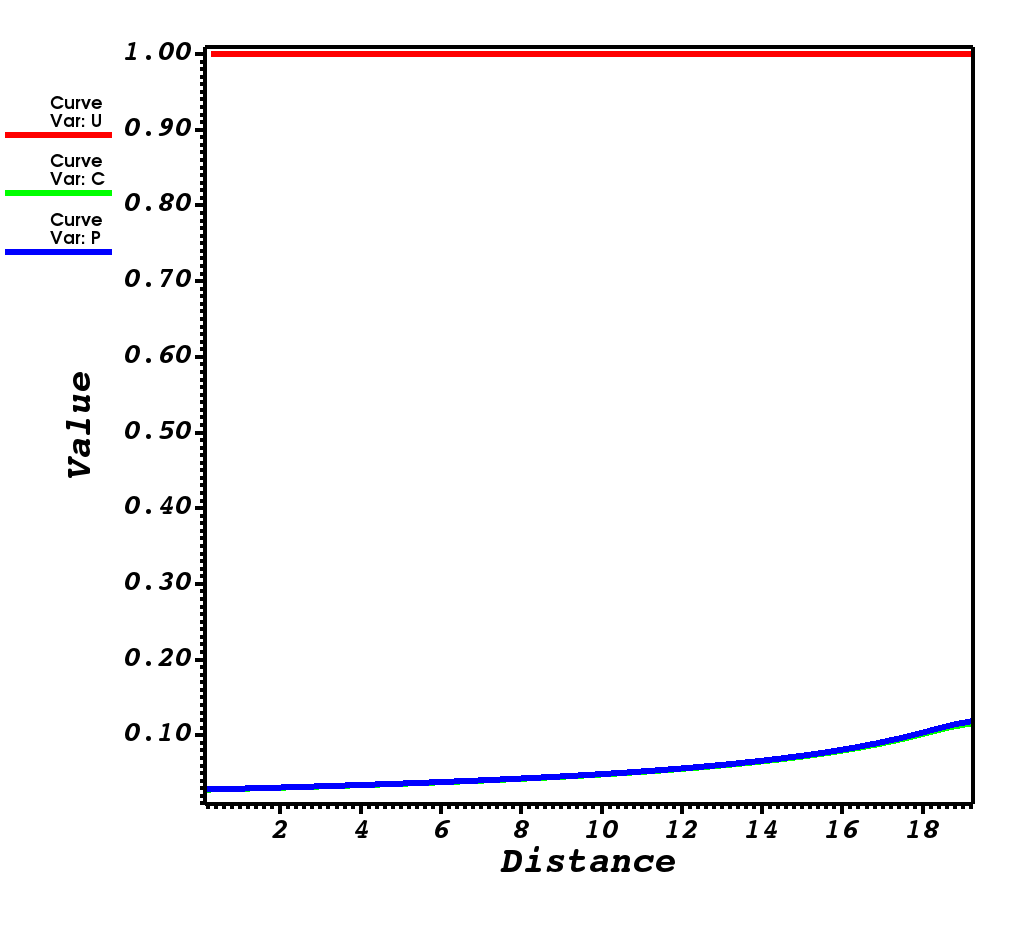}
		\caption{t = 35}
	\end{subfigure}
	\caption{
	The effect of proliferation rate on cancer cell invasion, connective tissue and protease at different time steps, $t=5, 15, 25, 35$ for  $\mu = 1.0$.
	 }
	\label{fig5}
\end{figure} 

\begin{figure}[H]
	\centering
	\begin{subfigure}{0.24\textwidth}
		\includegraphics[width=\textwidth]{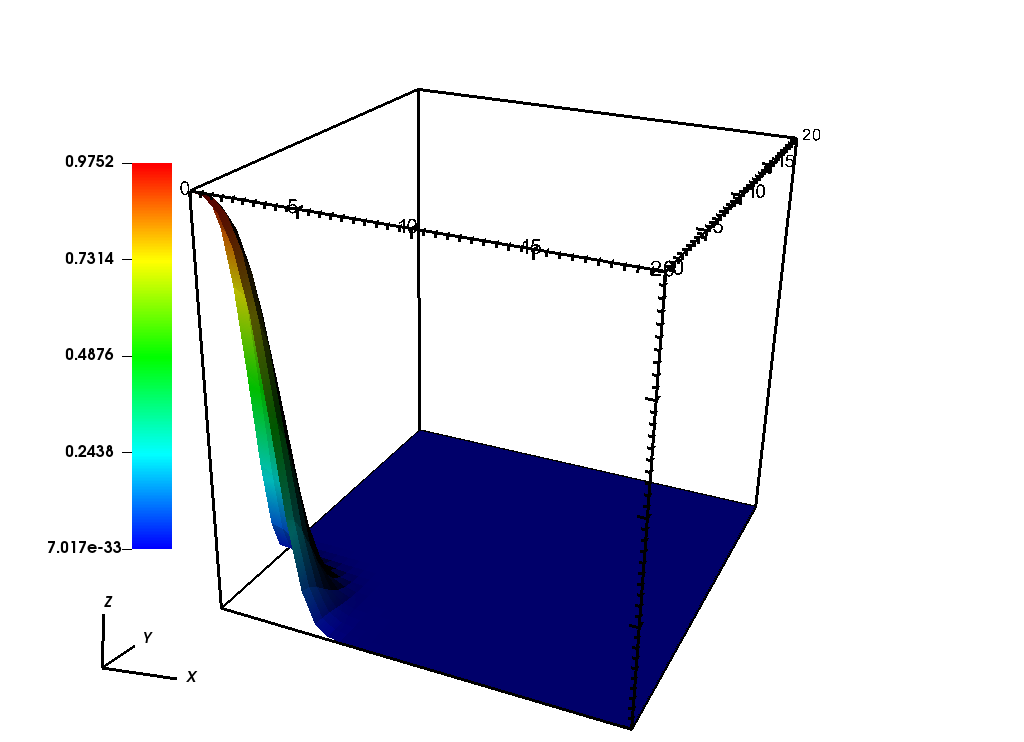}
		\caption{t = 5}
	\end{subfigure}
	\begin{subfigure}{0.24\textwidth}
		\includegraphics[width=\textwidth]{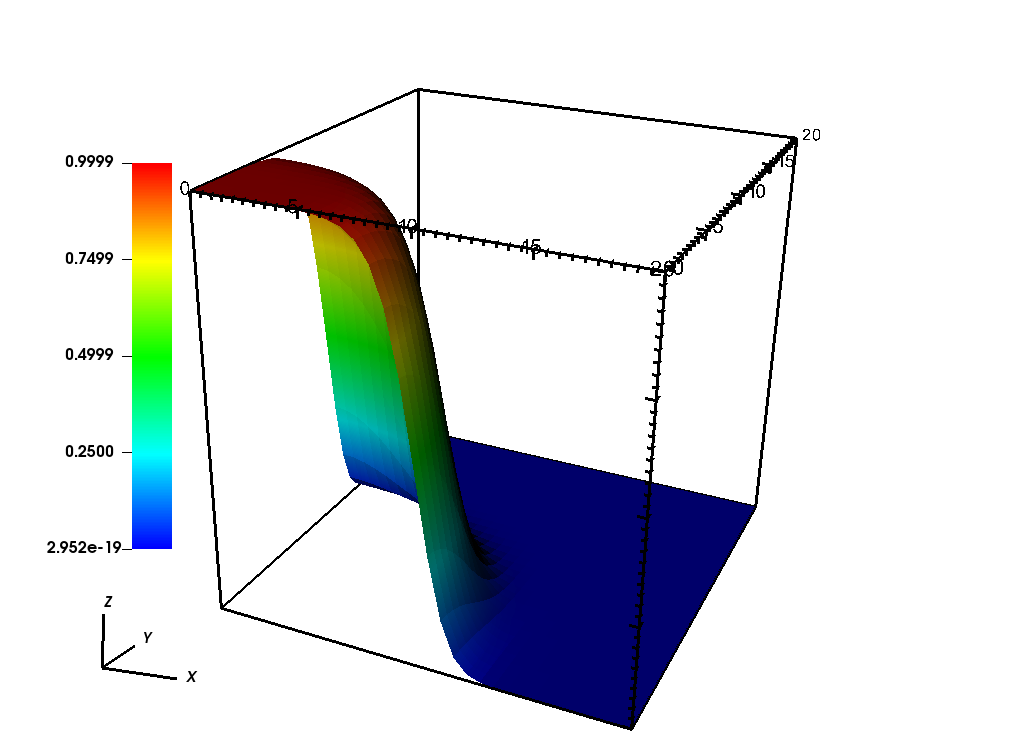}
		\caption{t = 15}
	\end{subfigure}
	\begin{subfigure}{0.24\textwidth}
		\includegraphics[width=\textwidth]{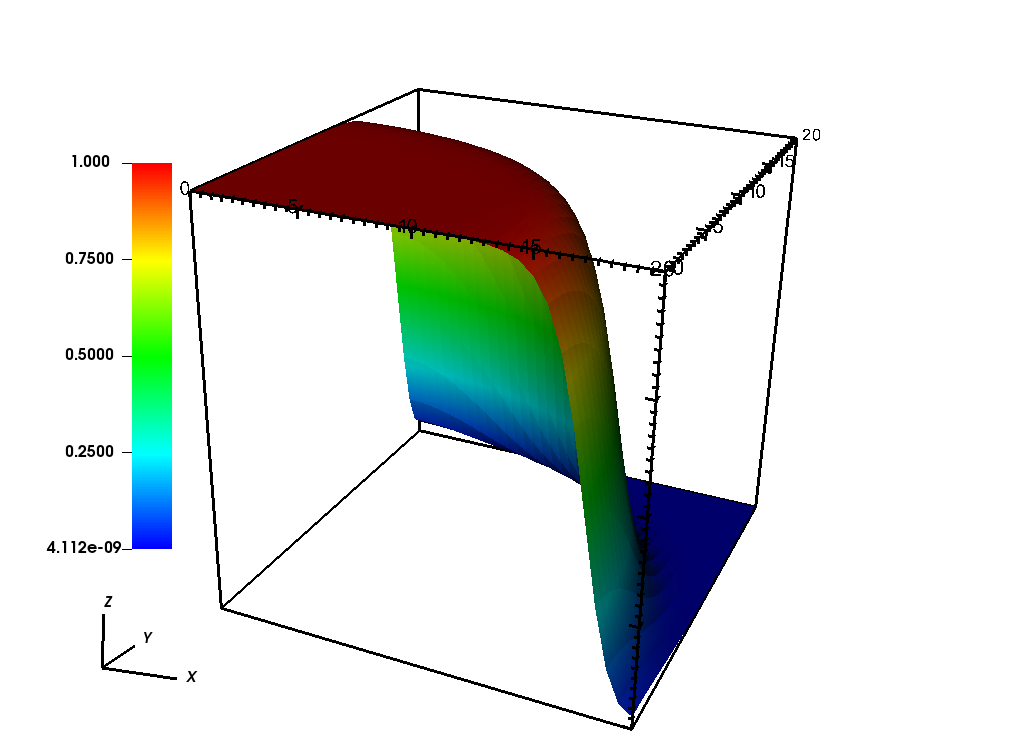}
		\caption{t = 25}
	\end{subfigure}
	\begin{subfigure}{0.24\textwidth}
		\includegraphics[width=\textwidth]{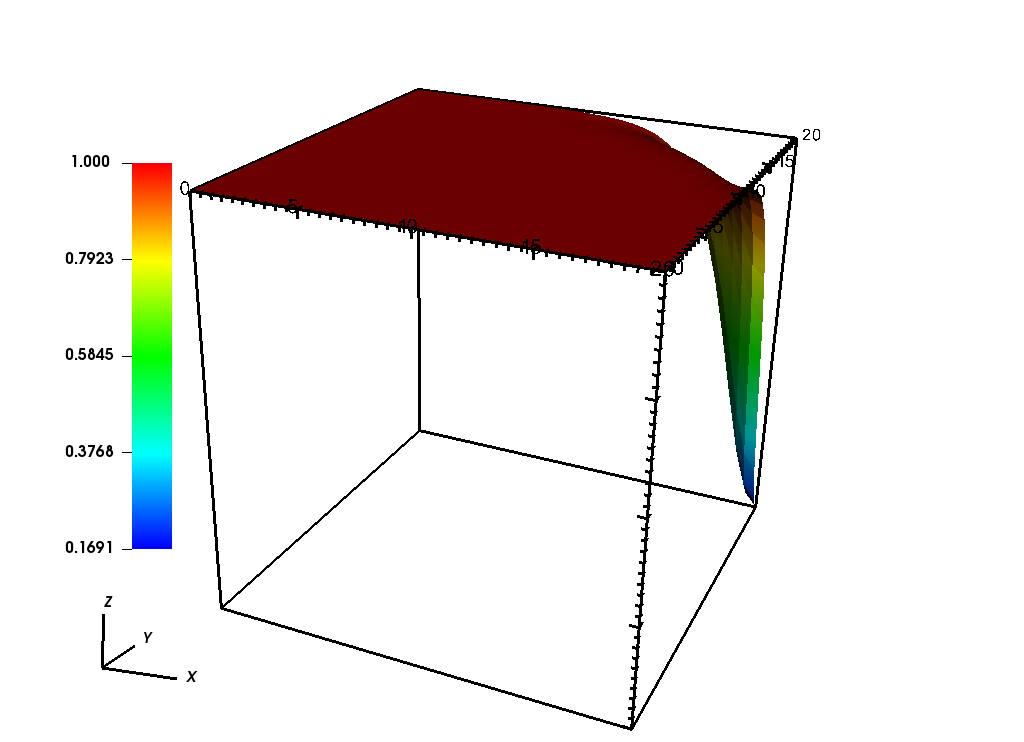}
		\caption{t = 35}
	\end{subfigure}
		\caption{
	The snapshots of cancer cell invasion $u$ for $\mu = 1.0$.
	 }
	\label{fig6}
\end{figure}

\subsection{Effects of the haptotactic coefficient \tops{$\chi$}{chi}}
In this subsection, we consider the effect of the haptotactic coefficient on
the connective cells degeneration by varying $\chi$. We choose $\chi = 0.25,
0.75, 1.25$ with small proliferation rate $\mu = 0.01$, diffusion coefficient
$\alpha^{-1} = 0.1$ and $\varepsilon = 0.2$. The effects of haptotactic
coefficient at different time instances are depicted in
Figs.~\ref{fig7}--\ref{fig12}. Starting with $\chi = 0.25$, the snapshot at
$t=5$ shows that the cancer cells reduce at the origin and start migrating
towards the direction of the gradient of connective tissue. The migration of
the cancer cells becomes more clear and the effect of haptotaxis can be clearly
seen at $t=5$ in Fig.~\ref{fig9} and Fig.~\ref{fig10}, where the small
cluster of cancer cells is created and becomes larger by time. Increasing the
amount of $\chi$ accelerates the cancer cells migration and the cancer cells
should move toward the boundary of the domain quickly, but as we can see from
Fig.~\ref{fig11} and Fig.~\ref{fig12}, oscillations start at $t=5$ and the numerical simulation breaks down for $\chi = 1.25$.\\
 
\begin{figure}[H]
	\centering
	\begin{subfigure}{0.23\textwidth}
		\includegraphics[width=\textwidth]{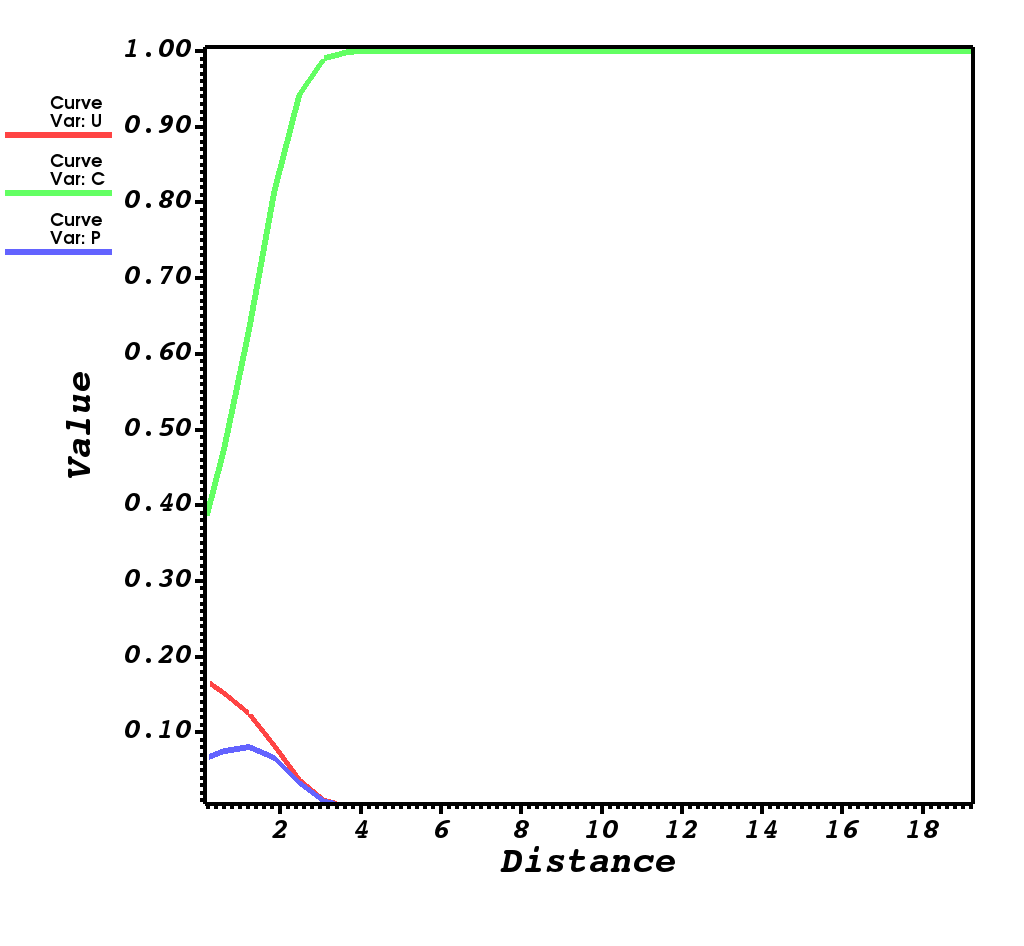}
		\caption{t = 5}
	\end{subfigure}
	\begin{subfigure}{0.23\textwidth}
		\includegraphics[width=\textwidth]{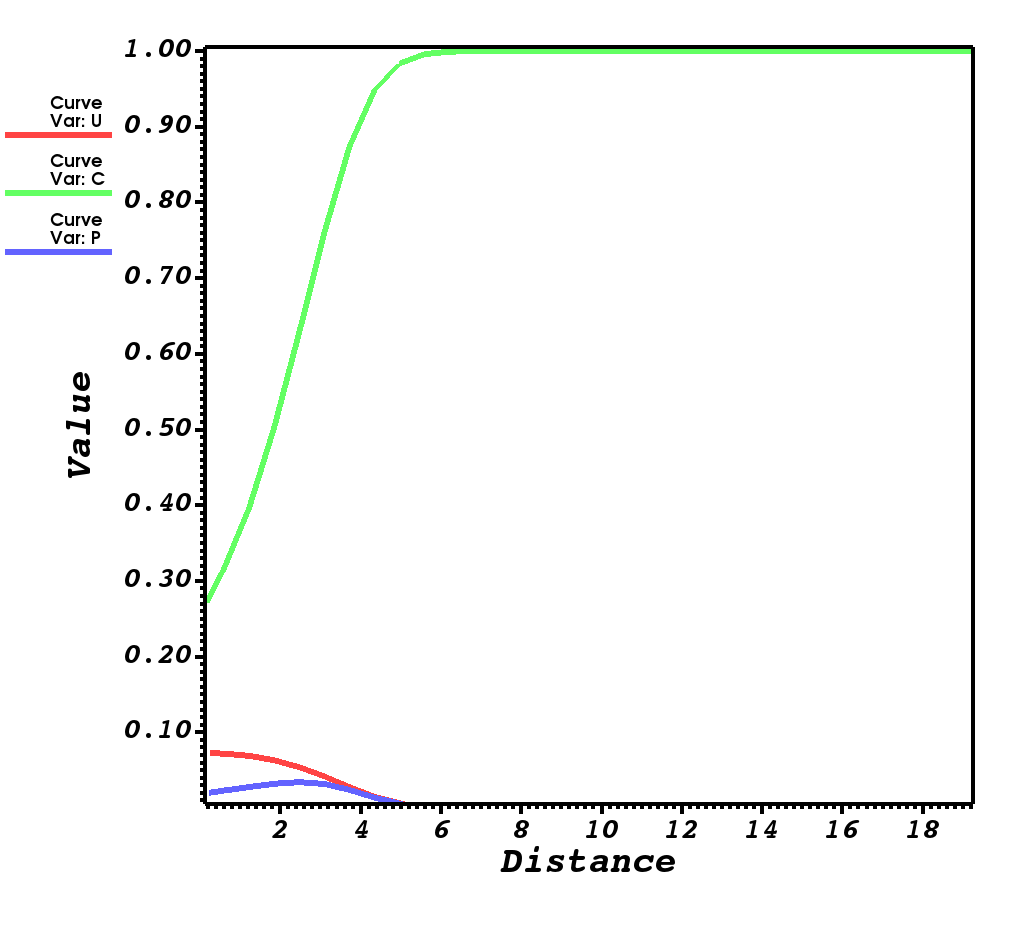}
		\caption{t = 15}
	\end{subfigure}
	\begin{subfigure}{0.23\textwidth}
		\includegraphics[width=\textwidth]{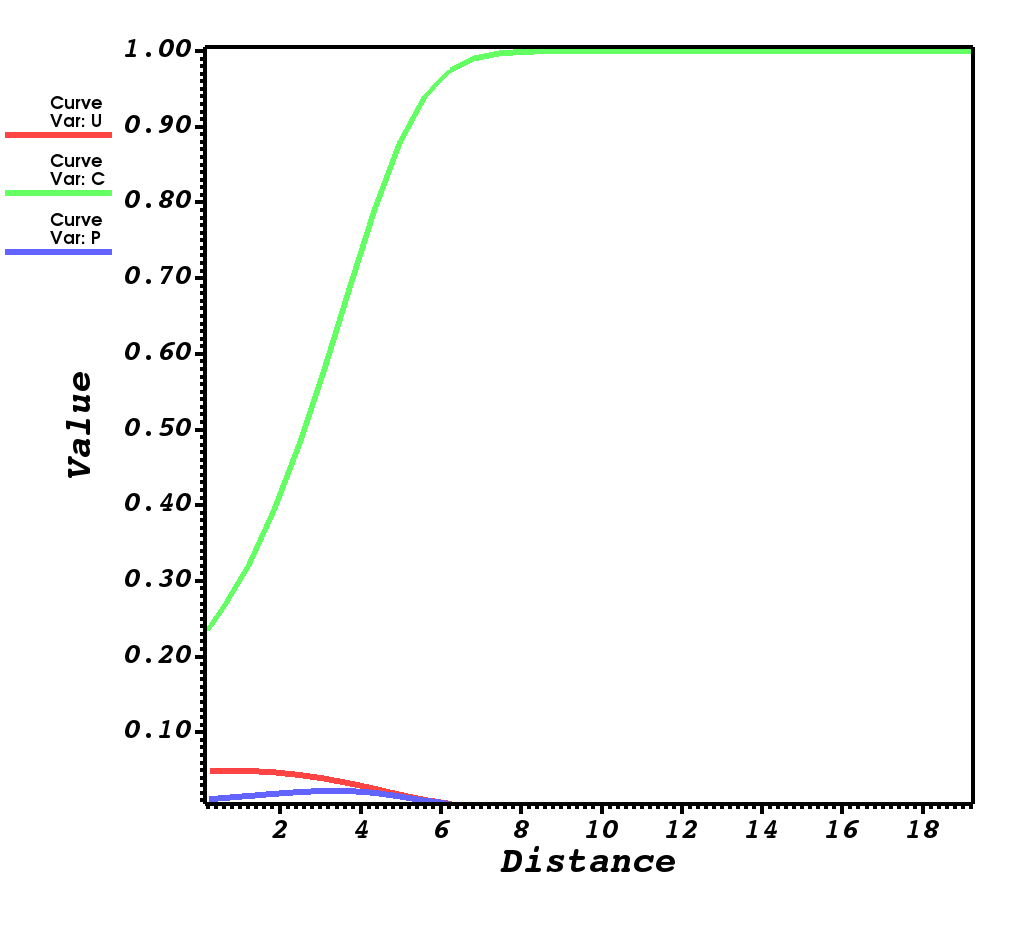}
		\caption{t = 25}
	\end{subfigure}
	\begin{subfigure}{0.23\textwidth}
		\includegraphics[width=\textwidth]{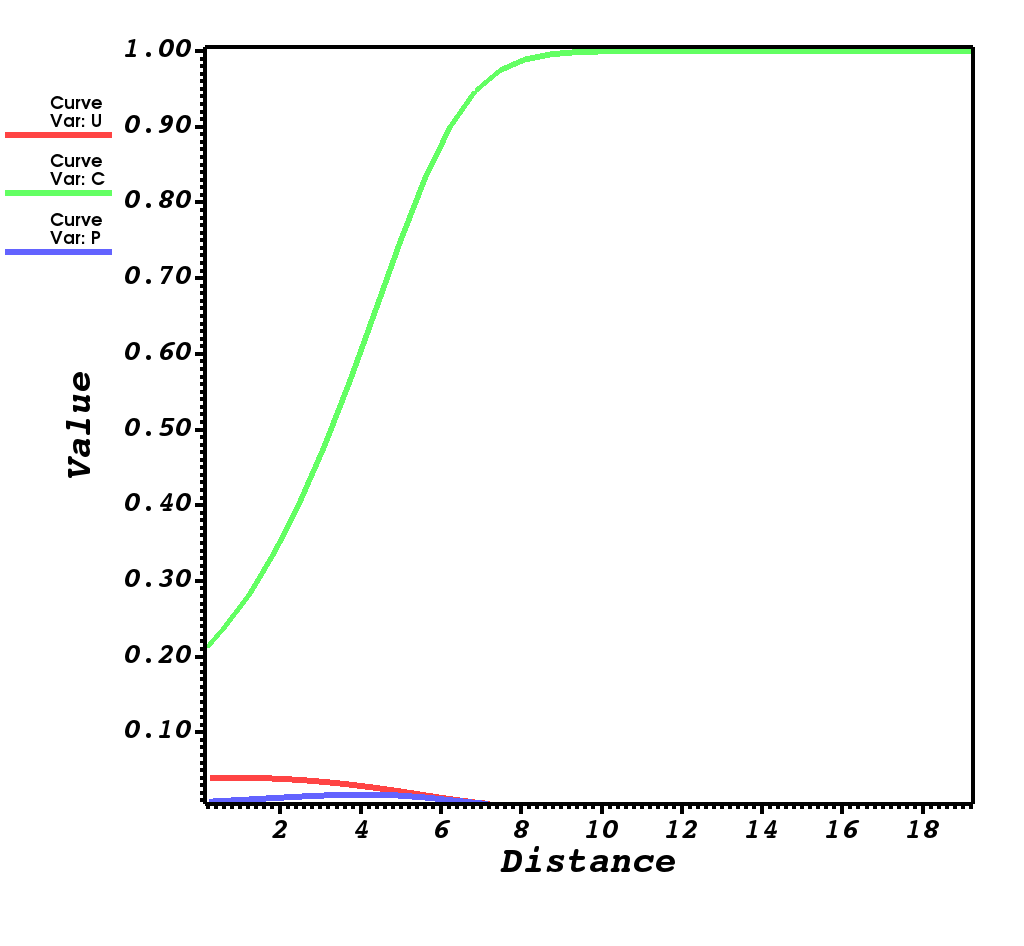}
		\caption{t = 35}
	\end{subfigure}
	\caption{Haptotactic effect on cancer cell invasion, connective tissue and protease at different time steps, $t=5, 15, 25, 35$ for $ \chi =0.25 $.}
	\label{fig7}
\end{figure} 

\begin{figure}[H]
	\centering
	\begin{subfigure}{0.23\textwidth}
		\includegraphics[width=\textwidth]{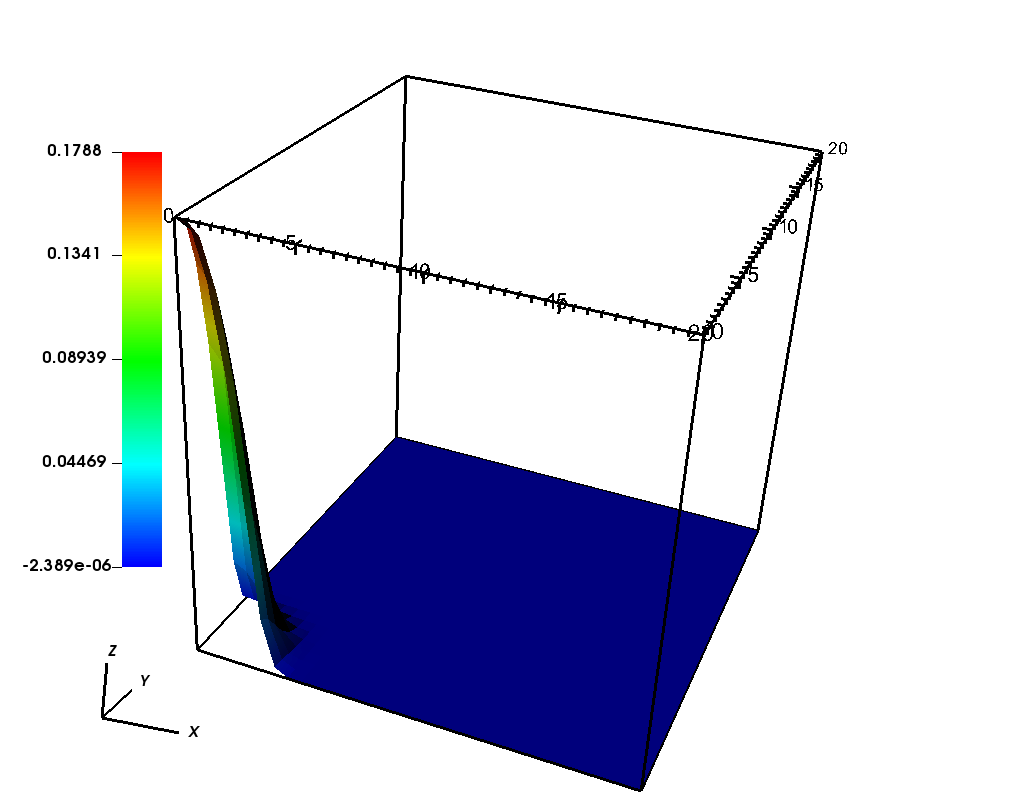}
		\caption{t = 5}
	\end{subfigure}
	\begin{subfigure}{0.23\textwidth}
		\includegraphics[width=\textwidth]{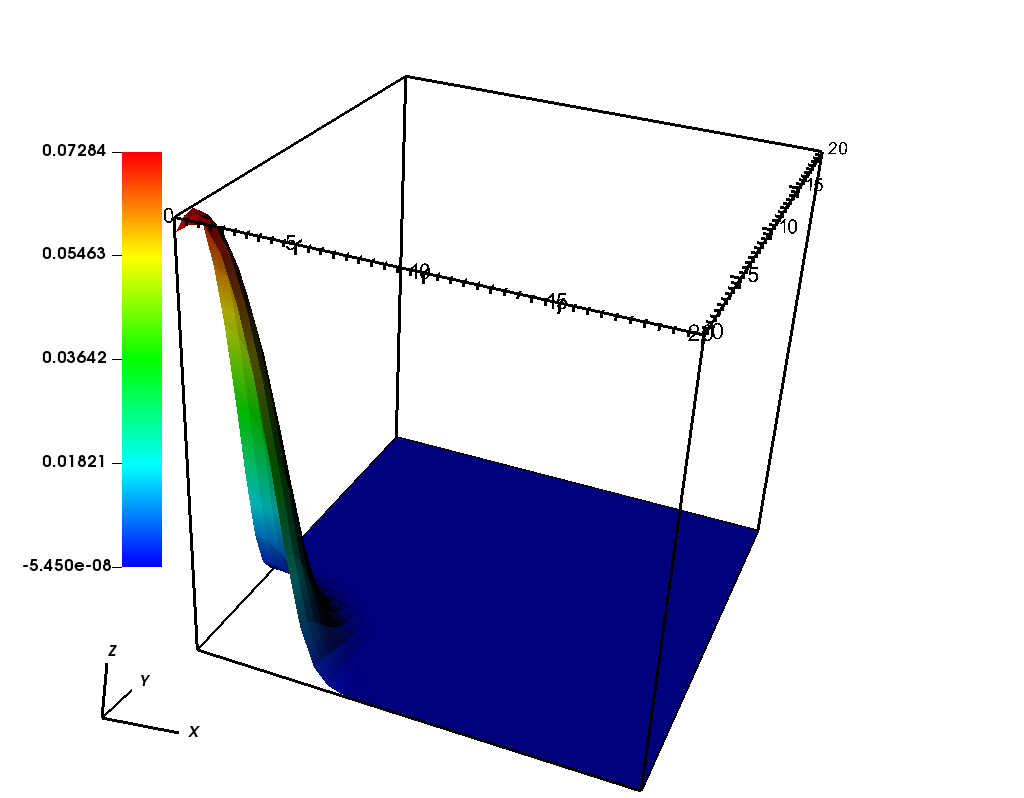}
		\caption{t = 15}
	\end{subfigure}
	\begin{subfigure}{0.23\textwidth}
		\includegraphics[width=\textwidth]{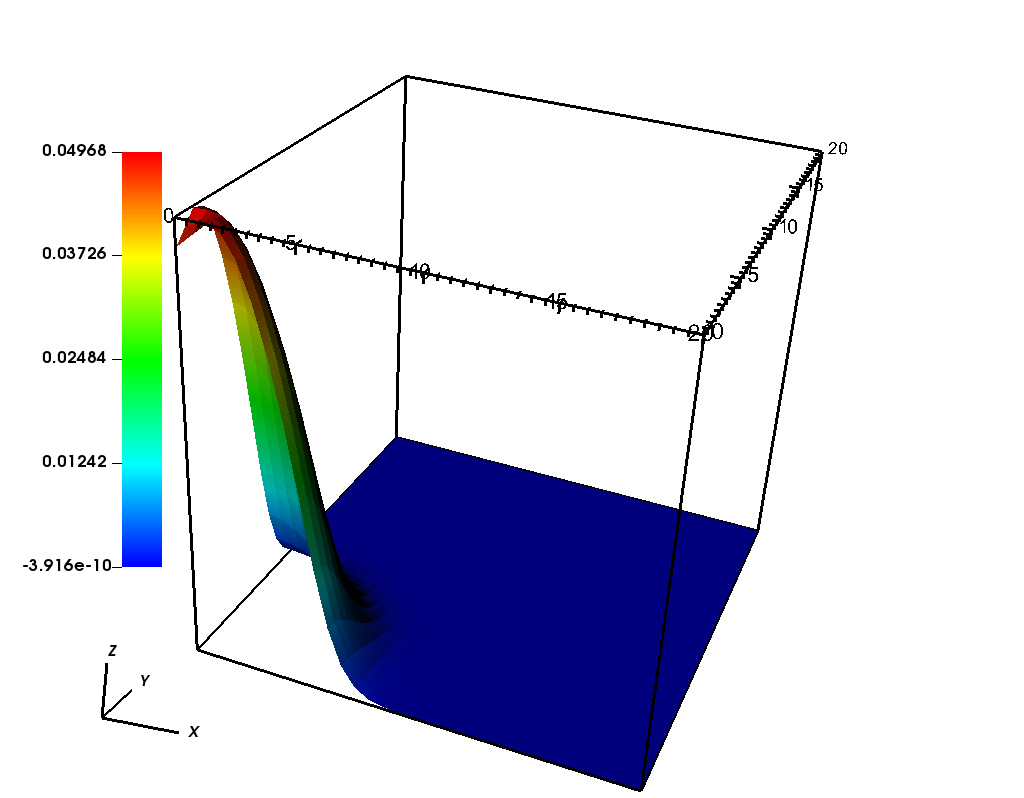}
		\caption{t = 25}
	\end{subfigure}
	\begin{subfigure}{0.23\textwidth}
		\includegraphics[width=\textwidth]{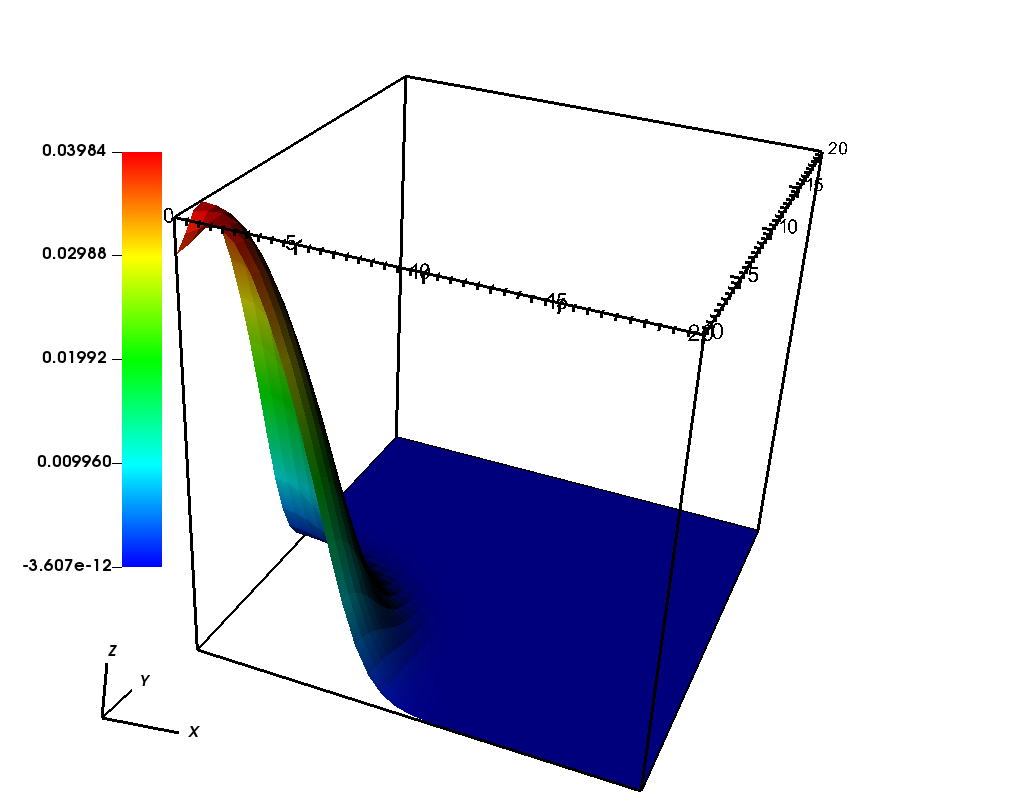}
		\caption{t = 35}
	\end{subfigure}
	\caption{The snapshots of cancer cell invasion $u$ for $\chi = 0.25$, the maximum amount of cancer cells decreasing from left to right is 0.1788, 0.07284, 0.04968, and 0.03984.}
	\label{fig8}
\end{figure}

\begin{figure}[H]
	\centering
	\begin{subfigure}{0.23\textwidth}
		\includegraphics[width=\textwidth]{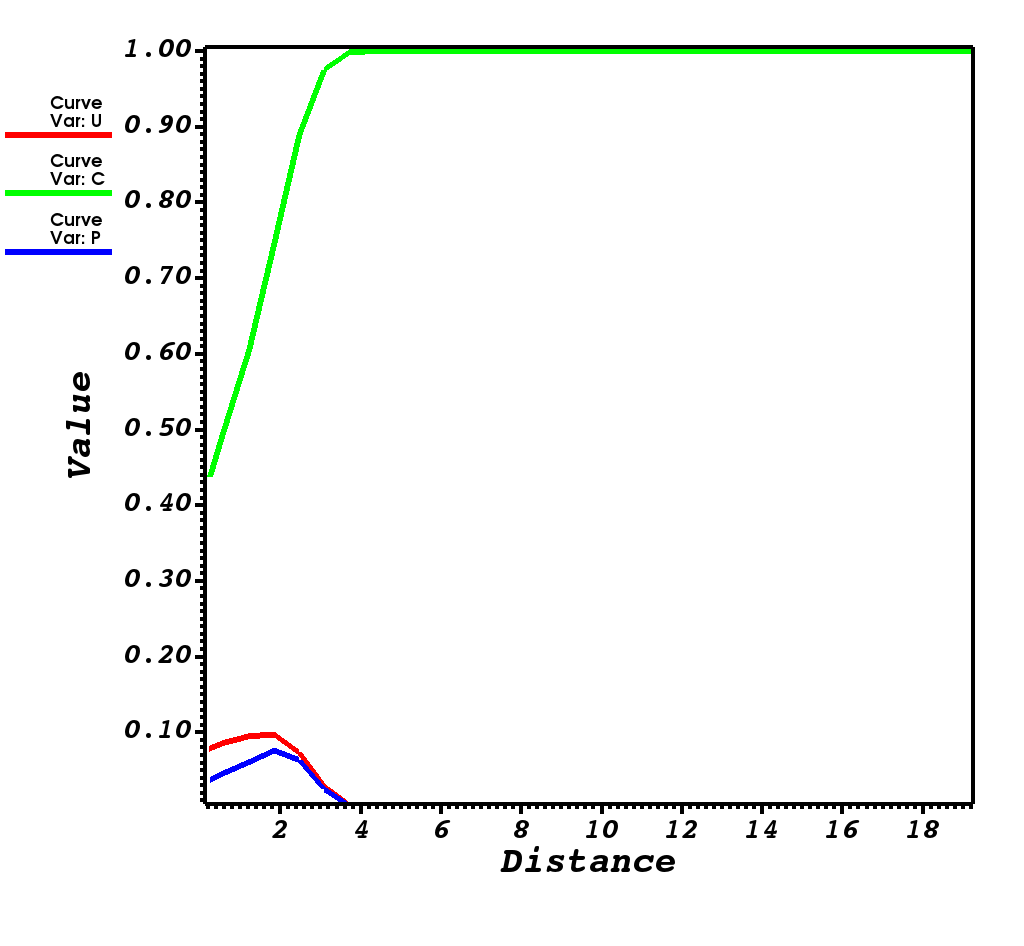}
		\caption{t = 5}
	\end{subfigure}
	\begin{subfigure}{0.23\textwidth}
		\includegraphics[width=\textwidth]{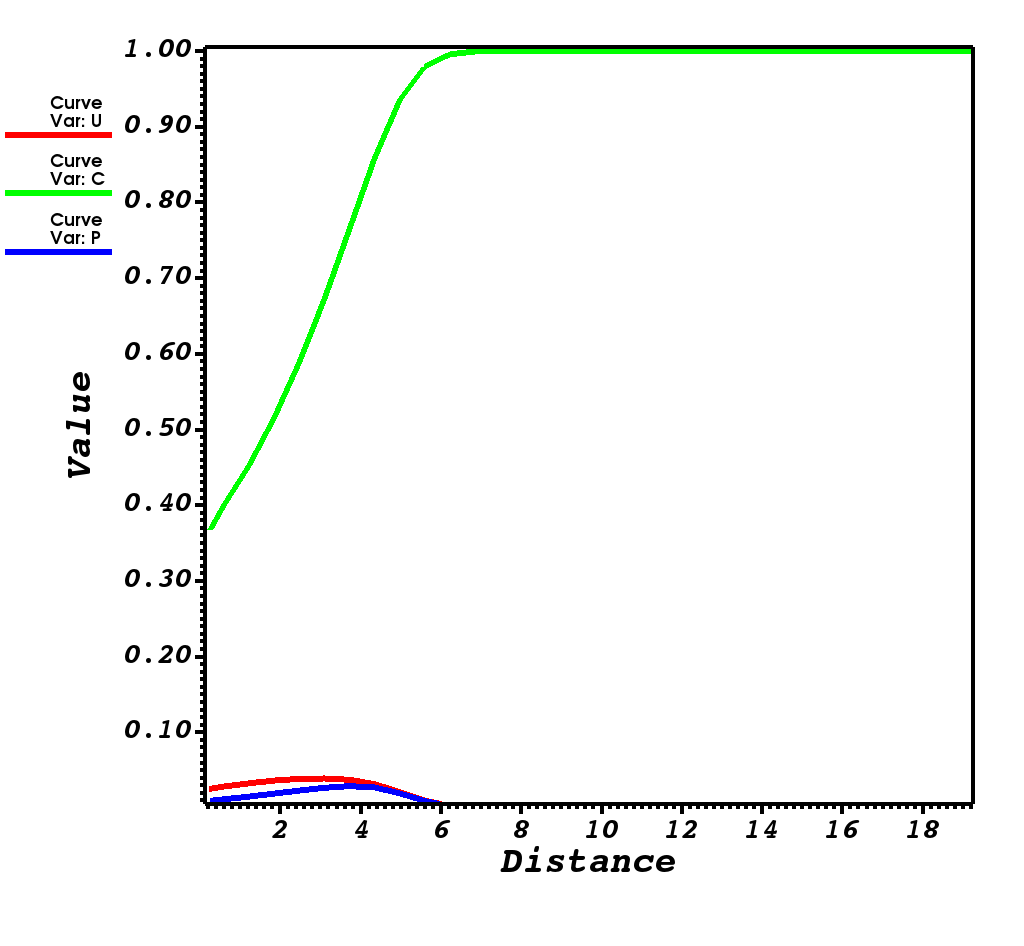}
		\caption{t = 15}
	\end{subfigure}
	\begin{subfigure}{0.23\textwidth}
		\includegraphics[width=\textwidth]{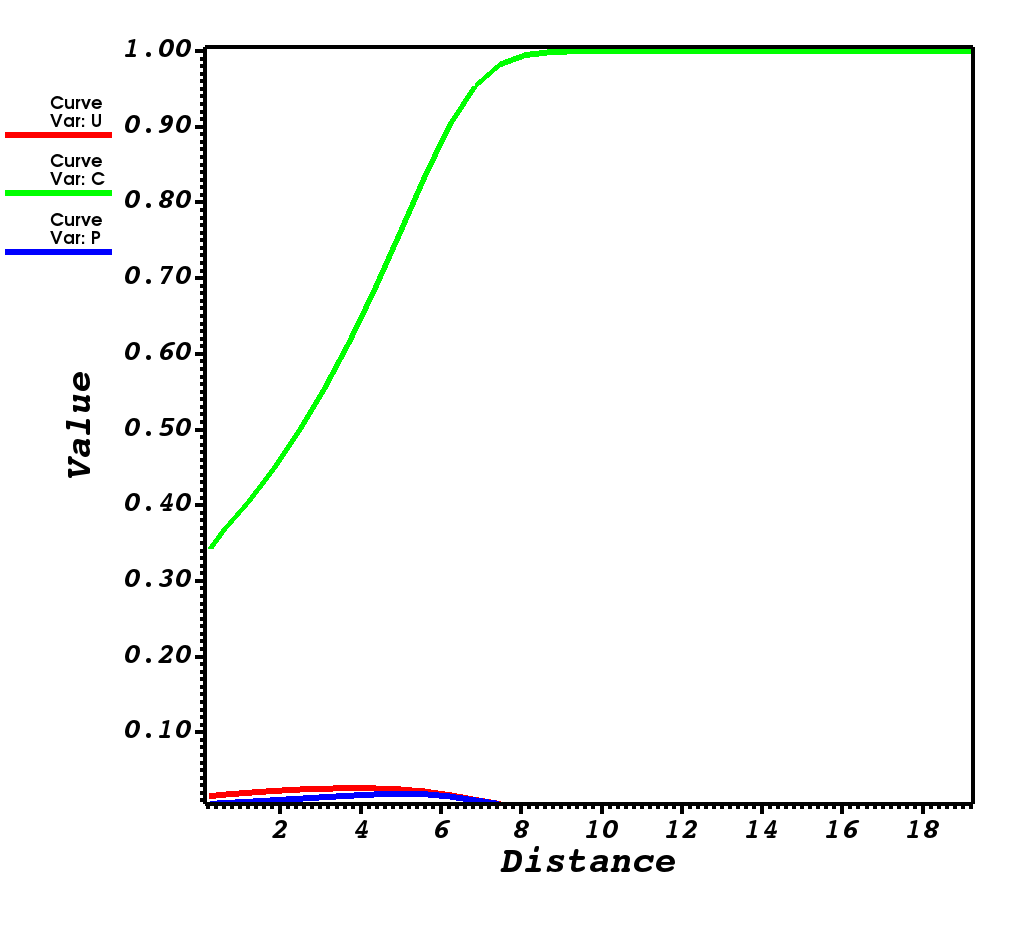}
		\caption{t = 25}
	\end{subfigure}
	\begin{subfigure}{0.23\textwidth}
		\includegraphics[width=\textwidth]{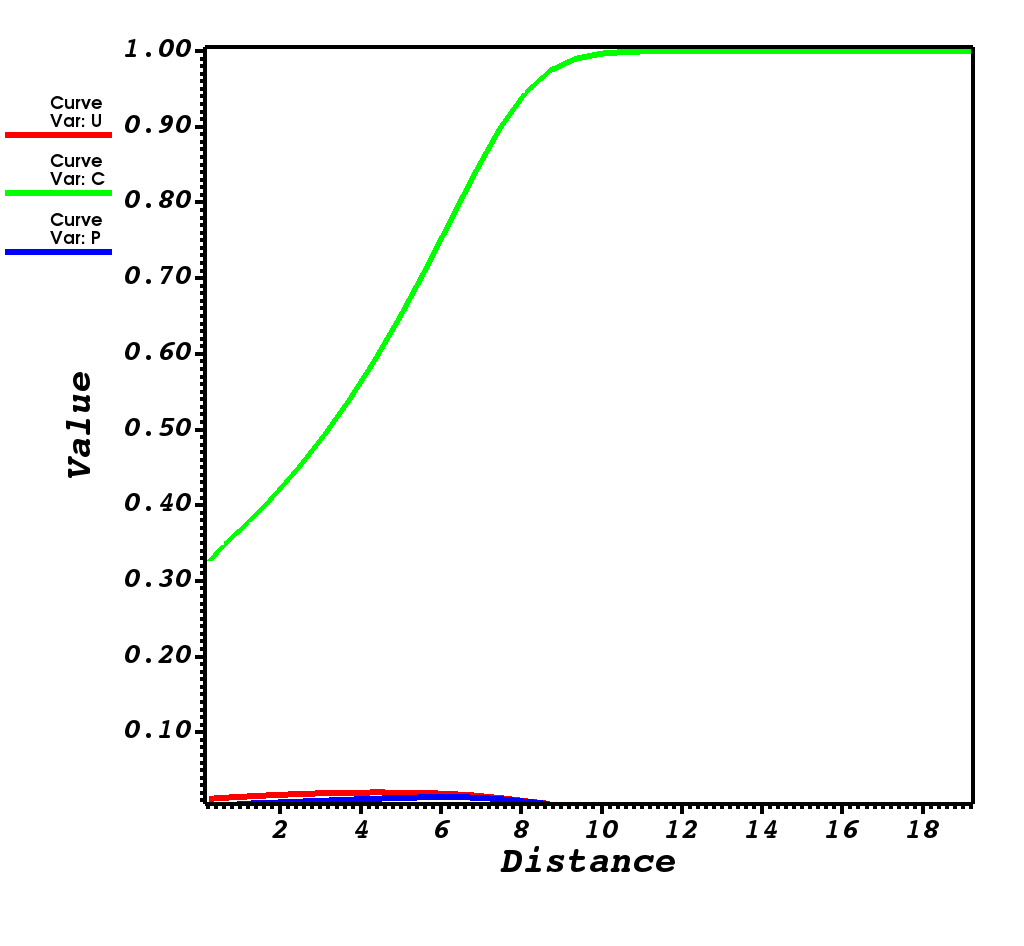}
		\caption{t = 35}
	\end{subfigure}
	\caption{Haptotactic effect on cancer cell invasion, connective tissue and protease at different time steps, $t=5, 15, 25, 35$ for $ \chi =0.75 $.}
	\label{fig9}
\end{figure} 

\begin{figure}[H]
	\centering
	\begin{subfigure}{0.23\textwidth}
		\includegraphics[width=\textwidth]{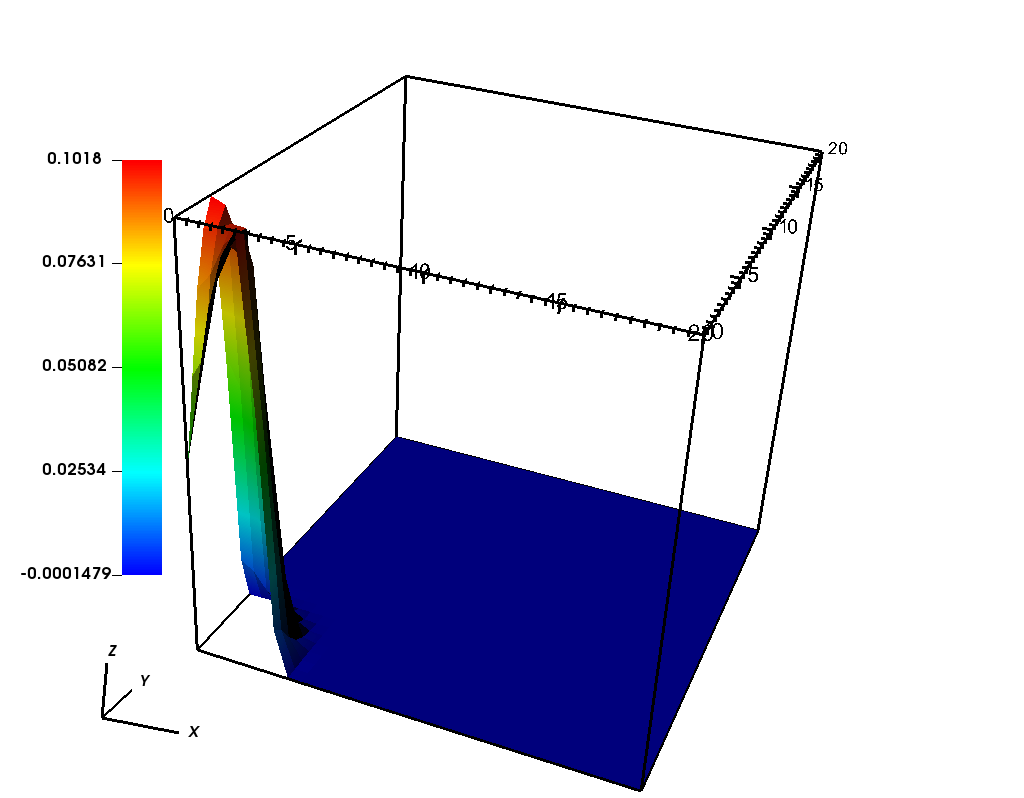}
		\caption{t = 5}
	\end{subfigure}
	\begin{subfigure}{0.23\textwidth}
		\includegraphics[width=\textwidth]{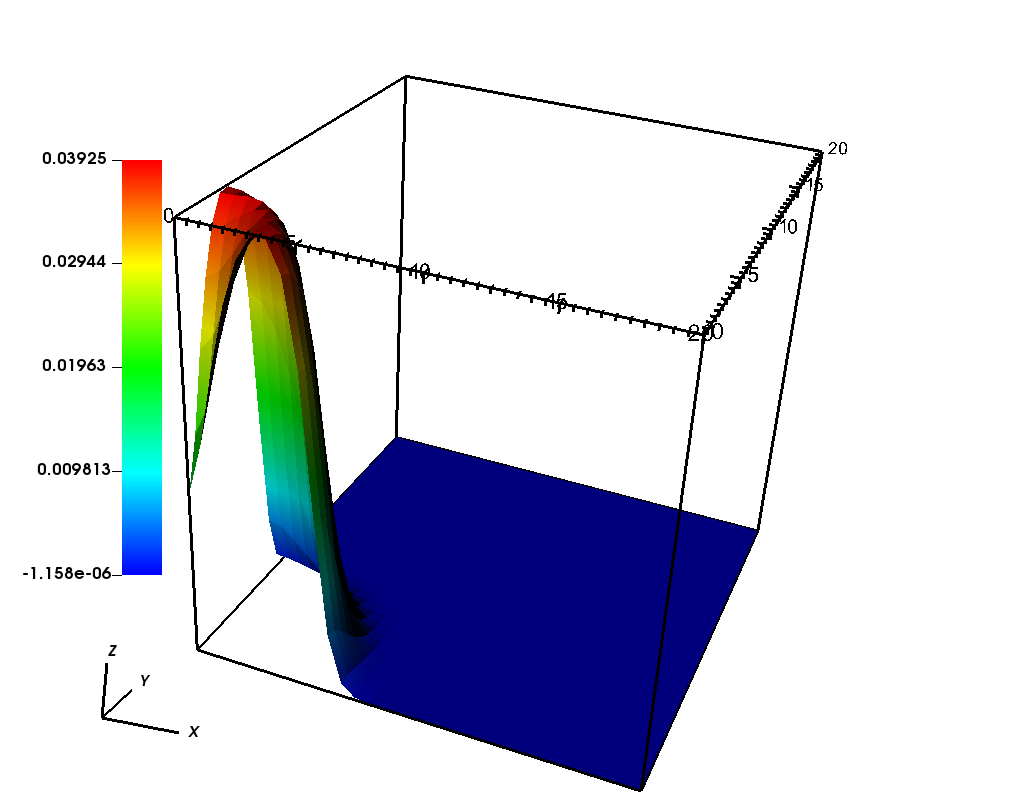}
		\caption{t = 15}
	\end{subfigure}
	\begin{subfigure}{0.23\textwidth}
		\includegraphics[width=\textwidth]{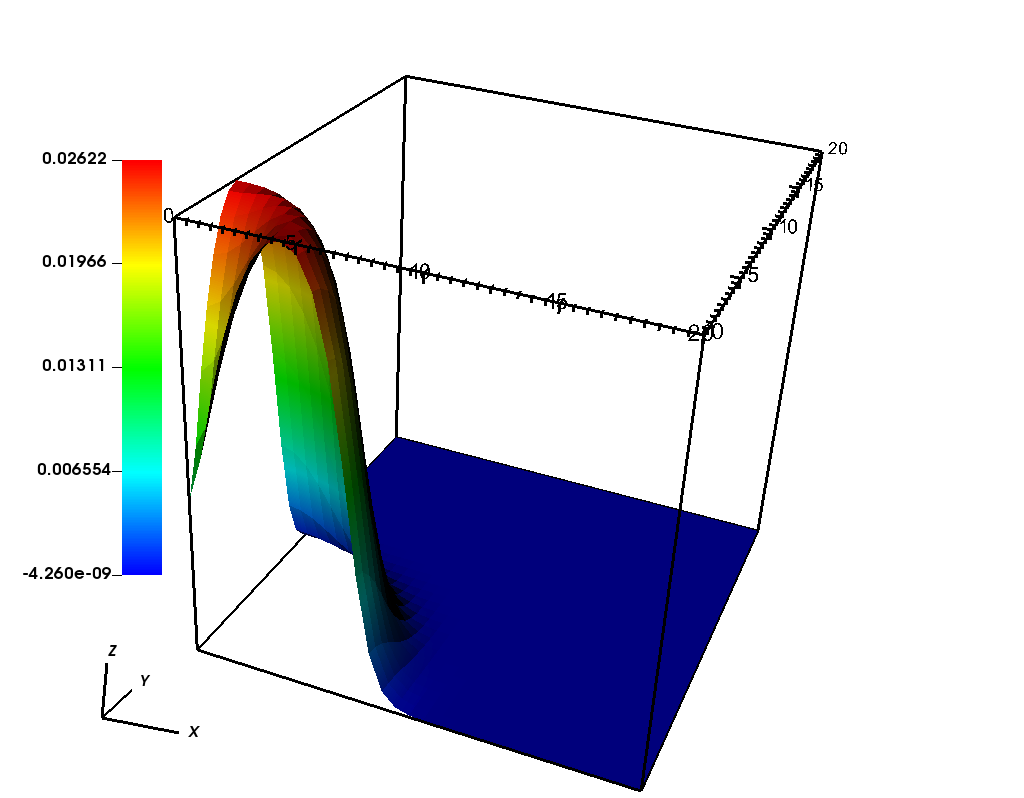}
		\caption{t = 25}
	\end{subfigure}
	\begin{subfigure}{0.23\textwidth}
		\includegraphics[width=\textwidth]{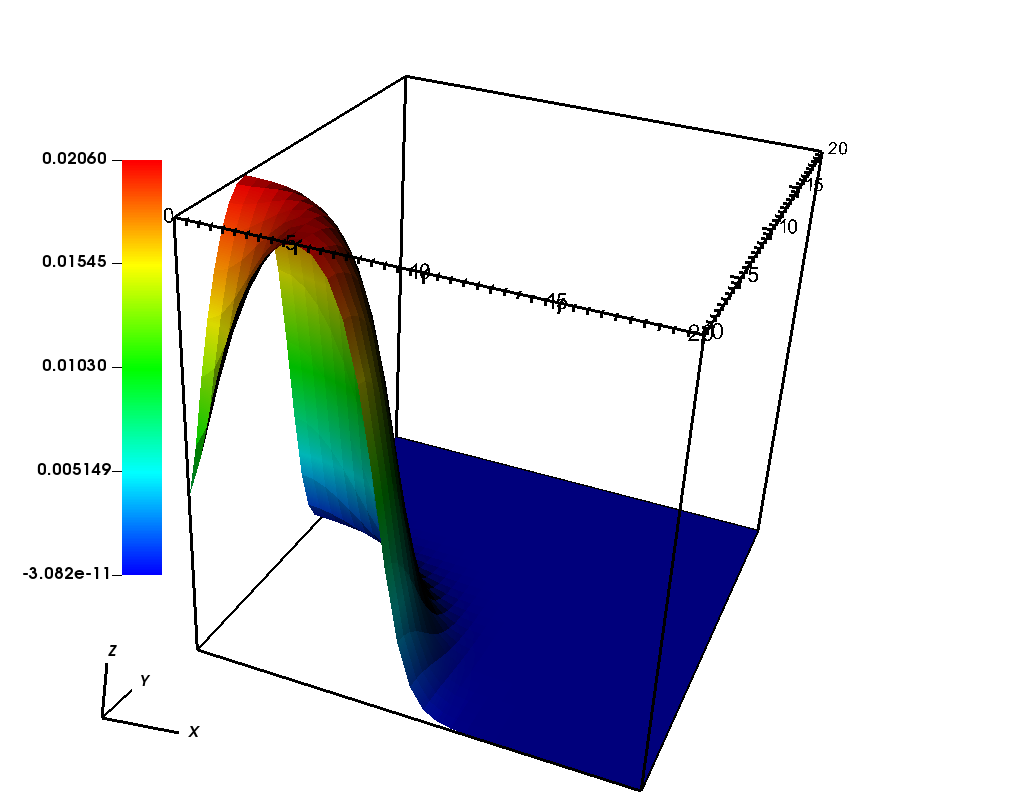}
		\caption{t = 35}
	\end{subfigure}
	\caption{The snapshots of cancer cell invasion $u$ for $\chi = 0.75$, the maximum amount of cancer cells decreasing from left to right is 0.1018, 0.03925, 0.02622, and 0.02060.}
	\label{fig10}
\end{figure} 

\begin{figure}[H]
	\centering
	\begin{subfigure}{0.23\textwidth}
		\includegraphics[width=\textwidth]{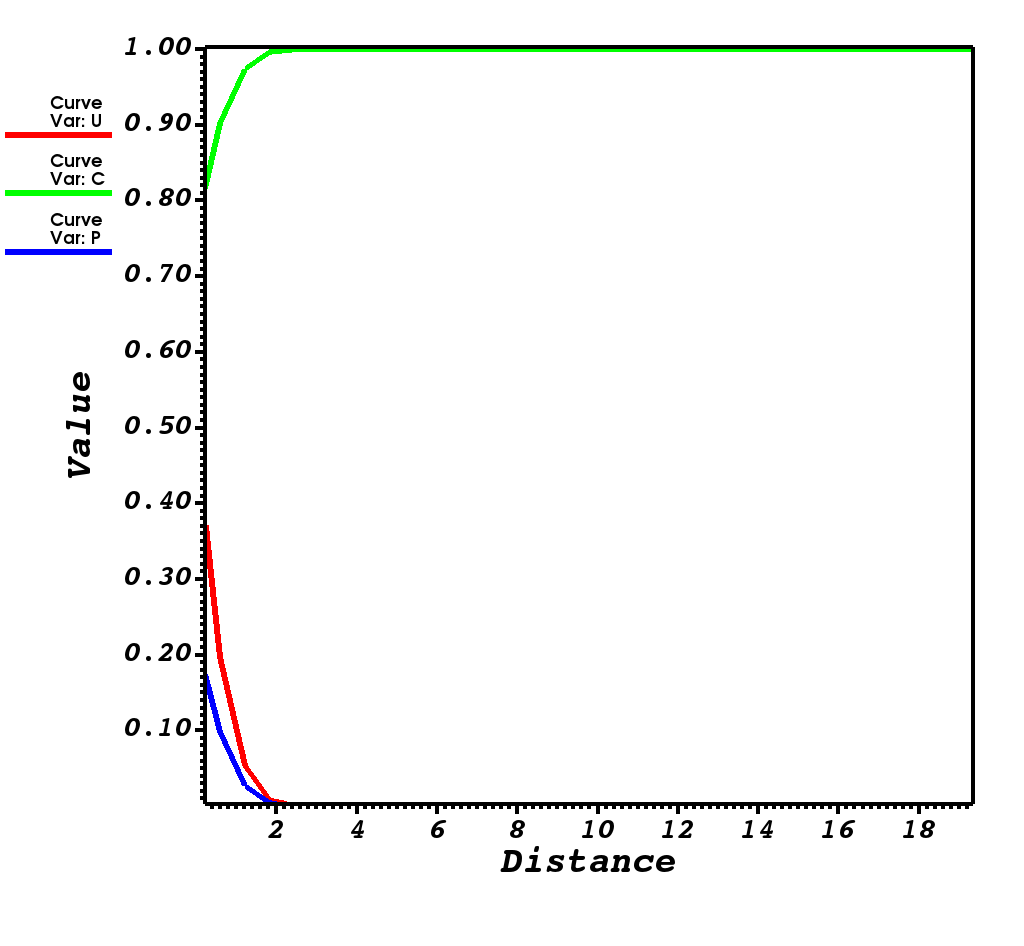}
		\caption{t = 0}
	\end{subfigure}
	\begin{subfigure}{0.23\textwidth}
		\includegraphics[width=\textwidth]{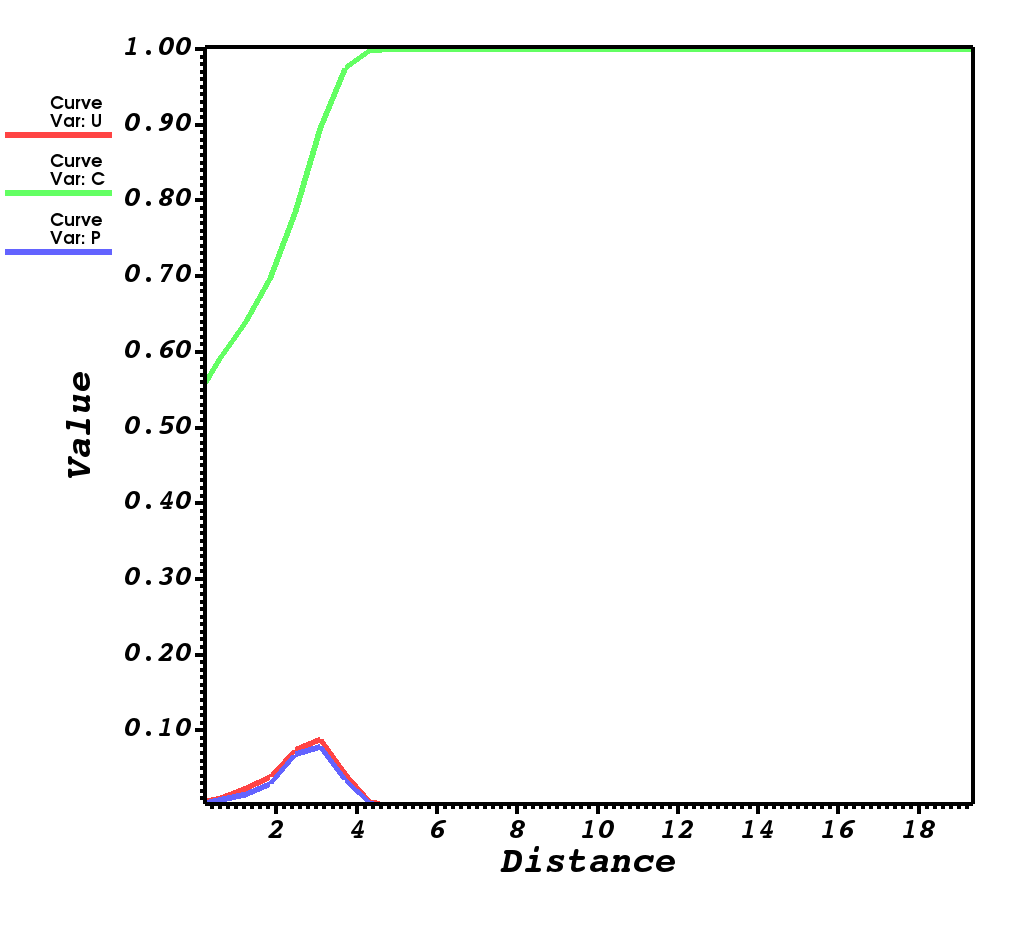}
		\caption{t = 5}
	\end{subfigure}
	\caption{Haptotactic effect on cancer cell invasion, connective tissue and protease at different time steps, $t=0, 5$ for $ \chi =1.25 $.}
	\label{fig11}
\end{figure} 

\begin{figure}[!htb]
	\centering
	\begin{subfigure}{0.23\textwidth}
		\includegraphics[width=\textwidth]{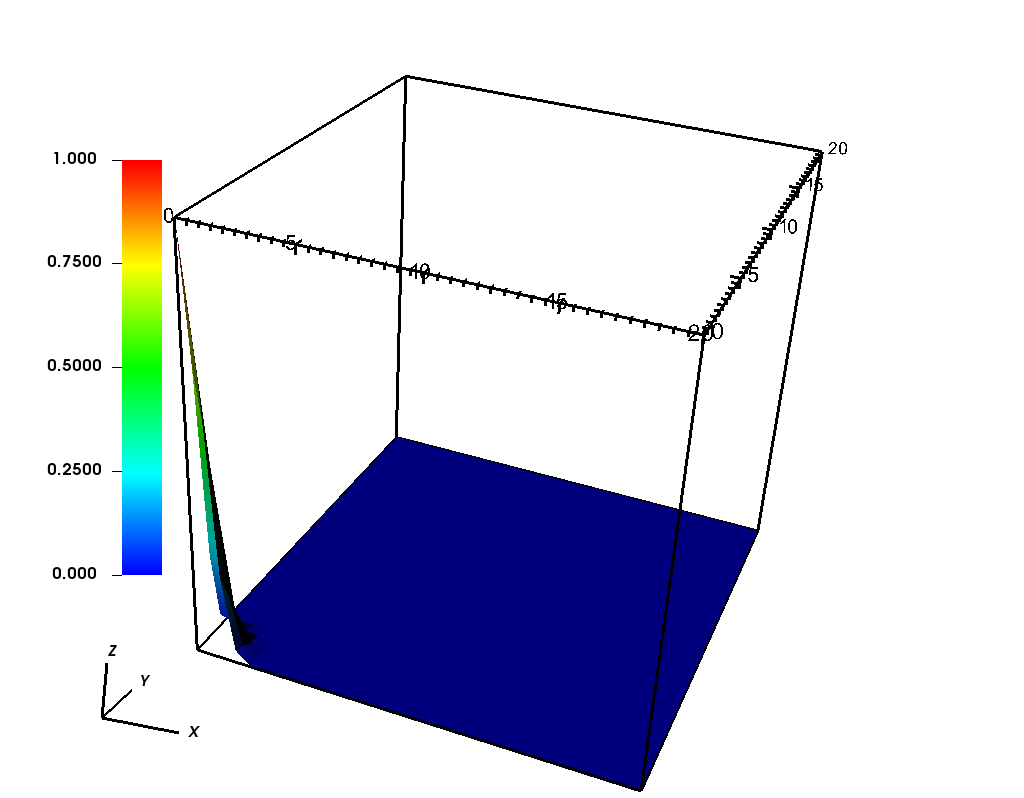}
		\caption{t = 0}
	\end{subfigure}
	\begin{subfigure}{0.23\textwidth}
		\includegraphics[width=\textwidth]{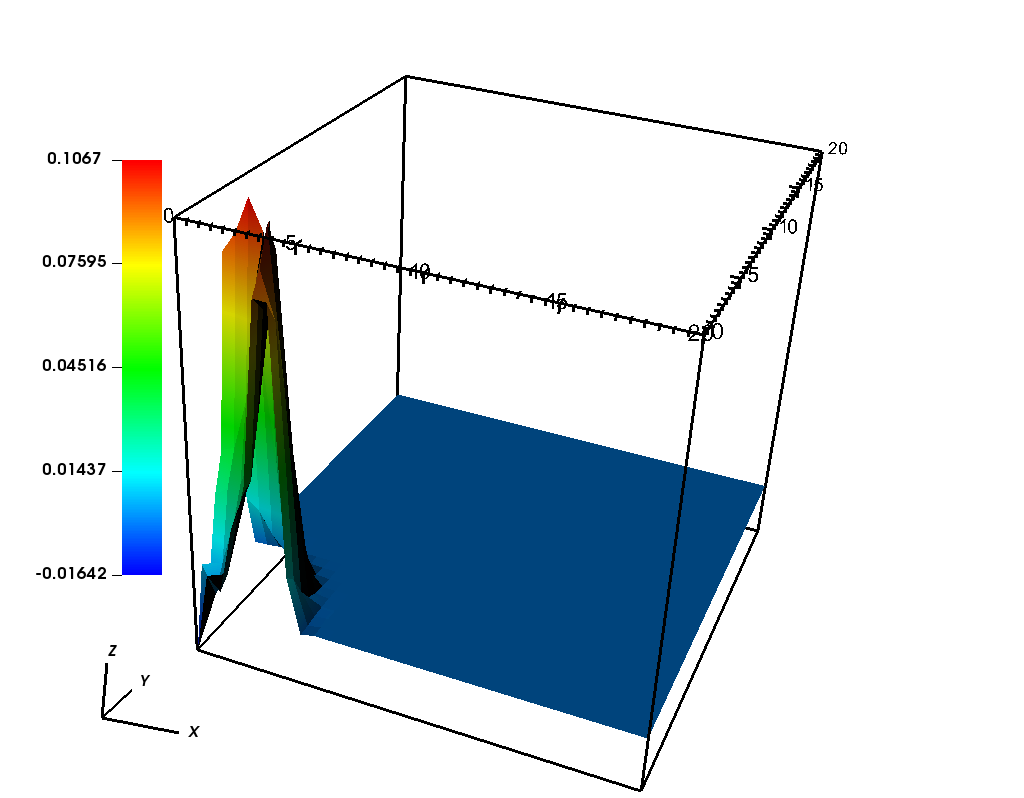}
		\caption{t = 5}
	\end{subfigure}
	\caption{The snapshots of cancer cell invasion $u$ for $\chi = 1.25$.}
	\label{fig12}
\end{figure}

\subsection{Identical proliferation and haptotactic coefficients}
In this subsection, we consider the case when the 
proliferation rate is equal to haptotactic coefficient, i.e., 
$\mu=\chi=1$, and all other parameters are the same as in the previous subsections. 
As it can be seen from Figs.~\ref{fig13} and \ref{fig14}, due to the proliferation rate, 
the concentration of cancer growths quickly even from the beginning resulting
from a
high amount of haptotaxis, therefore the tumour migrates rapidly inside the 
domain and degrades the connective tissue in a much shorter amount of time.

\begin{figure}[H]
	\centering
	\begin{subfigure}{0.24\textwidth}
		\includegraphics[width=\textwidth]{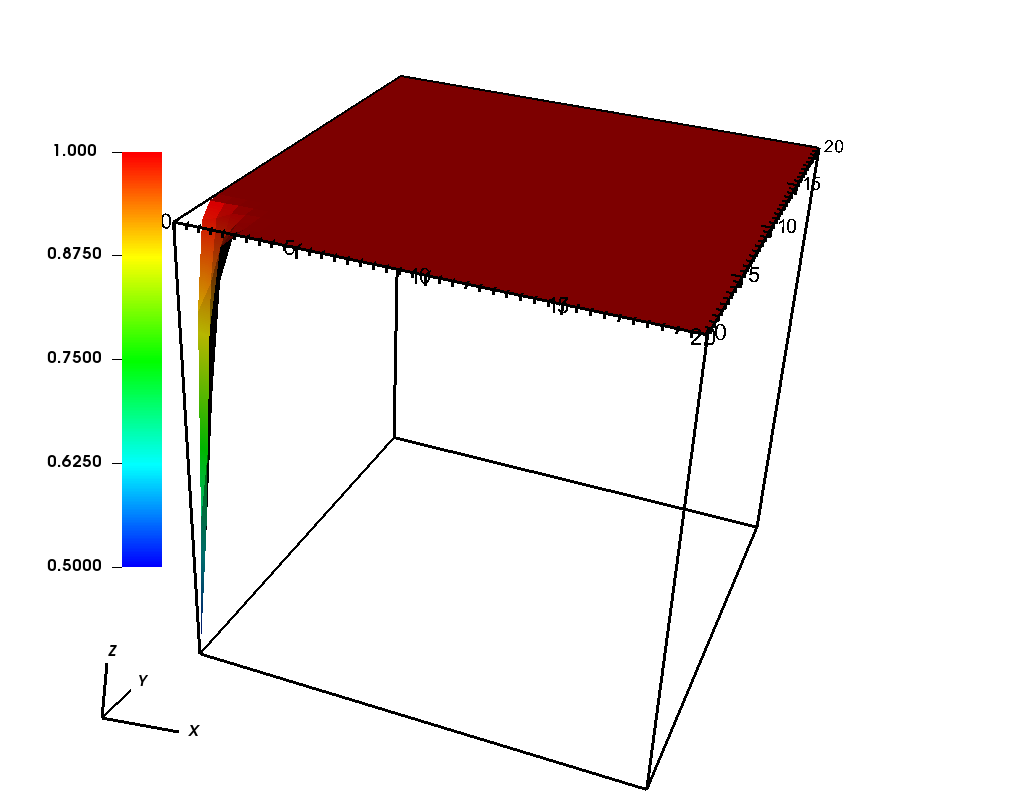}
		\caption{t = 0}
	\end{subfigure}
	\begin{subfigure}{0.24\textwidth}
		\includegraphics[width=\textwidth]{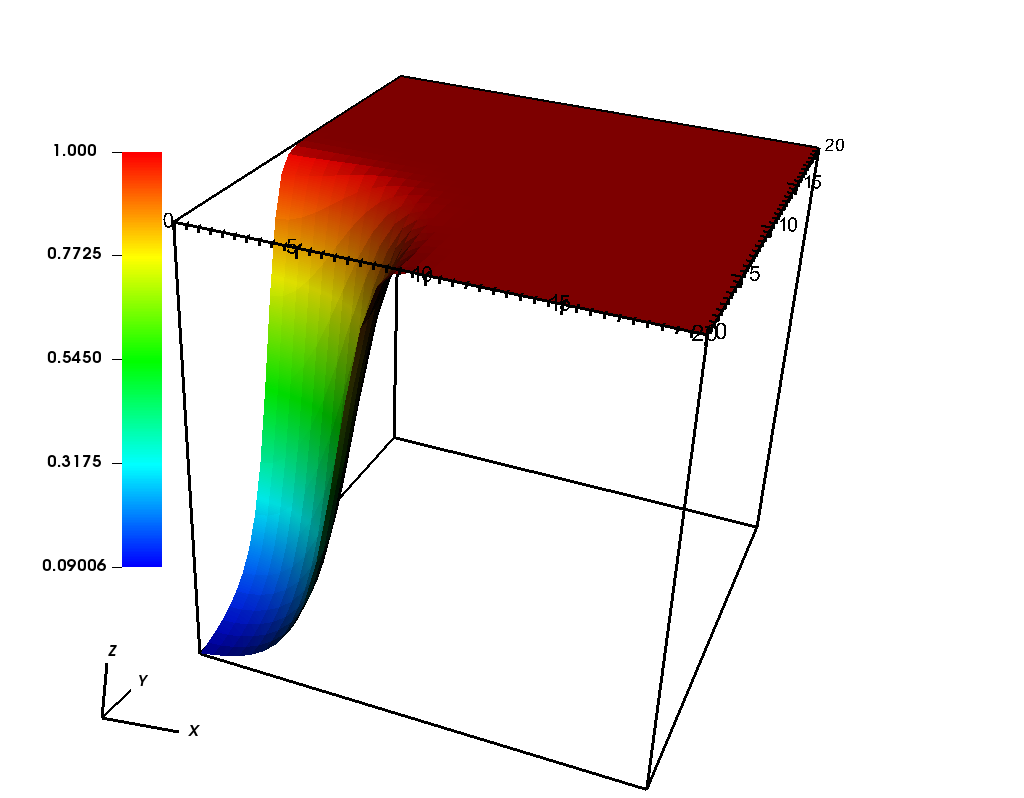}
		\caption{t = 10}
	\end{subfigure}
	\begin{subfigure}{0.24\textwidth}
		\includegraphics[width=\textwidth]{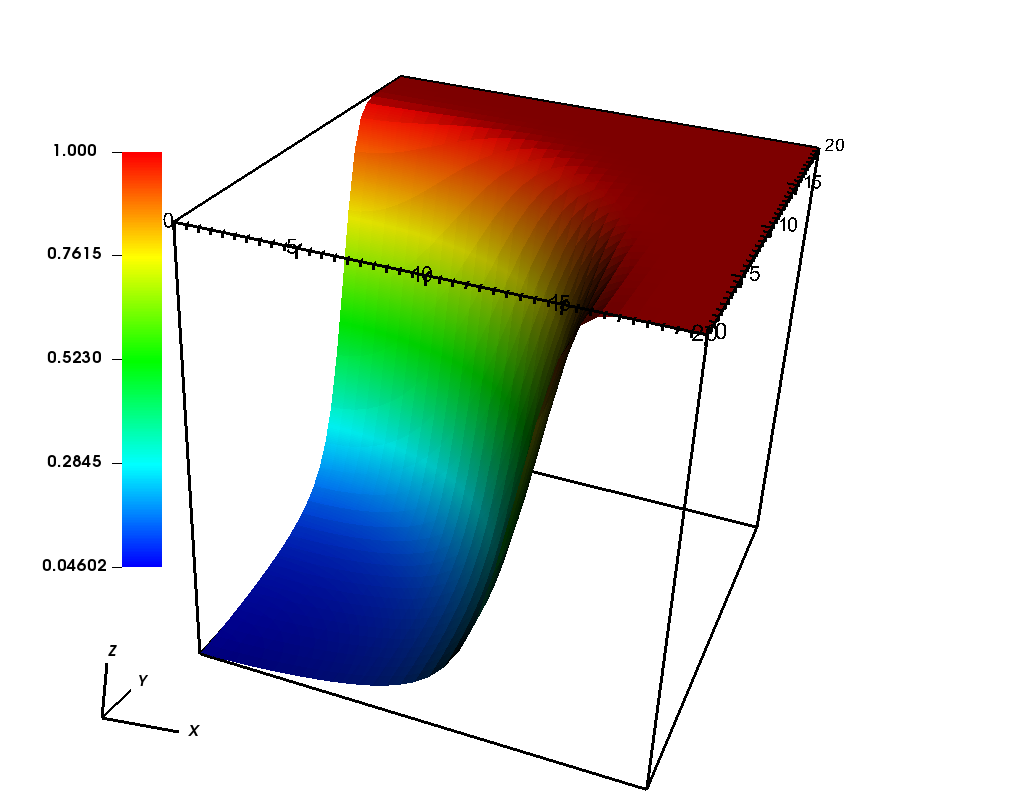}
		\caption{t = 20}
	\end{subfigure}
	\begin{subfigure}{0.24\textwidth}
		\includegraphics[width=\textwidth]{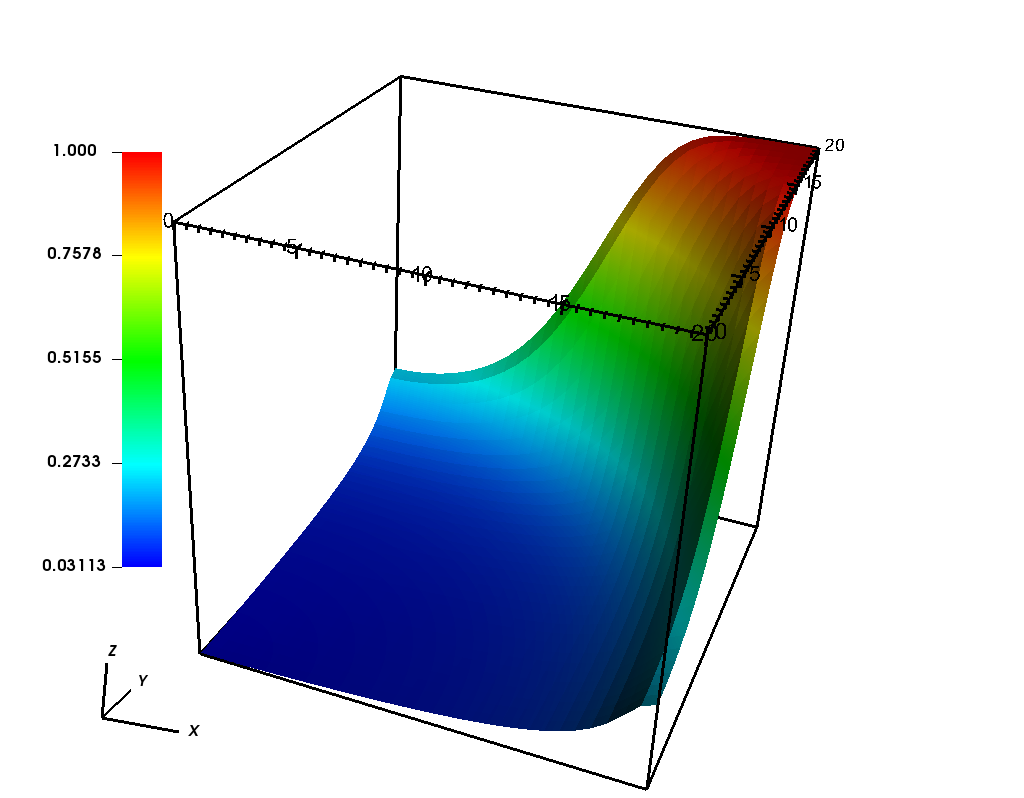}
		\caption{t = 30}
	\end{subfigure}
	\caption{
	Degradation of connective tissue $c$ for $\chi =1.0, \mu =1.0$ at different time instances $t=0, 10, 20$ and 30.
}
	\label{fig13}
\end{figure} 

\begin{figure}[H]
	\centering
	\begin{subfigure}{0.24\textwidth}
		\includegraphics[width=\textwidth]{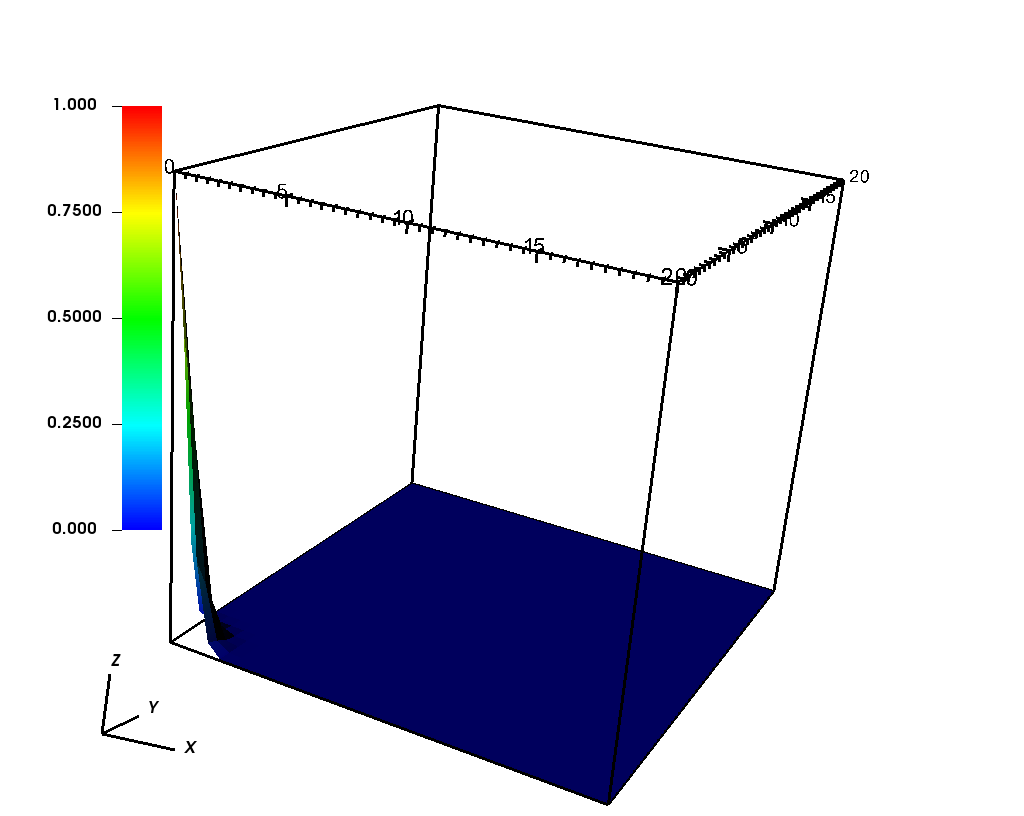}
		\caption{t = 0}
	\end{subfigure}
	\begin{subfigure}{0.24\textwidth}
		\includegraphics[width=\textwidth]{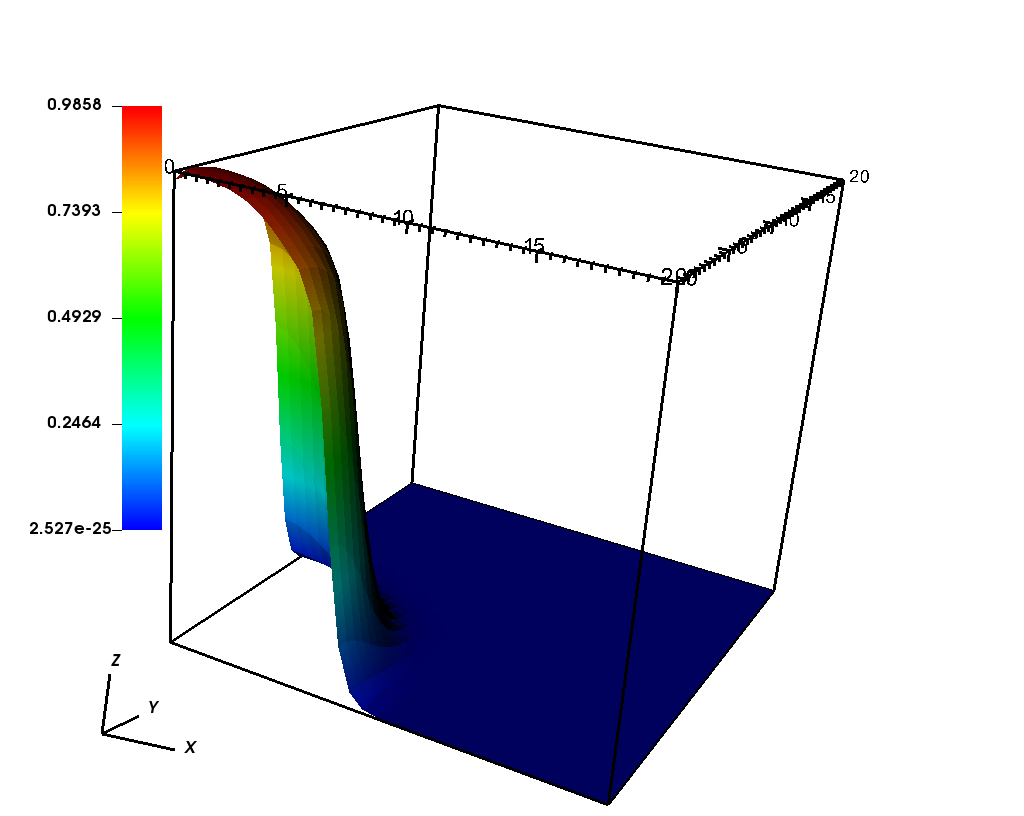}
		\caption{t = 10}
	\end{subfigure}
	\begin{subfigure}{0.24\textwidth}
		\includegraphics[width=\textwidth]{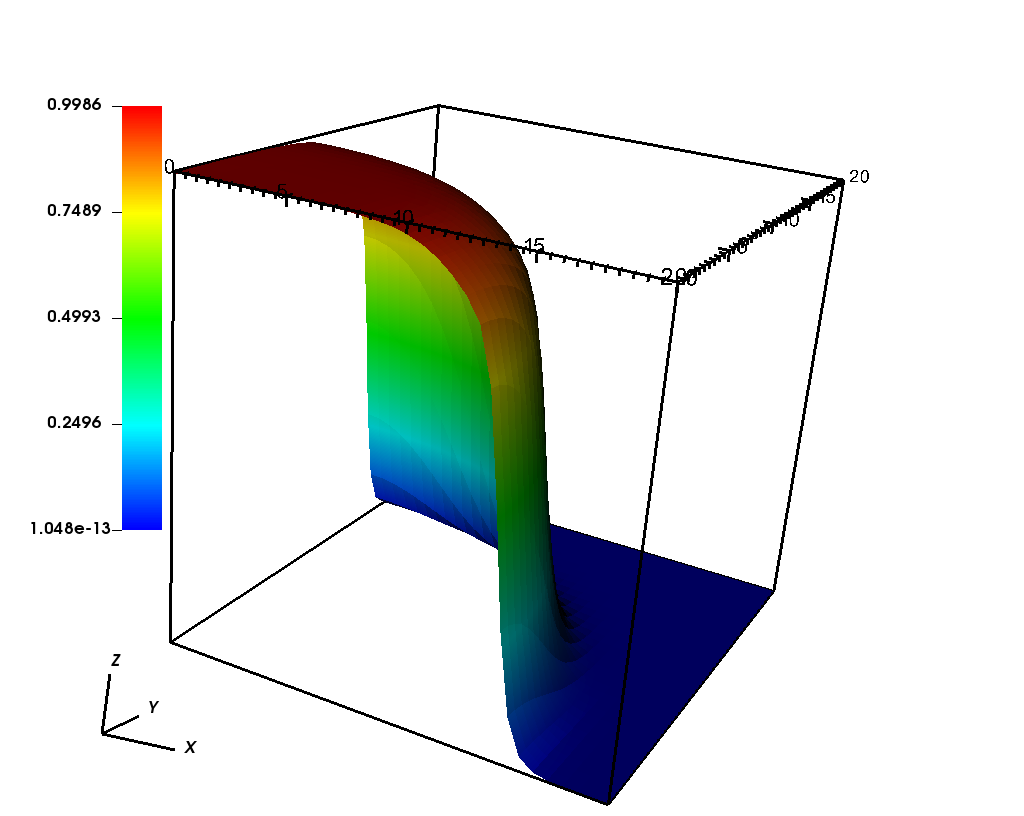}
		\caption{t = 20}
	\end{subfigure}
	\begin{subfigure}{0.24\textwidth}
		\includegraphics[width=\textwidth]{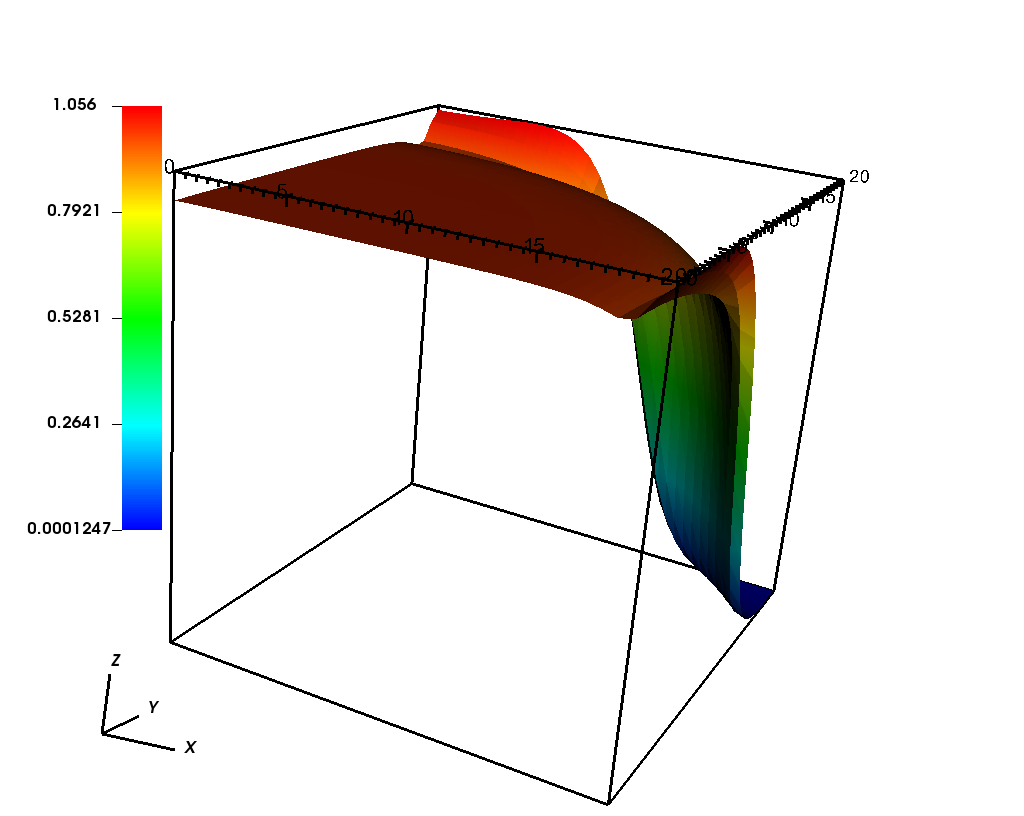}
		\caption{t = 30}
	\end{subfigure}
	\caption{
	Invasion of cancer cells $u$ for $\chi =1.0, \mu =1.0$ at different time instances $t=0, 10, 20$ and 30.
}
	\label{fig14}
\end{figure}

\subsection{Three dimensional simulations}\label{sec:3d}
In this final subsection, we perform numerical simulations in three spatial dimensions 
to consider some more realistic movement. Here, the 
experiments are performed on a mesh with $32\, 768$ hexahedral elements covering 
the domain $\Omega$. Figs.~\ref{fig15} and \ref{fig16} show the snapshots
of cancer cells and connective tissues for growth rate $\mu =1$ and haptotactic
coefficient $\chi = 1$. Further, we use the parameters $\alpha^{-1} = 0.1$ and
$\varepsilon = 0.2$. As it can be seen, at $t=5$ the 
connective tissue covers the entire domain and only a small amount of cancer
cells exists at the corner, by the time cancer cells growth and invade the domain of connective tissue quickly and by $t=35$ 
almost all the domain is occupied by cancer cells.

\begin{figure}[H]
	\centering
	\begin{subfigure}{0.24\textwidth}
		\includegraphics[width=\textwidth]{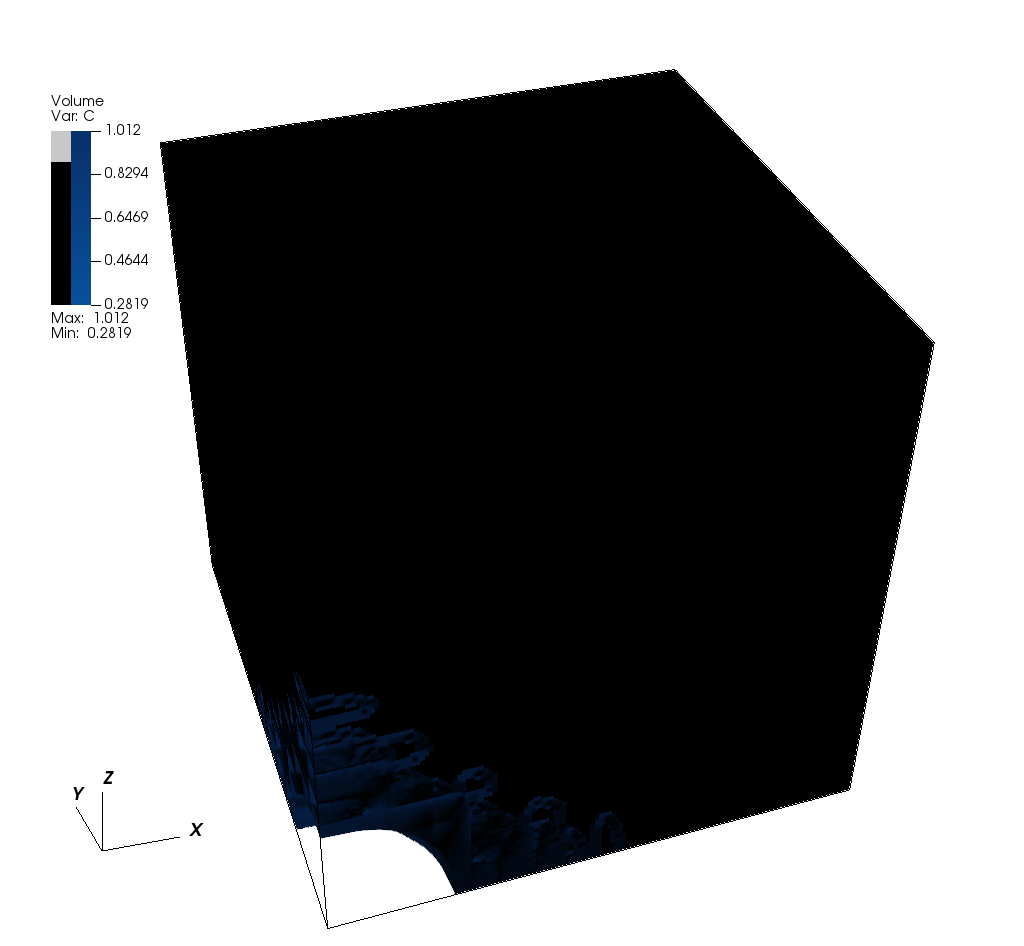}
		\caption{t = 5}
	\end{subfigure}
	\begin{subfigure}{0.24\textwidth}
		\includegraphics[width=\textwidth]{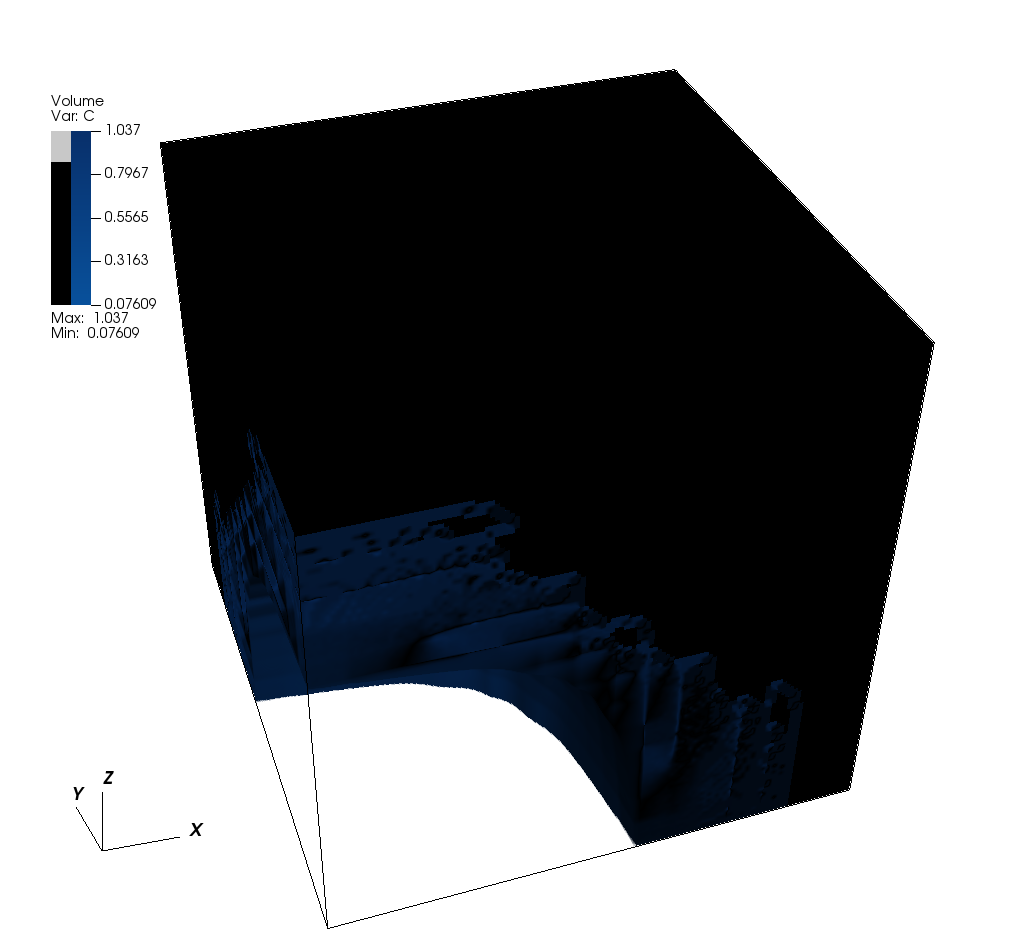}
		\caption{t = 15}
	\end{subfigure}
	\begin{subfigure}{0.24\textwidth}
		\includegraphics[width=\textwidth]{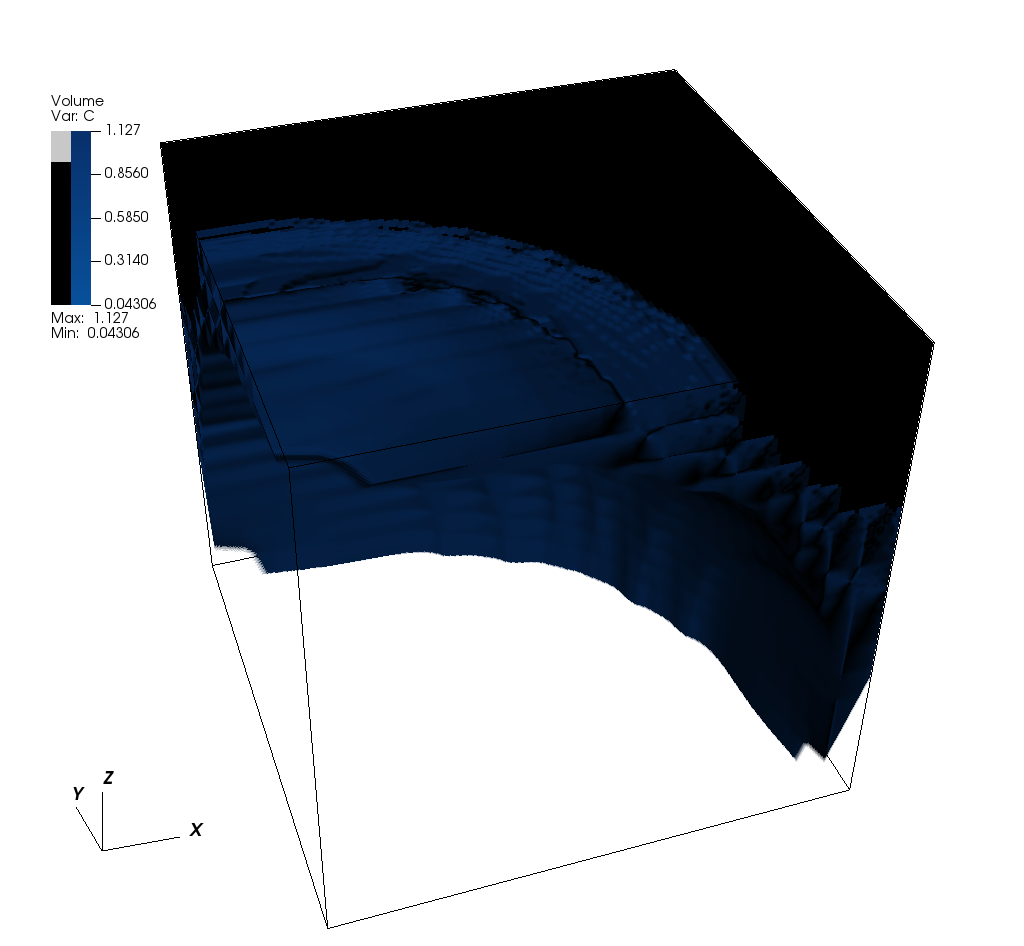}
		\caption{t = 25}
	\end{subfigure}
	\begin{subfigure}{0.24\textwidth}
		\includegraphics[width=\textwidth]{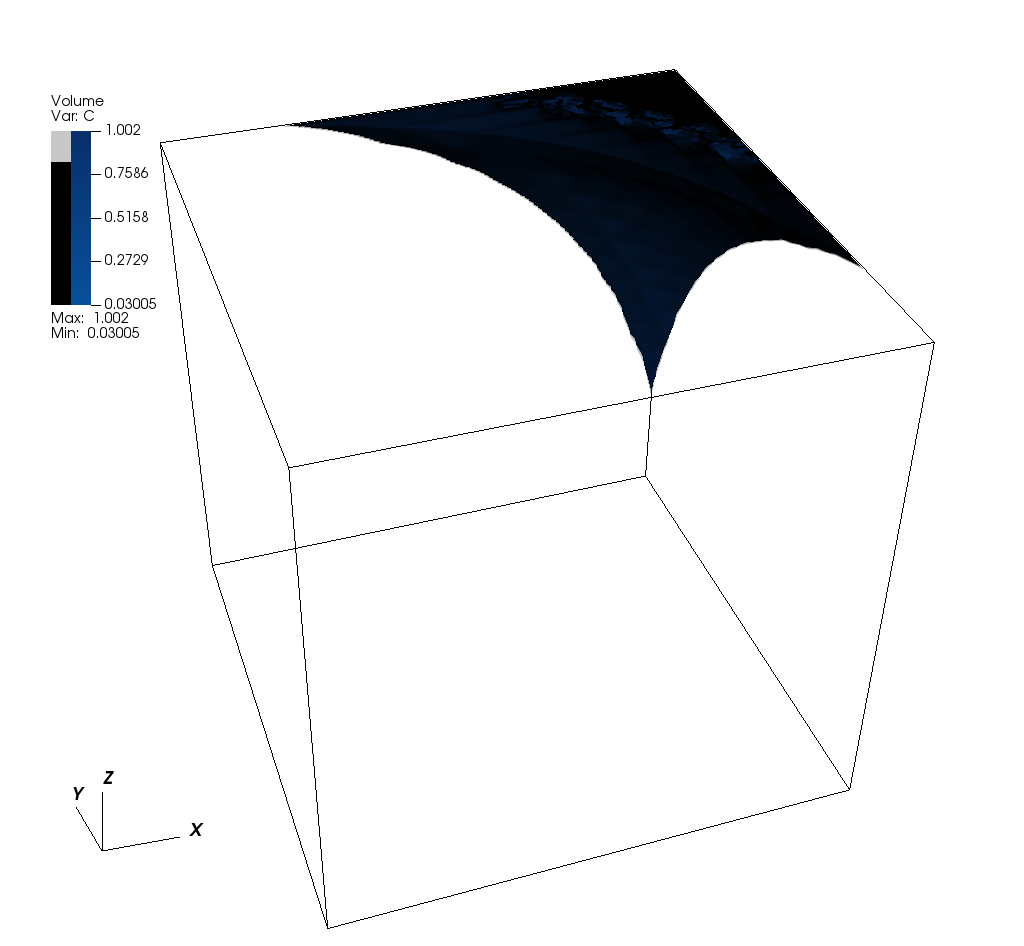}
		\caption{t = 35}
	\end{subfigure}
	\caption{
	Degradation of connective tissue $c$ for $\chi =1.0, \mu =1.0$ at different time instances $t=5, 15, 25$ and 35.
}
	\label{fig15}
\end{figure} 

\begin{figure}[H]
	\centering
	\begin{subfigure}{0.24\textwidth}
		\includegraphics[width=\textwidth]{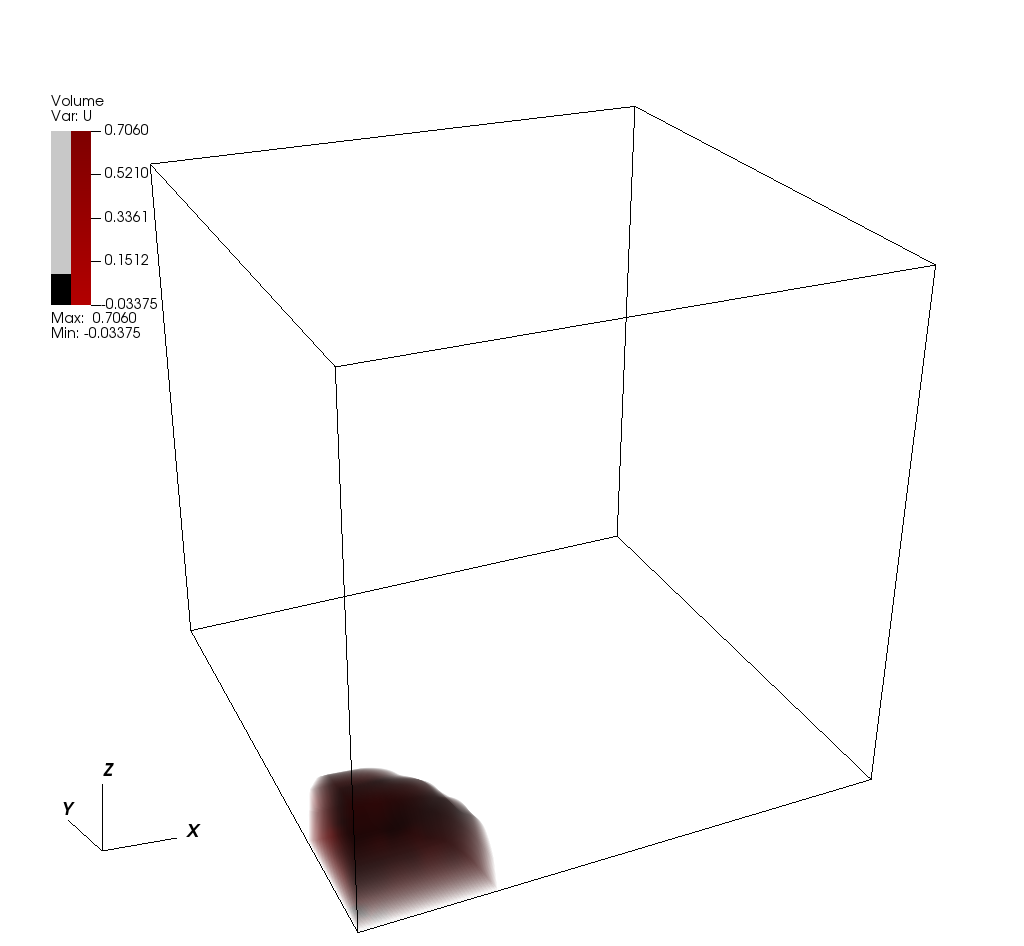}
		\caption{t = 5}
	\end{subfigure}
	\begin{subfigure}{0.24\textwidth}
		\includegraphics[width=\textwidth]{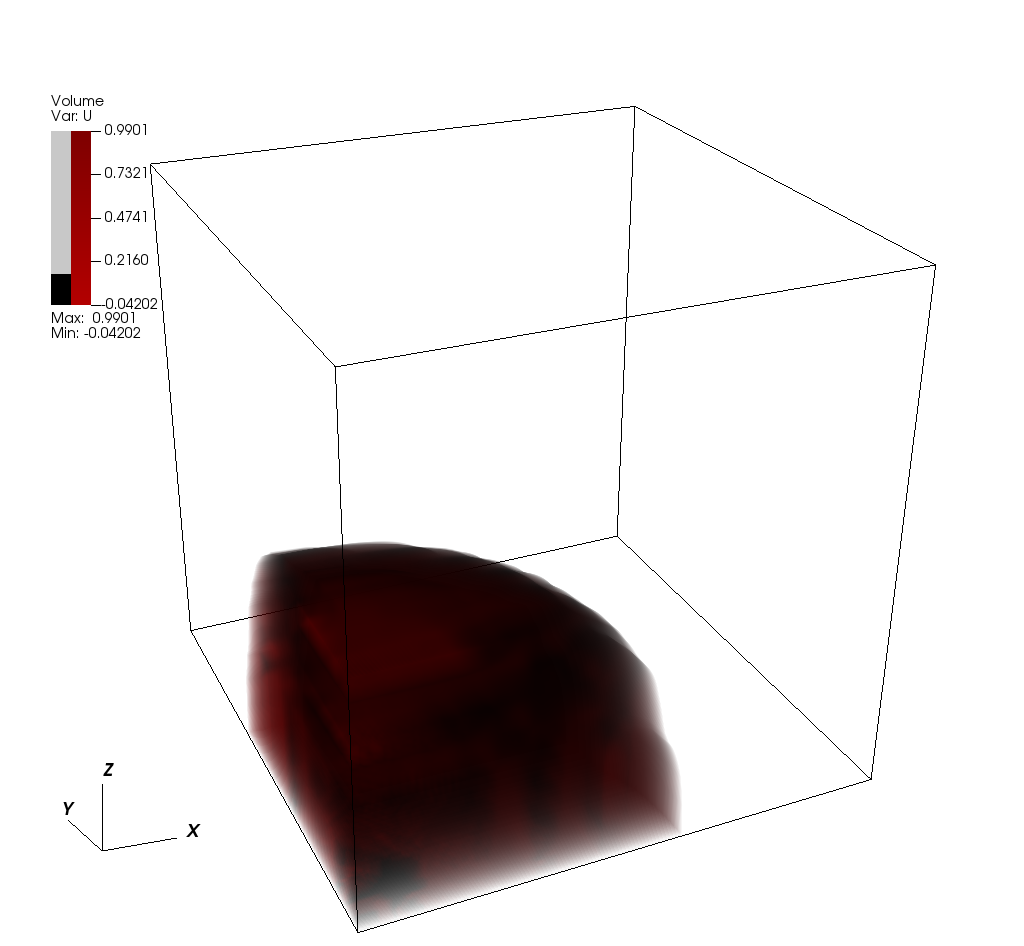}
		\caption{t = 15}
	\end{subfigure}
	\begin{subfigure}{0.24\textwidth}
		\includegraphics[width=\textwidth]{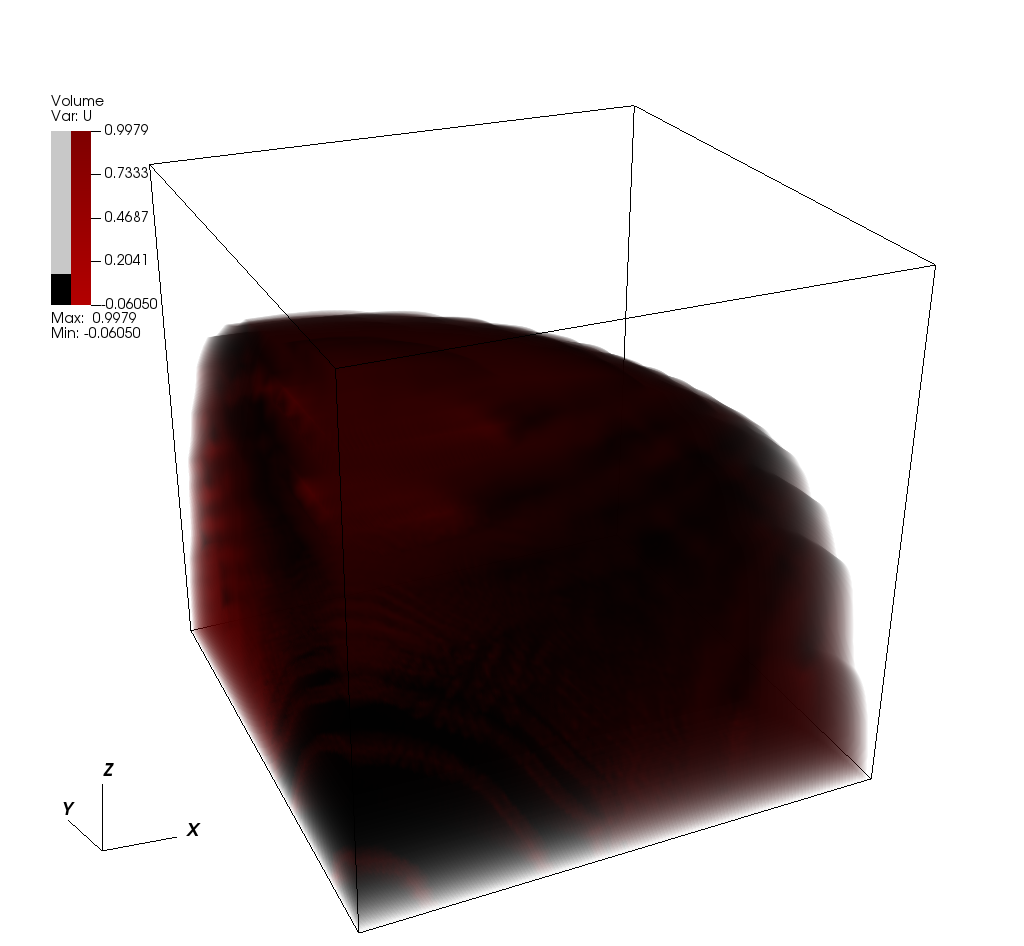}
		\caption{t = 25}
	\end{subfigure}
	\begin{subfigure}{0.24\textwidth}
		\includegraphics[width=\textwidth]{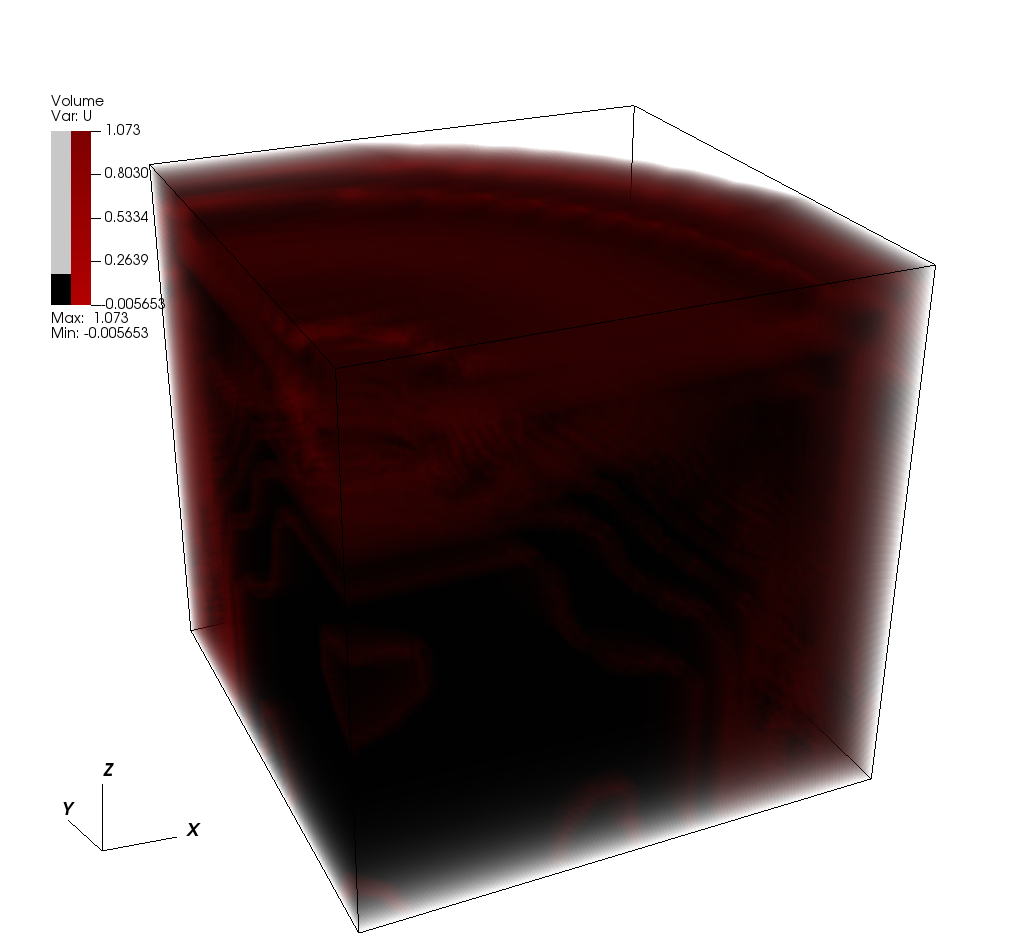}
		\caption{t = 35}
	\end{subfigure}
	\caption{
	Invasion of cancer cells $u$ for $\chi =1.0, \mu =1.0$ at different time instances $t=5, 15, 25$ and 35.
}
	\label{fig16}
\end{figure}

\section{Conclusions}
\label{sec_conclusions}
In this paper, we established theoretical proofs, numerical algorithms, 
implementations and numerical simulations for a cancer invasion model.
In our theoretical part, existence of global classical 
solutions in both two- and three-dimensional bounded domains were established.
In the proofs, we employed the fact that the second and third equation in \eqref{eq:system} at least regularize in time. For showing boundedness in $L^\infty$,
the comparison principle allowed us to conclude boundedness in small time intervals,
which then was iteratively applied to obtain the result also for larger times.
For the spatial derivatives, we secondly applied a testing procedure for deriving
estimates valid on small time intervals, again followed by an iteration procedure.
Parabolic regularity theory yielded global existence of the solutions.

The numerical stability of the system heavily depends on the haptotactic coefficient $\chi$. By fixing proliferation rate $\mu$ and varying the $\chi$ one can make either the diffusion or transport of the cells dominant. The later usually gives rise to spurious oscillations or numerical blow up in the system. 
In order to study such properties, \eqref{eq:system} was discretized using 
finite differences in time and Galerkin finite elements in space. A fixed-point 
scheme was designed to decouple the three equations, yielding a robust 
nonlinear procedure. These developments and their implementation
allowed us to study numerically variations 
in $\mu$ and $\chi$ in two and three spatial dimensions and to 
illustrate our theoretical results. 

As to future work, we notice that higher parameter variations resulting into 
convection-dominated regimes, require the design and implementation of
stabilization methods such as streamline upwind Petrov--Galerkin stabilizing
formulations or algebraic flux corrected transport.

\section*{Acknowledgments}
The work of Shahin Heydari has been supported through the grant No.396921 of the Charles University Grant Agency and Charles University Mobility Fund No. 2-068. She also would like to 
thank Institute of Applied Mathematics and the Leibniz University Hannover 
for their hospitality during the six months stay from November 2021 to April 2022.


\end{document}